\documentclass[11pt,a4paper,reqno]{amsart}

\usepackage[utf8]{inputenc}
\usepackage{hyperref}
\usepackage[capitalise]{cleveref}
\usepackage{amsmath}
\usepackage{caption}
\usepackage[dvipsnames]{xcolor}
\usepackage{graphicx}
\usepackage{amssymb}
\usepackage{amsthm}
\usepackage{pgfplots}
\usepackage{enumerate}
\usepackage{wasysym}
\usepackage{subcaption}
\usepackage{extarrows}
\usepackage{polynom}
\usepackage{dsfont}
\usepackage{verbatim}
\usepackage[shortlabels]{enumitem}
\usepackage[page,titletoc,title]{appendix}
\usepackage[english]{babel}
\usepackage{todonotes}
\usepackage{pdfpages}
\usetikzlibrary{math}
\usetikzlibrary{decorations.pathreplacing}
\usetikzlibrary{hobby}    
\usetikzlibrary{shapes.misc}
\usepackage{float}

\calclayout
\DeclareRobustCommand{\SkipTocEntry}[5]{}
\theoremstyle{definition}
\newtheorem{theorem}{Theorem}[section]
\newtheorem{prop}[theorem]{Proposition}
\newtheorem{definition}[theorem]{Definition}

\newtheorem{remark}[theorem]{Remark}
\newtheorem{lemma}[theorem]{Lemma}

\numberwithin{equation}{section}
\newcommand{\abs}[1]{\left\lvert#1\right\rvert}

\newcommand{\norm}[1]{\left\|#1\right\|}

\newcommand{\R}{\mathbb R}
\newcommand{\N}{\mathbb N}
\newcommand{\Z}{\mathbb Z}

\renewcommand{\epsilon}{\varepsilon}

\renewcommand{\O}{\mathcal O}
\newcommand{\A}{\mathcal A}
\newcommand{\dx}{\, \mathrm d}
\newcommand{\id}{\mathrm{Id}}

\newcommand{\E}{\mathcal E}

\newcommand{\W}{\mathcal W}
\newcommand{\fra}{\mathfrak a}
\newcommand{\B}{\mathcal B}
\newcommand{\br}{\mathfrak b}

\newcommand{\D}{\mathcal D}
\newcommand{\cW}{\mathcal W}

\newcommand{\homdim}{\mathrm Q}
\renewcommand{\phi}{\varphi}

\DeclareMathOperator{\supp}{supp}
\DeclareMathOperator{\kin}{kin}

\newcommand{\ssubset}{\subset\!\subset}

\renewcommand{\L}{\operatorname{L}} 
\newcommand{\C}{\operatorname{C}} 
\renewcommand{\H}{\operatorname{H}} 
\DeclareRobustCommand{\Hdot}{\dot{\H}\protect{\vphantom{H}}}
\DeclareRobustCommand{\Cdot}{\dot{\C}\protect{\vphantom{C}}} 

\definecolor{grey1}{HTML}{DCDCDC}
\definecolor{grey2}{HTML}{C0C0C0}
\definecolor{grey3}{HTML}{808080}

\begin{document}
\allowdisplaybreaks
\title{Critical trajectories in kinetic geometry}

\date{\today}

\author{Helge Dietert}
\address[Helge Dietert]{Universit\'e Paris Cit\'e and Sorbonne
  Universit\'e, CNRS, IMJ-PRG, F-75006 Paris, France}
\email{helge.dietert@imj-prg.fr}

\author{Cl\'ement Mouhot}
\address[Cl\'ement Mouhot]{University of Cambridge, Wilberforce
  Road, Cambridge CB3 0WA, United Kingdom}
\email{c.mouhot@dpmms.cam.ac.uk}

\author{Lukas Niebel}
\address[Lukas Niebel]{Institut f\"ur Analysis und Numerik,  Westf\"alische Wilhelms-Universit\"at M\"unster\\
Orl\'eans-Ring 10, 48149 M\"unster, Germany.}
\email{lukas.niebel@uni-muenster.de }

\author{Rico Zacher}
\address[Rico Zacher]{Institut f\"ur Angewandte Analysis, Universit\"at Ulm, Helmholtzstra\ss{}e 18, 89081 Ulm, Germany.}
\email{rico.zacher@uni-ulm.de}

\begin{abstract}
  We construct \emph{critical trajectories in kinetic geometry}, i.e.\ curves in $\R^{1+2n}$ that are: tangential to the vector fields $\partial_t+v\cdot \nabla_x$ and $\nabla_v$, connecting any two given points, respecting the underlying kinetic scaling, and with the property, that the singularity of the $v$-tangent vector near the starting point equates the degeneracy of the dependency of the curve velocity in terms of the endpoint velocity. 
   
   The construction is based on Newton's laws of motion, where the ansatz for the forcing of the kinetic trajectory is the superposition of functions combining the correct power scaling with desynchronised logarithmic oscillations. 

  These critical trajectories provide a robust and versatile ``almost exponential map'' that allows to prove several functional analytic estimates. We introduce a notion of \emph{kinetic mollification} and, as an application, deduce the kinetic Sobolev inequality with optimal exponent without relying on the fundamental solution. 

  Moreover, we establish a universal estimate for the logarithm of positive supersolutions to the Kolmogorov equation with rough coefficients inspired by the work of Moser (1961, 1964) on elliptic and parabolic problems. Combining this estimate with De~Giorgi-Moser iterations and a lemma due to Bombieri and Giusti, we give an alternative proof of the (weak) Harnack inequality for the Kolmogorov equation with rough coefficients, following the ideas of Moser (1971). Our result gives the optimal range of exponents in the weak Harnack inequality and the optimal (geometric) dependency of the Harnack constant on the bounds of the diffusion matrix.
\end{abstract}
\maketitle

\tableofcontents

\pagebreak

\section{Introduction}

Understanding --- and obtaining quantitative regularity estimates for --- the underlying hypoelliptic structure of diffusive kinetic partial differential equations has been a subject of mathematical interest, dating back at least to the seminal work of Kolmogorov~\cite{kolmogoroff_zufallige_1934}, where he derived the explicit fundamental solution to
\begin{equation} \label{eq:int:kol}
  (\partial_t  + v \cdot \nabla_x) f - \nabla_v \cdot ( \mathfrak a \nabla_v f) = S,
\end{equation}
with constant diffusion coefficient matrix $\fra = \id$, where $f = f(t,x,v) \colon \R^{1+2n} \to \R$ is the solution and $S \colon  \R^{1+2n} \to \R$ is a source term. 
A major conceptual advance came with H\"ormander's work in 1967~\cite{hormander_hypoelliptic_1967}, where he introduced a vector field formalism that allows one to observe the underlying regularisation mechanism without recourse to an explicit fundamental solution. The Kolmogorov equation with constant diffusion coefficient serves as a prototypical example of operators made up of a second-order sum-of-squares operator and a first-order drift operator studied by H\"ormander (``hypoelliptic operators of type B'' in his terminology).

With the notation $X_0 := \partial_t + v \cdot \nabla_x$ and $X_i := \partial_{v_i}$, we may rewrite \eqref{eq:int:kol} with $\fra = \id$ as
\begin{equation} \label{eq:hormander}
	X_0 f - \sum_{i = 1}^n X_i^2 f = S.
\end{equation}
H\"ormander's theorem asserts that if the vector fields $X_0,\dots,X_n$ and their commutators up to some finite order span the tangent space at every point, then the operator on the left-hand side of \eqref{eq:hormander} is hypoelliptic, i.e.\ $S \in \C^\infty$ implies $f \in \C^\infty$. 

In the kinetic setting, this condition can be directly verified via the computation
\begin{equation*}
	[\partial_{v_i},\partial_t+v \cdot \nabla_x] = \partial_{x_i},
\end{equation*}
and observing that $\partial_t+v \cdot \nabla_x,\partial_{v_1},\dots,\partial_{v_n},\partial_{x_1},\dots, \partial_{x_n}$ span the full tangent space, i.e.\ $\R^{1+2n}$, at every point.

The regularity theory of such vector fields and their associated operators has been extensively developed since then. Even prior to H\"ormander's work, it had been observed that although the vector fields themselves may be degenerate, their commutators can be used to recover the missing directions. Notable milestones in this broader geometric framework include the Chow--Rashevski\u{\i} theorem~\cite{chow_uber_1939,rs_connect_1938} from sub-Riemannian geometry and the related work of Carath\'eodory~\cite{MR1511534}. Their results state that, under
the bracket-generating condition, any two points can be connected by a curve whose
tangent vector is in the span of these vector fields. This geometric property reflects a deep connection between connectivity and regularity. We highlight~\cite{lanconelli_poincare_2000}, where this connection is made quantitative in the drift-free setting (type A). 

\medskip

In the present work, we develop a framework that links the kinetic geometry induced by the vector fields $\partial_t+v\cdot \nabla_x$ and $\nabla_v$ to sharp functionalanalytic estimates, through the construction of \emph{critical kinetic trajectories}. We begin by presenting our main contributions and then we review the related literature.

\subsection{Critical kinetic trajectories}

The vector field $\partial_t + v \cdot \nabla_x$ and the derivatives in velocity $\nabla_v$ dictate the well-known non-commutative structure of the Kolmogorov equation \eqref{eq:int:kol} on $\R^{1+2n}$. Given $ (t_1,x_1,v_1),(t_2,x_2,v_2) \in \R^{1+2n}$ we write
\begin{equation*}
	(t_1,x_1,v_1) \circ (t_2,x_2,v_2) = (t_1+t_2,x_1+x_2+t_2v_1,v_1+v_2).
\end{equation*}
This is a non-commutative group structure on $\R^{1+2n}$. Given $r>0$, we define the scaling
\begin{equation*}
  [\delta_r f](t,x,v) = f(r^2t,r^3x,rv).
\end{equation*}
The homogeneous space $(\R^{1+2n},\circ,\delta_r)$ has a $(1+2n)$-dimensional tangent space at any point, which we identify with  $\R^{1+2n}$, and the \emph{homogeneous dimension}  --- refering to how the volume is modified with respect to the given one-parameter scaling --- is $2+4n$. 

We are concerned with the construction of a ``kinetic almost exponential map'' on $\R^{1+2n}$ (borrowing the terminology from \cite{lanconelli_poincare_2000}) which respects the underlying kinetic geometry. That is a mapping
\begin{equation*}
	\gamma^{\bold m}(r) = \gamma^{\bold m}(r;(t,x,v)) \colon [0,\infty) \times \R^{1+2n} \times \R^{1+2n} \to \R^{1+2n},
\end{equation*}
where $(t,x,v) \in \R^{1+2n}$ is the starting point of the curve, $\bold m = ( m_0,  m_1,  m_2) \in \R^{1+2n}$ can be thought of as an element of a generalised tangent space and $r \in [0,\infty)$ is the artificial time. We denote by $\gamma^{\bold m} = (\gamma^{\bold m}_t,\gamma^{\bold m}_x,\gamma^{\bold m}_v)$ its components, and by $\dot \gamma^{\bold m}$ the derivative with respect to the parameter $r$. We choose the normalisation such that the time component is linear in $r$. 

We impose the following properties in order to be a kinetic almost exponential map:
\begin{enumerate}
\item[\hypertarget{link:C1}{\textbf{(C1)}}] given $(t,x,v) \in \R^{1+2n}$ and $\bold m \in \R^{1+2n}$, $\dot \gamma^{\bold m}(r;(t,x,v))$ is in the span of the vector fields $\partial_t+v\cdot \nabla_x$ and $\nabla_v$ for every $r >0$,
\item[\hypertarget{link:C2}{\textbf{(C2)}}] given $(t,x,v) \in \R^{1+2n}$ and some measurable subset $\bold M \subset \R^{1+2n}$, the volume 
  \begin{equation*}
    \abs{\left\{ \gamma^{\bold m}(r;(t,x,v)) \ : \ \bold m \in \bold M \right\}} \sim r^{\frac{2+4n}{2}} \abs{\bold M} = r^{1+2n} \abs{\bold M},
  \end{equation*}
  i.e.\ scales homogeneously in $r$ with order equal to half the homogeneous dimension, 
  \item[\hypertarget{link:C3}{\textbf{(C3)}}] given $(t,x,v) \in \R^{1+2n}$, $r>0$, $m_0 \neq 0$ and writing $(\tilde{m}_1,\tilde{m}_2) = \Phi_{r,(t,x,v)}(m_1,m_2) = \gamma^{\bold m}_{x,v}(r;(t,x,v)) $ the asymptotics of the left-hand and right-hand sides of
  \begin{equation*}
    \abs{\left[\nabla_{\tilde{m}_2} \Phi^{-1}_{r,(t,x,v)}\right]\left(\gamma^{\bold m}(r;(t,x,v))\right)} \sim 	\abs{\dot{\gamma}^{\bold m}_v(r;(t,x,v))}
  \end{equation*}
  match as $r \to 0^+$, and are locally integrable functions in $r$.
\end{enumerate}
Here, we have denoted by $\gamma_{x,v}^{\bold m} = (\gamma_x^{\bold m},\gamma_v ^{\bold m})$ the pair of spatial and velocity components of $\gamma^{\bold m}$.

It is the combination of the properties \hyperlink{link:C1}{\textbf{(C1)}}, \hyperlink{link:C2}{\textbf{(C2)}} and \hyperlink{link:C3}{\textbf{(C3)}} that allows us to prove sharp functional analytic estimates. Because of the degeneracy of the diffusive part of the equation in $x$, we would like to avoid the use of the vector field $\nabla_x$ in \hyperlink{link:C1}{\textbf{(C1)}}. The second property \hyperlink{link:C2}{\textbf{(C2)}} states that our trajectories are not degenerate as functions of the generalised tangent space variable $\bold m$, with a quantitative estimate on the measure of the set of points attained in terms of the measure of the set of variables $\bold m$ and the correct scaling in $r$. The last property \hyperlink{link:C3}{\textbf{(C3)}} states that the singularity of the $v$-tangent vector near the starting point $r \sim 0$ equates the degeneracy of the dependency of the curve velocity in terms of the endpoint velocity. For suitable maps, the inverse function theorem implies that 
\begin{equation*}
	\left[\nabla_{\tilde{m}_2} \Phi^{-1}_{r,(t,x,v)}\right]\left(\gamma^{\bold m}(r;(t,x,v))\right) = \left(\big(\nabla_{m_1,m_2} \gamma^{\bold m}_{x,v}(r;(t,x,v))\big)^{-1}\right)_{\cdot;2},
\end{equation*}
where the subscript $\cdot;2$ refers to the $2n \times n$ submatrix consisting of the last $n$ columns.

Note in particular that the modulus of the tangent vector $\abs{\dot{\gamma}^{\bold m}}$ of our
kinetic almost exponential map is not constant, as for the geodesics in the standard notion of exponential
map. It is easy to see that such a property would not be compatible with the natural
scaling.

We say $\gamma^{\bold m}$ is controllable (in the terminology of \cite{lanconelli_poincare_2000}) if for any $(t_0,x_0,v_0),(t_1,x_1,v_1) \in \R^{1+2n}$ with $t_0 \neq t_1$ there exists a parameter $\bold m \in \R^{1+2n}$ in the generalised tangent space such that $\gamma^{\bold m}(1;(t_0,x_0,v_0)) = (t_1,x_1,v_1)$.

\bigskip

Let us give a heuristic explanation for our choice of trajectories. A first naive suggestion for the kinetic almost exponential map is the Euclidean exponential map
\begin{equation*}
  \gamma^{\bold m}(r;(t,x,v)) := (t+r m_0,x+r m_1, v+rm_2).
\end{equation*}
For any smooth function $f \colon \R^{1+2n} \to \R$ we have 
\begin{equation*}
  \frac{\dx}{\dx r}f(\gamma^{\bold m}(r)) = m_0 [\partial_t f](\gamma^{\bold m}(r)) +m_1 \cdot[\nabla_x f](\gamma^{\bold m}(r)) + m_2 \cdot [\nabla_v f](\gamma^{\bold m}(r))
\end{equation*}
along this curve, so this choice does not satisfy \hyperlink{link:C1}{\textbf{(C1)}}. Nevertheless, this mapping is controllable by choosing $m_0=t_1-t_0 $, $m_1=x_1-x_0 $ and $m_2=v_1-v_0 $ for any given $(t_0,x_0,v_0),(t_1,x_1,v_1) \in \R^{1+2n}$. Further, the volume of the image of any measurable set $\bold M \subset \R^{1+2n}$ is 
\begin{equation*}
  \abs{\gamma^{\bold M}(r;(t,x,v))} = r^{1+2n}\abs{\bold M}
\end{equation*}
so it satisfies \hyperlink{link:C2}{\textbf{(C2)}}. But it also fails \hyperlink{link:C3}{\textbf{(C3)}} since $\dot \gamma_t^{\bold m} = m_0$, $\dot \gamma^{\bold m}_v=m_2$ and
\begin{equation*}
  \left((\nabla_{{m_1,m_2}} \gamma^{\bold m}_{x,v}(r;(t,x,v)))^{-1}\right)_{\cdot;2} = \left( \begin{matrix} 0 \\ r^{-1} \end{matrix} \right).
\end{equation*}

Second, we try to incorporate the kinetic scaling as follows
\begin{equation}
  \label{eq:in:cscurve}
  \gamma^{\bold m}(r;(t,x,v)) := (t+rm_0,x+r^{\frac{3}{2}}m_1, v+r^\frac{1}{2}m_2).
\end{equation}
This indeed ensures that now \hyperlink{link:C3}{\textbf{(C3)}} is satisfied since $\dot \gamma_t^{\bold m} = m_0$, $\dot \gamma^{\bold m}_v=\frac{1}{2}m_2 r^{-\frac{1}{2}}$ and
\begin{equation*}
  \left((\nabla_{{m_1,m_2}} \gamma^{\bold m}_{x,v}(r;(t,x,v)))^{-1}\right)_{\cdot;2} = \left( \begin{matrix} 0 \\ r^{-\frac{1}{2}} \end{matrix} \right)
\end{equation*}
However, $\dot \gamma^{\bold m}$ is along the vector field $\nabla_x$ so that \hyperlink{link:C1}{\textbf{(C1)}} is violated. 

Third, we try to use only the vector fields $\partial_t + v \cdot \nabla_x$ and $\nabla_v$, alternately. Each mapping that follows only one of these vector fields:
\begin{align}
  & \Phi_{\partial_t+v\cdot \nabla_x}(r,(t,x,v),m_0) := \gamma^{\bold m}(r;(t,x,v)) := (t+rm_0,x+rm_0 v,v) \quad  \mbox{ or } \label{eq:int:transport} \\
  & \Phi_{\nabla_v}(r,(t,x,v),m_2) :=	\gamma^{\bold m}(r;(t,x,v)) := (t,x,v+r^{\frac{1}{2}} m_2),\label{eq:int:nablav}
\end{align}
cannot reach enough points alone (and thus fails \hyperlink{link:C2}{\textbf{(C2)}}), but they respectively satisfy 
\begin{align*}
  & \frac{\dx}{\dx r}f(\gamma^{\bold m}(r)) = m_0 [(\partial_t+v\cdot \nabla_x) f](\gamma^{\bold m}(r)),\\
  & \frac{\dx}{\dx r}f(\gamma^{\bold m}(r)) = \frac{1}{2}r^{-\frac{1}{2}} m_2 \cdot [ \nabla_v f](\gamma^{\bold m}(r)).
\end{align*}
Therefore, \hyperlink{link:C1}{\textbf{(C1)}} is satisfied for each of these mappings.

While each of these mappings alone is unable to reach all points, a combination of these vector fields reveals the underlying hypoellipticity and allows us to reach the full set of points. Given $m_0 \neq 0$, $m_1, m_2 \in \R^n$, we can obtain the point
\begin{equation*}
 	(t+rm_0,x+r^{\frac{3}{2}}m_1, v+r^\frac{1}{2}m_2)
\end{equation*}
by the following combination of the maps~\eqref{eq:int:transport} and~\eqref{eq:int:nablav}: 
\begin{align}
  \gamma^{\bold m}(r;(t,x,v))
  &= (t+rm_0,x+r^{\frac{3}{2}}m_1, v+r^\frac{1}{2}m_2) \nonumber \\
  \label{eq:int:pwcurve}
  &= \Phi_{\nabla_v}\left(\frac{r}{3},\Phi_{\partial_t+v\cdot \nabla_x}\left(\frac{r}{3},\Phi_{\nabla_v}\left(\frac{r}{3},(t,x,v),\omega_1\right),\omega_2\right),\omega_3\right)
\end{align}
for the choice of parameters
\begin{equation*}
  \omega_1 = \sqrt{3} \left( \frac{m_1}{m_0}-r^{-\frac{1}{2} }v\right), \quad \omega_2 = 3m_0,\quad \omega_3 = \sqrt{3} \left( m_2-\frac{m_1}{m_0}+r^{-\frac12}v\right).
\end{equation*}
This choice of mapping also satisfies \hyperlink{link:C2}{\textbf{(C2)}} as in the previous attempt, since the endpoint is the same, but it fulfils \hyperlink{link:C1}{\textbf{(C1)}} only on the subcurves. However, as previously observed in~\cite{guerand_quantitative_2022}, the second subcurve fails the condition \hyperlink{link:C3}{\textbf{(C3)}}.

\bigskip

We now explain how to satisfy the conditions \hyperlink{link:C1}{\textbf{(C1)}}, \hyperlink{link:C2}{\textbf{(C2)}}, and \hyperlink{link:C3}{\textbf{(C3)}} at the same time. The previous attempts suggest to move along the vector fields $\partial_t+v\cdot \nabla_x$ and $\nabla_v$ \emph{simultaneously}. This is ensured by imposing the following condition on $\gamma^{\bold m}$:
\begin{equation}
  \label{eq:kinetic-relation}
  \dot{\gamma}_x^{\bold m}(r) = \dot{\gamma}_t^{\bold m}(r){\gamma}_v^{\bold m}(r)
\end{equation}
for all $r \in (0,\infty)$. 
This condition is the relation between position and velocity written in the parametrisation by $r$:
\begin{equation*}
  \gamma_v^{\bold m} = \frac{\partial \gamma_x^{\bold m}}{\partial \gamma_t^{\bold m}} = \frac{\dot{\gamma}_x^{\bold m}(r)}{\dot{\gamma}_t^{\bold m}(r)}.
\end{equation*}
It implies
\begin{equation*}
  \frac{\dx}{\dx r} f(\gamma^{\bold m}(r)) = \dot{\gamma}_t^{\bold m}(r)[(\partial_t+v \cdot \nabla_x)f](\gamma^{\bold m}(r)) +  \dot{\gamma}^{\bold m}_v(r)\cdot [\nabla_v f](\gamma^{\bold m}(r))
\end{equation*}
for all differentiable functions $f = f(t,x,v)\colon \R^{1+2n} \to \R$, so that \hyperlink{link:C1}{\textbf{(C1)}} is satisfied. 

Since $\dot \gamma_x$ is prescribed by the previous kinetic relation, we need to specify $\dot \gamma_t$ (time speed) and $\dot \gamma_v$ (the forcing). Our construction is based on the ansatz $\dot{\gamma}_t^{\bold m} = m_0$ with $m_0 \in \R \setminus \{ 0 \}$, and 
\begin{equation}
  \label{eq:int:forcing}
  \dot{\gamma}_v^{\bold m}(r) = \ddot{g}_1(r) m_1 + \ddot{g}_2(r) m_2
\end{equation}
for two vectorial parameters $(m_1,m_2) \in \R^{2n}$ and two control functions $g_1,g_2 \colon [0,\infty) \to \R$. If one assumes that particles have unit mass, and interprets $r$ as an artificial time along the trajectory, the right-hand side of~\eqref{eq:int:forcing} can be interpreted as a forcing. Integration of the identity \eqref{eq:int:forcing} together with $\gamma_v^{\bold m}(0) = v$ yields $\gamma_v^{\bold m}$, from which we may then deduce $\gamma_x^{\bold m}$ by integrating the kinetic relation $\dot{\gamma}_x^{\bold m}  = \dot{\gamma}_t^{\bold m} \gamma_v^{\bold m}$ and using that $\gamma_x^{\bold m}(0) = x$.

To achieve the critical scaling and maintain linear independence, we use \emph{desynchronised logarithmic oscillations} with the following forcing functions:
\begin{equation*}
  g_1(r) := r^{\frac{3}{2}} \cos(\log(r)) \quad \mbox{ and } \quad g_2(r) := r^{\frac{3}{2}} \sin(\log(r)).
\end{equation*}
The resulting curve is called \emph{kinetic trajectory} because it satisfies the kinetic relation~\eqref{eq:kinetic-relation}, and it is called \emph{critical} because, as we shall see, it satisfies \hyperlink{link:C2}{\textbf{(C2)}} and \hyperlink{link:C3}{\textbf{(C3)}} at the same time. We prove \hyperlink{link:C3}{\textbf{(C3)}} in Section~\ref{sec:kintraj}, and this allows us to perform an integration by parts along the trajectory at the right scale. The trajectory also provides a ``kinetic almost exponential map'' because, as we shall see, it satisfies \hyperlink{link:C2}{\textbf{(C2)}}. 

In particular, it is controllable as we can connect any two points $(t_0,x_0,v_0)$ and $(t_1,x_1,v_1) \in \R^{1+2n}$, with $t_0 \neq t_1$ by a careful choice of the generalised tangent variables $\bold m$: see~\eqref{eq:mconnect} for the precise formula. 

We note that we can not identify a canonical tangent vector at $r = 0$ as $\dot{\gamma}^{\bold m}(r;(t,x,v))$ does not exist as $r \to 0^+$. Even though the forcing is diverging in an oscillating manner at $r \to 0^+$, the map $\gamma^{\bold m}$ remains continuous at $r = 0$ and the volume scales like
\begin{equation*}
  \abs{\gamma^{\bold M}(r;(t,x,v))} = r^{\frac{2+4n}{2}} \abs{\bold M}
\end{equation*}
as indeed required by \hyperlink{link:C2}{\textbf{(C2)}}. The map $\gamma^{\bold m}$ satisfies other useful properties, which we collect and prove in Section~\ref{sec:kintraj}. We also show in Section~\ref{sec:smoothnotwork} that the critical scaling cannot be obtained in a large class of non-oscillating forcings. 

\subsection{Applications and consequences}

In the following, we consider a \emph{rough matrix of diffusion coefficients} $\mathfrak a = \mathfrak a(t,x,v) \in \R^{n \times n}$ that is uniformly elliptic, bounded, and measurable, in the Kolmogorov equation \eqref{eq:int:kol}. We give the precise assumptions on $\fra$ in Section~\ref{sec:mr}. The notion of weak solutions is explained in Definition \ref{def:weaksol}. For balls centered around the origin of radius $R>0$, we use the shorthand notation $B_R$.

\subsubsection*{Kinetic mollification and the Sobolev inequality}

In Riemannian geometry, regularising along the exponential map is reminiscent of regularising along the fundamental solution of the equation generating the geodesics. 

We use our \emph{kinetic almost exponential map}, built upon the critical kinetic trajectories, to introduce a concept of regularisation that relies solely on the kinetic vector fields $\partial_t +v \cdot \nabla_x$ and $\nabla_v$, and scales at the right way to work in kinetic function spaces, i.e.\ spaces which control $\partial_t+v\cdot \nabla_x$ in a negative Sobolev space.

We define the kinetic mollification of a function $f = f(t,x,v) \colon \R^{1+2n} \to \R$ in the direction $m_0 = \pm 1$ of time as 
\begin{equation}
  \label{eq:kinsmoothing}
  \left[S_r^{\varphi} (f)\right](t,x,v) :=  \int_{\R^{2n}}f\big(\gamma^{(m_0,m_1,m_2)}(r;(t,x,v))\big) \varphi(m_1,m_2) \dx \bold  (m_1,m_2)
\end{equation}
for some nonnegative smooth function $\varphi \in \C_c^\infty(\R^{2n})$ with unit integral. 

This object can be thought of as an approximation to the identity as $r\to 0^+$ in kinetic geometry at the scale of weak solutions. It shares some conceptual analogy with the regularisation obtained by convolution with the fundamental solution of the constant coefficient Kolmogorov equation, i.e.\ the solution of the Cauchy problem of \eqref{eq:int:kol} with $\fra = \id$ with initial value $f(t+ m_0 r,\cdot)$ evaluated at time $r$.

As an application, we obtain a quantitative proof of the following kinetic Sobolev inequality with the optimal gain of integrability of factor $\kappa = 1+\frac{1}{2n}$, by simple scaling arguments and avoiding the use of the fundamental solution of the Kolmogorov equation \eqref{eq:int:kol} with constant diffusion coefficient ($\fra = \id$).

\begin{theorem} \label{thm:int:sobolev}
  Let $f \in  \L^2(\R^{1+2n};\H^1(\R^n))$ such that $(\partial_t+v \cdot \nabla_x) f = \nabla_v \cdot S$ for some $S \in \L^2(\R^{1+2n};\R^n)$, then 
  \begin{equation*}
    \norm{f}_{\L^{2\kappa}(\R^{1+2n})} \le C \left(\norm{\nabla_v f}_{\L^{2}(\R^{1+2n})} +\norm{S}_{\L^{2}(\R^{1+2n})}  \right)
  \end{equation*}
  with $\kappa = 1+\frac{1}{2n}$ and $C = C(n)>0$.
\end{theorem}

Moreover, the kinetic mollification allows us to prove a localised version of this Sobolev inequality for weak (sub-, super-) solutions, which we state in Theorem \ref{thm:kinemb}. Both statements can be easily extended to $\L^p$-spaces with $p$ different from $2$. As a simple consequence of the interpolation of Lebesgue spaces and this Sobolev inequality, we obtain a kinetic Nash inequality in Theorem \ref{thm:nash}. 

\subsubsection*{A universal  estimate on the log-level sets in the spirit of Moser (1964)}

An important step in Moser's proof of the Harnack inequality for elliptic and parabolic equations with rough coefficients is a weak $\L^1$ estimate for the logarithm of positive supersolutions. We prove the kinetic analogue of this estimate for weak supersolutions to the Kolmogorov equation with rough coefficients \eqref{eq:int:kol}. A simplified version is the following.

\begin{theorem} \label{thm:int:logpoinc}
  There exists a universal constant $R>1$ and a constant $C>0$ that only depends on the dimension $n$ and the ellipticity constants of $\fra$, such that the following holds. 
  
  For any positive weak supersolution $f
  \in \L^\infty_t\L^2_{x,v} \cap \L^2_{t,x}\H^1_v$ (in the distributional sense) to the Kolmogorov equation~\eqref{eq:int:kol} in
  \begin{equation*}
    \tilde{Q} = (0,1) \times B_{2R} \times B_{2R},
  \end{equation*}
  there is a constant $c = c(f) \in \R$ depending on $f$ such that 
  \begin{equation*}
    \forall \, s >0, \quad
    \begin{cases}
      \displaystyle
      \big|\big\{ (t,x,v) \in K_- \, \colon \, \log f-c(f)>s \big\}\big| \le \frac{C \abs{K_-}}{s},  \\[3mm]
      \displaystyle
       \big|\big\{ (t,x,v) \in K_+ \, \colon \, c(f)-\log f>s \big\}\big| \le \frac{C \abs{K_+}}{s},
    \end{cases}
  \end{equation*}	
  where
  \begin{equation*}
    K_- = \left(0,\frac{1}{3}\right) \times B_{1} \times B_1 \quad \text{ and } \quad 
    K_+ = \left(\frac{2}{3},1\right) \times B_1 \times B_1.
  \end{equation*}
  The constant $c(f)$ is a weighted average of $\log f$ at time $t = \frac{1}{2}$ with respect to the variables $x$ and $v$ in $B_{2R}(0)$.
\end{theorem}

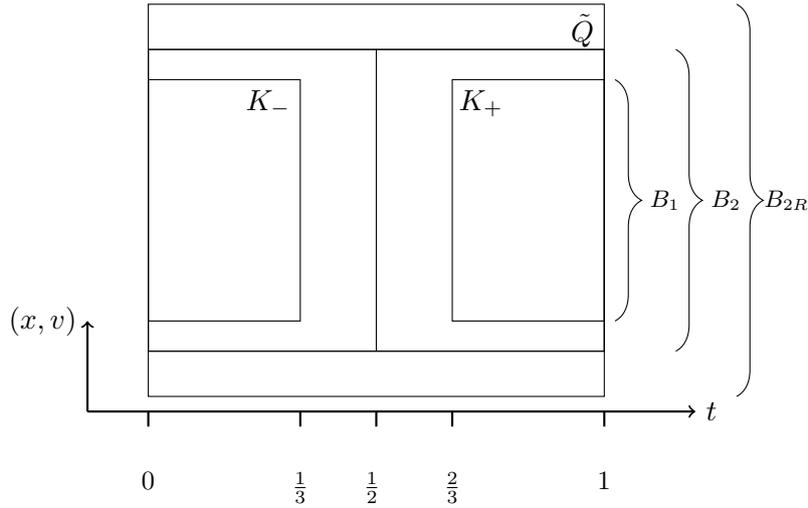
\begin{figure}[H]
\centering
\tikzmath{\x = 0.2; \y = 0.2;}
\begin{tikzpicture}[scale=4]
  	\draw (1.15-\x,0.4-\y) -- (1.15-\x,1.4-\y);
  	\draw[thick,->] (0, 0) -- (2,0) node[right] {$t$};
  	\draw[thick,->] (0, 0) -- (0,0.3) node[left] {$(x,v)$};
  	\draw[draw=black] (0.4-\x,0.25-\y) rectangle ++(1.5,1.3);
  	\draw[draw=black] (0.4-\x,0.4-\y) rectangle ++(1.5,1);
  	\draw[draw=black] (0.4-\x,0.4-\y) rectangle ++(1.5,1);
  	\draw[draw=black] (0.4-\x,0.5-\y) rectangle ++(0.5,0.8);
  	\draw[draw=black] (1.4-\x,0.5-\y) rectangle ++(0.5,0.8);
  	\draw (0.9-\x,1.3-\y) node[anchor=north east] {$K_-$};
 	\draw (1.6-\x,1.3-\y) node[anchor=north east] {$K_+$}; 
  	\draw (1.9-\x,1.55-\y) node[anchor=north east] {$\tilde{Q}$};
  	\draw [thick] (0.4-\x, 0) -- ++(0, -.05) ++(0, -.15) node [below, outer sep=0pt, inner sep=0pt] {\small\(0\)};
  	\draw [thick] (1.4-\x, 0) -- ++(0, -.05) ++(0, -.15) node [below, outer sep=0pt, inner sep=0pt] {\small\( \frac{2}{3} \)};
  	\draw [thick] (1.15-\x, 0) -- ++(0, -.05) ++(0, -.15) node [below, outer sep=0pt, inner sep=0pt] {\small\(\frac{1}{2} \) \vphantom{\small\(\eta+\iota \)}};
  	\draw [thick] (0.9-\x, 0) -- ++(0, -.05) ++(0, -.15) node [below, outer sep=0pt, inner sep=0pt] {\small\(\frac{1}{3}\)};
  	\draw [thick] (1.9-\x, 0) -- ++(0, -.05) ++(0, -.15) node [below, outer sep=0pt, inner sep=0pt] {\small\(1 \)};
	\draw [decorate,decoration={brace,amplitude=10pt,mirror,raise=4pt},yshift=0pt] (1.9-\x,0.5-\y) -- (1.9-\x,1.3-\y) node [black,midway,xshift=0.8cm] {\footnotesize ${B_1}$};
	\draw [decorate,decoration={brace,amplitude=10pt,mirror,raise=4pt},yshift=0pt] (2.3-\x,0.25-\y) -- (2.3-\x,1.55-\y) node [black,midway,xshift=0.8cm] {\footnotesize $B_{2R}$};
	\draw [decorate,decoration={brace,amplitude=10pt,mirror,raise=4pt},yshift=0pt] (2.1-\x,0.4-\y) -- (2.1-\x,1.4-\y) node [black,midway,xshift=0.8cm] {\footnotesize ${B}_{2}$};
\end{tikzpicture}
\end{figure}
We present a rescaled version with some more degrees of flexibility in the choice of cylinders in Theorem~\ref{thm:weakl1poin}. We emphasise that the assumption that $f$ is a weak supersolution in the large cylinder $\tilde{Q}$ is merely qualitative and no values of $f$ outside of $(0,1) \times B_2 \times B_2$ appear in the estimate.

\subsubsection*{Extension of Moser-Bombieri-Giusti's proof of the Harnack inequality to the kinetic setting}

The (weak) $\L^1$ estimate for weak supersolutions to the Kolmogorov equation is the missing puzzle piece to establish an alternative proof of the weak Harnack inequality for weak solutions to the Kolmogorov equation based on the ideas of Moser-Bombieri-Giusti~\cite{moser_pointwise_1971,bombieri_harnacks_1972}. In Section~\ref{sec:harnack} we provide the precise statements and compare them carefully to the literature. We obtain, for the first time in the kinetic setting, the weak Harnack inequality in Theorem \ref{thm:weakH} with optimal range for the exponent ($p \in (0,1+\frac{1}{2n})$) and the Harnack inequality in Theorem \ref{thm:harnack} with the optimal dependency of the Harnack constant on the ellipticity constants of the diffusion matrix $\fra$. We obtain $C_{\text{Harnack}} = C^\mu$, for a dimensional constant $C>0$ and where $\mu$ is the sum of the upper bound on $\fra$ and the reciprocal of the lower bound on $\fra$, see Section~\ref{sec:mr} \hyperlink{link:H1}{\textbf{(H1)}} and \hyperlink{link:H2}{\textbf{(H2)}} for the precise definition.

To achieve the latter, we need to redo the proof of the well-known $\L^p-\L^\infty$ estimates for weak (sub-, super-) solutions, usually obtained by ``De~Giorgi-Moser iteration'', while keeping track of the dependency of the constants on the ellipticity constants of $\fra$. We are going to prove that the constants in the $\L^{p}-\L^\infty$-estimate for the reciprocal of weak supersolutions, and in the $\L^p-\L^\infty$ estimate for weak solutions, with $0<p<\frac{1}{\mu}$ can be chosen independently of $\mu$. We refer to Proposition \ref{prop:u-1} and Proposition \ref{prop:uLinfLpsmallp}. Moreover, we give a new estimate, which we could not find in the literature: an $\L^{p_0}-\L^{\gamma}$-estimate for weak supersolutions with $p_0 \in (0,\kappa)$ and $\gamma \in (0, \frac{p_0}{\kappa}]$ in Proposition \ref{prop:ulplg}, where again, the constant is independent of $\mu$ for small $p$. The latter will be used in the proof of the weak Harnack inequality for weak supersolutions. Moreover, as a by-product, we obtain that the constant in the $\L^p-\L^\infty$ bound for weak subsolutions for all $p \in (0,\infty)$ scales like $\mu^{\frac{4n+2}{p}}$, see Theorem \ref{thm:loclinf}.

We follow Moser's approach based on testing with powers of the weak (sub-, super-) solution itself and deriving an energy estimate. This energy estimate is then combined with the kinetic Sobolev inequality of Theorem \ref{thm:kinemb} to deduce a gain of integrability. This inequality can then be iterated on a sequence of nested cylinders to obtain increasing Lebesgue exponents. To track the precise dependency on the ellipticity constants, we provide the corresponding abstract De~Giorgi-Moser iterations with explicit constants in the Appendix \ref{sec:abstract}.

Altogether, this circle of arguments provide a quantitative proof of the (weak) Harnack inequality based purely on the underlying kinetic geometry and on energy estimates obtained by testing the equation with the (super-) solution itself. Our proof is self-contained (besides the well-known lemma of Bombieri and Giusti and the abstract De~Giorgi-Moser iterations), and it does not rely on the fundamental solution of the Kolmogorov equation with constant coefficients. We only go back to this fundamental solution in the end to demonstrate the optimality of our statements.

\subsubsection*{On working with weak solutions to kinetic equations}

In comparison to the existing literature, we are able to work with a broader class of weak (sub-) supersolutions that includes the natural space dictated by the energy estimate. A precise comparison with the literature can be found in Definition \ref{def:weaksol} and in Remark \ref{rem:weaksol}. We explain in Section~\ref{sec:rigorous} how to make the formal calculations of this article rigorous, i.e.\ how to work with our notion of weak (sub-, super-) solutions. The method applies to previous results, when a stronger notion of weak solutions was used, or when the notion of weak solutions used was not properly discussed.

\subsubsection*{Higher-order kinetic vector fields}

Our method can be generalised to a broad class of more degenerate hypoelliptic equations of Kolmogorov type, as proposed in the work of H\"ormander~\cite{hormander_hypoelliptic_1967} and studied, for instance, in~\cite{lanconelli_class_1994} and many subsequent works.

Given $k \in \N$, the prototypical equation, on the unknown $f = f(t,x^1,\dots,x^k)$ depending on variables $x^i \in \R^{d_i}$ with $1 \le d_1 \le \dots \le d_{k-1} \le d_k$, writes as follows: 
\begin{equation} \label{eq:int:kolk}
 \left( \partial_t  +  \sum_{j = 1}^{k-1} \left( \mathfrak b_{j} x^{j+1} \right) \cdot \nabla_{x^j} \right) f = \nabla_{x^k} \cdot \left( \mathfrak a \nabla_{x^k} f \right).
\end{equation}
The matrices $\br_1, \dots, \br_{k-1}$ have a ``cascade structure'', see~\eqref{eq:frab} for more details. Here, $(k-1)$-commutators are required between the vector fields
\begin{equation*}
  \left( \partial_t  +  \sum_{j = 1}^{k-1} \left( \mathfrak b_{j} x^{j+1} \right) \cdot \nabla_{x^j} \right) \quad \text{ and } \quad \nabla_{x_k}
\end{equation*}
to span the tangent space at any point.

We construct higher-order critical kinetic trajectories for the geometry induced by these vector fields. With the help of these trajectories, we can generalise our results to this more general class of equations. Note that this general setting includes the parabolic case for $k = 1$ and the kinetic setting for $k = 2$. Actually, once the trajectories are constructed, the proof of the weak Harnack inequality is virtually identical for parabolic, kinetic and higher-order kinetic equations, and their stationary variants.

\subsection{A review of the related literature}
A driving force for many beautiful regularity results in the theory of kinetic partial differential equations is the quest to understand the Landau and Boltzmann equations. We note that a proof of the global existence of solutions to the homogeneous Landau and Boltzmann equation (no $x$-dependency) has been recently obtained in \cite{guillen2025landauequationdoesblow} and \cite{imbert2024monotonicityfisherinformationboltzmann}. It is therefore of great interest to understand the underlying kinetic geometry if one hopes to transfer regularity results from the spatially homogeneous setting to the kinetic setting. 

\subsubsection{Kinetic De~Giorgi-Nash-Moser theory}

This theory is concerned with the a priori H\"older regularity of weak solutions $f=f(t,x,v)$ to the Kolmogorov equation~\eqref{eq:int:kol} with rough diffusion coefficients. Two important ingredients in the classical De~Giorgi-Nash-Moser theory for elliptic and parabolic equations are the Sobolev inequality (gain of integrability) and the Poincar\'e inequality. The following contributions have developed kinetic counterparts to these inequalities. 

A.~Pascucci and S.~Polidoro obtained in~\cite{pascucci_mosers_2004} the local boundedness---the first lemma of De~Giorgi---of weak subsolutions to~\eqref{eq:int:kol} by following Moser's iteration scheme~\cite{moser_harnack_1964} combined with a gain of integrability for weak (sub-, super-) solutions similar to the theory of \emph{averaging lemma}~\cite{golse_regularity_1988} (which we explain below). This work was later extended in~\cite{anceschi_mosers_2019} to include lower-order terms and a source term. 

H\"older continuity of weak solutions was first obtained by W.~Wang and L.~Zhang in \cite{wang_calpha_2009,zhang_calpha_2011,wang_calpha_2011} (see also the recent refinement in~\cite{wang_calpha_2019}) by a constructive method based on Moser's approach modified by Kruzhkov~\cite{MR171086}. This method relies on a Poincar\'e type inequality; also the proof in~\cite{wang_calpha_2009,zhang_calpha_2011,wang_calpha_2011} weakens this inequality by including an intricate corrector that satisfies an auxiliary equation (see the discussion below). This proof was simplified and revisited by J.~Guerand and C.~Imbert in~\cite{guerand_log-transform_2022}, in order to obtain the Harnack inequality with quantitative arguments.

Following the original De~Giorgi method, F.~Golse, C.~Imbert, the second author and A.~Vasseur obtained in~\cite{golse_harnack_2019} the first proof of both the second lemma of De~Giorgi---reduction of oscillation by rescaling---and the Harnack inequality for kinetic equations, by 
extending the method of De~Giorgi combined with averaging lemmas and ideas from homogenisation theory; this proof was, however
non-constructive; in particular, a contradiction argument is used in the proof of the kinetic version of De~Giorgi's ``intermediate value lemma''.

J. Guerand and the second author gave in~\cite{guerand_quantitative_2022} a shorter quantitative proof of the second De~Giorgi lemma (reduction of oscillation) by introducing a method based on trajectories. This method is the starting point of the subsequent works~\cite{niebel_kinetic_2022-1,anceschi2024poincare} and of this paper. The core of the method is, first, a quantitative proof of a kinetic Poincar\'e inequality by trajectories, and, second, a quantitative proof of De~Giorgi intermediate value lemma based on this kinetic Poincar\'e inequality (this last argument draws inspiration from an elliptic-parabolic argument in~\cite{MR4398231}). They also simplified the proof of the weak Harnack inequality by successfully adapting ideas from elliptic equations and showing that the measure-to-pointwise estimate at all scales is enough to yield the weak Harnack inequality by elementary covering arguments. This proof was significantly simplified in~\cite{niebel_kinetic_2022-1,anceschi2024poincare} and can be further simplified thanks to the ideas in the present article, as observed in the lecture notes \cite{bm_degiorgi_2025}.

The first author, together with J. Hirsch, established in~\cite{dietert2022regularity} the local boundedness as well as a ``measure-to-pointwise'' estimate, by a related method based upon estimates on the fundamental solution of the hypoelliptic problem with smooth coefficients.

For a review of the development until 2018, we refer to
\cite{mouhot_DGNMH_2018}. More recent surveys can be found in~\cite{MR4865535}, mostly covering non-quantitative methods, and~\cite{bm_degiorgi_2025}, more focused on quantitative methods.

\subsubsection{Kinetic Poincar\'e inequalities and kinetic trajectories}

Poincar\'e inequalities are a central ingredient in the De~Giorgi-Nash-Moser theory. A first ``modified'' kinetic Poincar\'e inequality for weak subsolutions to the Kolmogorov equation with rough coefficients was obtained in~\cite{wang_calpha_2011,wang_calpha_2009,zhang_calpha_2011}, with an additional corrector, obtained by solving an auxiliary hypoelliptic equation with smooth coefficients. We refer to~\cite{guerand_log-transform_2022} for a clean presentation. The proof crucially relies on the explicit formula of the fundamental solution. 

A core result of \cite{guerand_quantitative_2022} is the
constructive proof of a kinetic Poincar\'e inequality in $\L^1$ for
nonnegative subsolutions to the Kolmogorov equation. The starting point is similar to the proof of the elliptic Poincar\'e inequality. To compare a function with its mean, one needs to connect two points and integrate the function along this curve; however, in contrast with the elliptic case, the curve now must be a kinetic trajectory rather than a line. In~\cite{guerand_quantitative_2022}, the kinetic trajectory used is similar to~\eqref{eq:int:pwcurve}, with an additional argument to tame the singularity by a small curve along the $\nabla_x$ vector field; the latter is in turn controlled by a transfer of low-level Sobolev regularity from the $v$ variable to the $x$ variable for subsolutions, in the spirit of averaging lemmas, and established in~\cite{guerand_quantitative_2022}.

Inspired by this work, kinetic trajectories were proposed in \cite{niebel_kinetic_2022-1}, leading to a neater kinetic Poincar\'e inequality and avoiding the use of control in the $x$-variable in the form of an averaging lemma. These kinetic trajectories combine \emph{simultaneously} the vector fields $\partial_t+v\cdot \nabla_x$ and $\nabla_v$. This new construction of trajectories in~\cite{niebel_kinetic_2022-1}, crucially relying on the kinetic relation~\eqref{eq:kinetic-relation}, was then simplified, interpreted from an optimal control viewpoint, and extended to the higher-order Kolmogorov equation~\eqref{eq:int:kolk} in~\cite{anceschi2024poincare}. The forcings in~\cite{niebel_kinetic_2022-1,anceschi2024poincare} correspond to power functions of the form $g_1(r) = r^{\frac{3}{2}+\epsilon_1}$ and $g_2(r) = r^{\frac{3}{2}+\epsilon_2}$ for some $\epsilon_1,\epsilon_2 \in (1,2)$ in \eqref{eq:int:forcing}. The corresponding exponential map fails properties \hyperlink{link:C2}{\textbf{(C2)}} and \hyperlink{link:C3}{\textbf{(C3)}}. In fact, they cannot be used to reach criticality, as shown in Section~\ref{sec:smoothnotwork}.
  
Another global kinetic Poincar\'e inequality in $\L^2$-spaces with a Gaussian weight was proven in \cite{albritton2021variational} by the use of the Bogovski\u{\i} inversion of the divergence operator in $(t,x)$-variables~\cite{zbMATH03787290,zbMATH03753729}. This idea is also used in \cite{dietert2023quantitativegeometriccontrollinear} in the context of quantitative geometric control in linear kinetic theory and in \cite{guerand2024gehringslemmakineticfokkerplanck} to prove a local kinetic Poincar\'e inequality in $\L^p$-spaces for $p \in [1,2]$, which is then employed to deduce a kinetic version of the Gehring's lemma.  
  
\subsubsection{Kinetic Sobolev inequality}

The De~Giorgi-Moser iterative scheme is based on the gain of integrability for weak (sub-, super-) solutions; it proves, for example the local boundedness of subsolutions. In the elliptic and parabolic setting, this gain of integrability is a consequence of the Sobolev inequality. We call \emph{kinetic Sobolev inequality} a similar gain of integrability --- in all variables --- in the kinetic case. Unlike the classical Sobolev inequality, it will be formulated for functions that are (sub-) supersolutions to a partial differential equation rather than elements of some Sobolev space.

Such a kinetic Sobolev inequality can be deduced from the transfer of regularity in the work of H\"ormander~\cite{hormander_hypoelliptic_1967}, when $f \in \L^2_{t,x}\Hdot^1_v$ and $(\partial_t+v \cdot \nabla_x) f \in \L^2_{t,x}\H^{-1}_v$; H\"ormander's proof shows that $f$ is in $\H^\varepsilon_{t,x}\L^2_v$ locally for some small $\varepsilon>0$. Then $\L^2_{t,x}\Hdot^1_v \cap \H^\varepsilon_{t,x}\L^2_v \subset \H^{\varepsilon'}_{t,x,v}$ locally, for some $\varepsilon'\in (0,\varepsilon)$, and the gain of integrability in all variables --- the kinetic Sobolev inequality --- is deduced from the standard Sobolev inequality for fractional derivatives; the exponent is, however, not optimal. This strategy crucially relies on the assumption $(\partial_t+v \cdot \nabla_x) f \in \L^2_{t,x}\H^{-1}_v$ and therefore does not apply to weak subsolutions in the energy class $\L^\infty_t\L^2_{x,v} \cap \L^2_{t,x}\Hdot^1_v$. 

The gain of integrability for weak subsolutions to~\eqref{eq:int:kol} is established in~\cite{pascucci_mosers_2004} with the optimal exponent; the authors prove an \emph{interior} gain of integrability $\L^2_{t,x,v}$ to $\L^{2\kappa}_{t,x,v}$ with $\kappa = 1+\frac{1}{2n}$. It, however, assumes (qualitatively) that $(\partial_t+v \cdot \nabla_x) f \in \L^2_{t,x,v}$ in order to justify the calculations in the iterative scheme. The transfer of regularity in this proof is based on precise estimates for the fundamental solution. It was later refined in
\cite{anceschi_note_2022}, and \cite{guerand_quantitative_2022}, where the qualitative assumption was respectively relaxed to $(\partial_t+v \cdot \nabla_x) f \in \L^2_{t,x}\H^{-1}_v$, and to $\L^\infty_t\L^2_{x,v}$ together with the validity of a renormalisation formulation.

In \cite{golse_harnack_2019}, a comparable result is established, however, with a suboptimal exponent: the transfer of regularity in the proof is based on an averaging lemma, in the variant proved by F.~Bouchut~\cite{bouchut_hypoelliptic_2002}. We emphasise that this proof does not require any knowledge of the fundamental solution and has interest per se for dealing with more general degenerate transport operators. It is assumed that the subsolutions satisfy $(\partial_t+v \cdot \nabla_x) f \in \L^2_{t,x}\H^{-1}_v$ in this work.

\subsubsection{Weak solutions to kinetic equations}

Most of the existing literature deals with the regularity theory of a priori
given solutions. The existence and uniqueness of weak solutions to kinetic equations with local or integral diffusion, including the Kolmogorov equation with rough coefficients, is studied in~\cite{auscher_weak_2024} (see also the previous influential works~\cite{MR875086,albritton2021variational}). We refer to the introduction of~\cite{auscher_weak_2024} for a historical overview of this topic. The notion of weak solutions we use in the current paper is directly inspired by~\cite{auscher_weak_2024}. Let us also mention the related question of $\L^p$ solutions to the Kolmogorov equation, studied in~\cite{niebel_kinetic_2022,niebel_kinetic_2021,MR3906169,MR4444079,MR4704640}.

\subsubsection{Further developments in kinetic regularity theory}

We briefly list other important results:
\renewcommand\labelitemi{\tiny$\bullet$} 
\begin{itemize}
\item the Harnack inequality for the constant coefficient Kolmogorov equation in \cite{MR571952,MR998126,lanconelli_class_1994}, and a proof in the full space case via the Li-Yau inequality \cite{pascucci_harnack_2004},
\item geometric version of the Harnack
inequality~\cite{anceschi_geometric_2019},
\item existence of the fundamental solution to \eqref{eq:int:kol} with rough coefficients and
pointwise upper bounds~\cite{auscher_fundamental_2024},
\item pointwise lower bound on the fundamental
solution~\cite{anceschi_fundamental_2023},
\item gradient estimates for nonlinear variants of \eqref{eq:int:kol}
in~\cite{kim2025gradientestimatesnonlinearkinetic},
\item regularity up to the boundary
in~\cite{silvestre_boundary_2022,zhu2022regularity,hou_boundedness_2024},
\item the treatment of more general transport operators, including the relativistic transport operator in~\cite{zhu_velocity_2021},
\item weak Harnack inequality for non-local kinetic
operators in~\cite{stokols_nonlocal_2019,imbert_weak_2020} and (quantitatively by the De~Giorgi trajectory method) in~\cite{MR4688651},
\item a semi-local Harnack inequality in the non-local
case~\cite{loher2024local},
\item a counter-example to the strong Harnack inequality for non-local kinetic equations when the equation is satisfied in a bounded domain in the velocity variable in~\cite{kassmann2024harnack},
\item Schauder estimates for local and non-local kinetic equations in~\cite{di_francesco_schauder_2006,imbert_schauder_2021,imbert_schauder_2021,lunardi_schauder_1997,loher2023quantitativeschauderestimateshypoelliptic,menozzi_martingale_2018}. 
\end{itemize}

\subsubsection{On hypoelliptic vector fields}

In the case of hypoelliptic operators without a drift, i.e.\ ``H\"ormander sum-of-squares operators of type A'', the Sobolev and Poincar\'e inequalities as well as the De~Giorgi-Nash-Moser theory, and, in particular, the Harnack inequality, are well-understood~\cite{jerison_poincare_1986,lu_weighted_1992}. The work~\cite{lanconelli_poincare_2000} proves a Poincar\'e inequality for vector fields which admit a controllable almost exponential map and we take some inspiration from their work for the interpretation of our critical kinetic trajectories.  

However, these results do not apply to the so-called ``H\"ormander hypoelliptic operators of type B'' that include a first-order skew-symmetric ``drift'' term, where the drift direction has a different weight than the diffusive directions. Given a family of vector fields  $X_0,\dots,X_k$ on $\R^n$ with an underlying homogeneous Lie group structure and satisfying a controllability condition, and a matrix $\fra$ of rough diffusion coefficients with the usual ellipticity bounds, the a priori H\"older continuity of weak solutions $f$ to 
\begin{equation*}
  X_0 f + \sum_{i = 1}^k X_i^* \fra X_i f= 0
\end{equation*}
was first proven in~\cite{wang_calpha_2011}. The proof relies on the $\L^p-\L^\infty$ bounds in~\cite{MR2386472,pascucci_mosers_2004}, together with a generalised Poincar\'e inequality with a corrector satisfying an auxiliary hypoelliptic equation with smooth coefficients. The first author, together with J. Hirsch, established in~\cite{dietert2022regularity} the measure-to-pointwise estimate for this equation by a related but different method, which also relies upon estimates on the solution to the hypoelliptic problem with smooth coefficients.

In the context of constant diffusion coefficients $(\mathfrak{a}_{ij}) \in \R^{k \times k}$, an invariant Harnack inequality is proven in \cite{MR2088032}. In the general case, i.e.\ without an underlying Lie group structure, we want to mention the Harnack inequality of Bony \cite{MR262881}.

An interesting special class of hypoelliptic operators of type B is provided by the higher-order Kolmogorov equation~\eqref{eq:int:kolk}. This special class has been studied more extensively: the $\L^2-\L^\infty$ bound for weak (sub-) solutions was first proven in~\cite{pascucci_mosers_2004} (see also~\cite{wang_calpha_2019,anceschi_mosers_2019,anceschi2024poincare}), and the H\"older continuity of weak solutions is studied in \cite{zhang_calpha_2011,wang_calpha_2009,wang_calpha_2011,anceschi2024poincare}. The Harnack inequality is obtained in~\cite{anceschi_note_2022} and \cite{anceschi2024poincare} with proofs inspired, respectively, by the arguments in \cite{guerand_log-transform_2022} and \cite{guerand_quantitative_2022,niebel_kinetic_2022-1}. In \cite{anceschi2024poincare}, a quantitative Poincar\'e inequality is proven by the trajectorial argument of \cite{guerand_quantitative_2022,niebel_kinetic_2022-1}. 

Let us finally mention that the idea of constructing ``hypoelliptic trajectories'', i.e.\ integral curves along the vector fields satisfying a bracket-generating condition, is older and of broader relevance. It features in~\cite{rs_connect_1938,chow_uber_1939,MR1511534}, and in the more recent~\cite{nagel_balls_1985}. Such trajectories for the vector fields defining \eqref{eq:int:kolk} were also constructed in~\cite{pascucci_harnack_2004} (in a subcritical way, however). The higher-order kinetic trajectories proposed in \cite{anceschi2024poincare} were almost, but not quite, critical in the case $d_1 = \dots =d_k $ and far from critical if the dimensions are not the same.

On a different topic, in the recent work \cite{bedrossian_regularity_2022}, the precise understanding of exponential maps for hypoelliptic vector fields plays an important role in estimating the Lyapunov exponent for stochastic differential equations.

\subsubsection{The Moser-Bombieri-Giusti approach}
\label{sec:mbg}
The theory of a priori H\"older estimates for linear elliptic and parabolic equations with rough coefficients is often referred to as the De~Giorgi-Nash-Moser theory. In order to shed light on the specific contributions of Moser, Bombieri and Giusti, let us give a brief historical overview.

Given an open set $\Omega \subset \R^n$ and bounded measurable uniformly elliptic diffusion coefficients $\fra=\fra(v)$, the H\"older continuity of weak solutions $f=f(v) \colon \Omega \to \R $ to the \textit{elliptic} equations (in divergence form)
\begin{equation*}
  -\nabla_v \cdot (\fra \nabla_v f) = 0,
\end{equation*}
was first proved by E.~De~Giorgi in 1957~\cite{de_giorgi_sulla_1957}.

A different proof was proposed independently by J.~F.~Nash in 1958 \cite{nash_continuity_1958}, who studied the \textit{parabolic} equation
\begin{equation}
  \label{eq:int:par}
  \partial_t f = \nabla_v \cdot (\fra \nabla_v f),
\end{equation}
where $f =f(t,v) \colon (0,T) \times \Omega \to \R$ and $\fra = \fra(t,v)$ is a uniformly elliptic bounded matrix of measurable coefficients.

The results of De~Giorgi and Nash solved \emph{Hilbert's nineteenth problem} regarding the regularity of solutions to variational problems. Moreover, they opened up many interesting developments in the regularity theory of partial differential equations. In particular, the a priori H\"older estimate can be used to study quasilinear equations in higher dimensions (see e.g.~\cite{gilbarg_elliptic_2001}) or to prove the global-in-time existence of strong or classical solutions to parabolic equations (see e.g.~\cite{zacher_global_2012}). 

Later, in 1961-1964~\cite{moser_harnacks_1961,moser_harnack_1964,moser_correction_1967}, J.~Moser proposed an alternative proof of the De~Giorgi-Nash theorem.  He proved a stronger regularity estimate: a \textit{Harnack inequality}, which implies H\"older continuity. In 1971~\cite{moser_pointwise_1971}, J.~Moser published a simplified version of his proof of the parabolic Harnack inequality. The new ingredient in the paper~\cite{moser_pointwise_1971} is the use of a lemma due to E.~Bombieri and E.~Giusti~\cite{bombieri_harnacks_1972}, in order to combine  $\L^p-\L^\infty$ estimates for small $p>0$ with a universal logarithmic weak $\L^1$-estimate to deduce the Harnack inequality. This significantly simplified Moser's original proof~\cite{moser_harnack_1964} by avoiding the recourse to $\mathrm{BMO}$-functions. This allowed him to obtain the elliptic and parabolic Harnack inequalities with optimal dependency on the ellipticity constants. It was later observed in \cite{saloff-coste_aspects_2002} that this approach can also be used to prove the weak Harnack inequality, see also \cite{clement_priori_2004}.

This strategy has then been applied to several other contexts, e.g. to H\"ormander vector fields of type A (no drift) by G.~Lu~\cite{lu_weighted_1992}, to Markov chains on graphs by T.~Delmotte~\cite{delmotte_parabolic_1999}, to parabolic non-local problems by M.~Felsinger and M.~Kassmann~\cite{felsinger_local_2013}, and to time-fractional equations by the last author~\cite{zacher_weak_2013}. We also refer to the recent papers~\cite{albritton_regularity_2023,bonforte_explicit_2020,MR4848678,MR4801835} that employ this strategy.

An important ingredient in Moser's proof of the parabolic Harnack inequality is the fact that for any positive weak supersolution $f $ the function $\log f$ is a weak supersolution to the following equation 
\begin{equation} \label{eq:int:quad}
  \partial_t \log f = \nabla_v \cdot (\fra \nabla_v \log f) + \langle \fra \nabla_v \log f , \nabla_v \log f \rangle,
\end{equation}
which features an additional \emph{nonnegative quadratic} term on the right-hand side. Using the full power of this quadratic term together with a weighted $\L^2$ Poincar\'e inequality gives a universal control on $(\log f - \langle \log f \rangle_v)$, where we subtract a weighted $v$-average, for a fixed time. This estimate is then propagated in time by an intricate argument. The word ``universal'' here means that the control is independent of $f$. This finally leads to a universal weak $\L^1$-estimate on $\log f -c(f)$, where $c(f)$ is a weighted mean of $\log f$ at some intermediate time. This estimate is crucial to make the ``crossover'' between the $\L^p-\L^\infty$ estimates on $f$ and $1/f$ for small $p>0$ with the lemma of Bombieri and Giusti, see~\cref{sec:abstract}.

This ingredient in Moser's proof is already hidden in the work of J.F.~Nash in 1958
\cite{nash_continuity_1958} (see \cite[pp. 65--98]{MR3930576} for a modern
presentation). There, Nash proved a ``$G$-bound'', i.e.\ a lower bound of the form
\begin{equation} \label{eq:nash:G}
  \int_{\R^n} e^{-|w|^2} \log \left[ t^{\frac{n}2} F\left( t, \sqrt t w \right) + \delta \right] \dx w \gtrsim - \sqrt{- \log \delta}
\end{equation}
for a fundamental solution $F$ born from a Dirac mass; his proof relies on the nonnegative term in~\eqref{eq:int:quad}. The difference in homogeneity
on both sides of this inequality is the key to establishing universal lower
bounds on the fundamental solution in Nash's proof. Nash's idea of using the
logarithm was also used by Kruzhkov~\cite{MR171086,MR151703} to prove a lower bound similar to~\eqref{eq:nash:G}, allowing him to simplify Moser's original method by bypassing the issue of connecting integral bounds on $f$ and $1/f$. This lower bound was then combined by Kruzhkov with an estimate of the propagation of positivity and a classical covering argument to deduce the Harnack inequality. The conclusion is that implementing Moser's approach requires proving the (weak) $\L^1$ estimate for the logarithm of supersolutions. 

In the existing literature on the parabolic case~\cite{felsinger_local_2013,zacher_weak_2013,delmotte_parabolic_1999,albritton_regularity_2023,bonforte_explicit_2020,MR4848678,lu_weighted_1992}, a log-estimate is always established first in the spatial/diffusive variables, followed by a second separate step to propagate it in time. It was, however, shown by the third and fourth authors in~\cite{niebel_kinetic_2022-1} that the (weak) $\L^1$ estimate for the logarithm of nonnegative supersolutions to~\eqref{eq:int:par} can be established in a simple and quantitative way by using critical parabolic trajectories. We follow this strategy to 
prove Theorem \ref{thm:int:logpoinc} in the kinetic setting.

\subsection{Outline of the article}

Section~\ref{sec:kintraj} deals with the construction of critical kinetic trajectories and their properties. The extension of this construction to higher-order kinetic vector fields is discussed in Section~\ref{sec:hypo}. 
The assumptions on the diffusion coefficient matrix, kinetic cylinders, and the notion of weak solutions are introduced in Section~\ref{sec:mr}. 
Section~\ref{sec:sobolev} is concerned with the kinetic Sobolev inequality. 
In Section~\ref{sec:weakL1log} we prove the universal (weak) $\L^1$-estimate for the logarithm of nonnegative supersolutions to the Kolmogorov equation. In Section~\ref{sec:harnack}, we collect the statements of the main results on the Kolmogorov equation~\eqref{eq:int:kol} (and its higher-order counterpart): local boundedness, weak Harnack inequality, strong Harnack inequality, H\"older continuity. In Section~\ref{sec:mvineq}, we provide proofs for the mean value inequalities, which arise when testing the equation with powers of the (sub-, super-) solution itself. Finally, in Section~\ref{sec:proofsmr}, we prove the weak Harnack inequality along the Moser-Bombieri-Giusti approach, thanks to the universal (weak) $\L^1$-estimate established in Section~\ref{sec:weakL1log}.

We then collect various useful technical results in the appendices. In the Appendix~\ref{sec:rigorous}, we provide a roadmap on how to make the formal calculations of this article rigorous, i.e.\ how to work with our notion of weak solutions. We prove the optimality of the Harnack constant on the ellipticity constants in Appendix~\ref{sec:opt}. We explain the optimality of the range of exponents for the weak Harnack inequality in Appendix~\ref{sec:optC}. And we discuss why smooth trajectories do not achieve criticality in Appendix~\ref{sec:smoothnotwork}. We provide the abstract De~Giorgi-Moser iterations and the lemma of Bombieri and Giusti in Appendix~\ref{sec:abstract}.

\subsection{Notation}

In what follows $n\in \N$ denotes the dimension. The subscripts $t$, $x$ and $v$ refer to the first, the following $n$ and the last $n$ variables, respectively. Such subscripts are used for functions, for gradients $\nabla$, for Lebesgue $\L^p$, and Sobolev $\H^s$ spaces, in a way that should always be clear from the context. We follow the standard notation for Lebesgue spaces $\L^p$ and the weak Lebesgue spaces $\L^{p,\infty}$. Given $\Omega_v \subset \R^{n}$ the space $\H^1(\Omega_v)$ with norm $\norm{\cdot }_{\H^1}^2 = \norm{\cdot}_{\L^2}^2 + \norm{\nabla \cdot}_{\L^2}^2$ denotes the Hilbert space of $\L^2$-functions weakly differentiable in velocity in an $\L^2$-sense. Moreover, $\H^{-1}(\Omega_v)$ denotes its dual. With $\Hdot^1(\Omega_v)$ we denote the homogeneous Sobolev space. Given functions $g_1,\dots,g_k$ we denote by $\W(g_1,\dots,g_k)$ their Wro\'nski matrix and by $\mathbf w = \det \W(g_1,\dots,g_k)$ their Wronskian. We use the notation $f \sim g$ as $r \to r_0$ if $\lim\limits_{r \to r_0} \frac{f(r)}{g(r)} = C$ for some constant $C \in (0,\infty)$. Whenever we write $\lesssim,\gtrsim,\approx$, we estimate or compare by a universal constant.

\section{Critical kinetic trajectories} 
\label{sec:kintraj}

Let us start by defining the key concept of this article.
\begin{definition}
  \label{def:kintraj}
  Let $(t_0,x_0,v_0)$ and $(t_1,x_1,v_1) \in \R^{1+2n}$ with $t_0 \neq t_1$.

  \noindent
  A \emph{kinetic trajectory} is a map
  \begin{equation*}
    \gamma = \gamma(r) =\gamma(r;(t_0,x_0,v_0),(t_1,x_1,v_1)) = (\gamma_t(r),\gamma_x(r),\gamma_v(r)) \in \R^{1+2n}
  \end{equation*}
  defined for $r \in [0,1]$, that is
  \begin{itemize}
  \item continuous over $r \in [0,1]$ (and in particular bounded),
  \item differentiable over $r \in (0,1)$, 
  \item with endpoints $\gamma(0) = (t_0,x_0,v_0)$ and $\gamma(1) = (t_1,x_1,v_1)$,
  \item satisfying the constraint (kinetic relation)
    $\dot{\gamma}_x(r) = \dot{\gamma}_t(r) \gamma_v(r)$ for $r \in (0,1)$.
  \end{itemize}
  
  \noindent A kinetic trajectory is called a \emph{critical kinetic trajectory} if it additionally satisfies
  \begin{equation*}
    \det \left( \nabla_{(t_1,x_1,v_1)} \gamma(r;(t_0,x_0,v_0),(t_1,x_1,v_1)) \right)   \sim \abs{ \frac{\gamma_t(r)-t_0}{t_1-t_0} }^{{2+4n}} \quad \text{ as } r \to 0^+
  \end{equation*}
  and
  \begin{equation}
    \label{eq:criticality}
    \left| \dot \gamma_t(r) \right| \left| \left( \left[\nabla_{(x_1,v_1)} \gamma_{x,v}(r;(t_0,x_0,v_0),(t_1,x_1,v_1))\right]^{-1} \right)_{\cdot;2} \right| \sim |\dot{\gamma}_v(r)|  \quad \text{ as } r \to 0^+.
  \end{equation}
  Note that in practice we take $\dot{\gamma}_t$ constant.
\end{definition}

Here, and in what follows $2n \times 2n$ matrices are written as $2 \times 2$ block matrices of $4$ blocks consisting of a scalar function multiplied by $\id_n$. Let $A \in \R^{2n \times 2n}$ be such a matrix, then $A_{i;j}$ with $i,j \in \{1,2\}$ refers to the $(i,j)$-block of size $n \times n$ and $A_{\cdot;j}$ refers to the $2n \times n$ matrix column.

\addtocontents{toc}{\SkipTocEntry}
\subsection{Construction}
\label{sec:connecting}

We extend the construction of kinetic trajectories as proposed in~\cite{niebel_kinetic_2022-1}, with the optimal control interpretation of~\cite{anceschi2024poincare}, in order to reach the critical case. The idea in these works is to consider $\dot \gamma_t$ constant and make an ansatz for the forcing $\dot \gamma_v$ with $2n$ free parameters. The trajectory $\gamma$ can then be recovered by solving Newton's laws of motion and adjusting the parameters to recover the initial and final endpoints. Our ansatz for the forcing is a linear combination of functions that only depend on the parameter $r$, so that the Newton equations can be explicitly solved, and differentiated twice in order to simplify the calculations.

Let $(t_0,x_0,v_0)$ and $(t_1,x_1,v_1) \in \R^{1+2n}$ with $t_0 \neq t_1$. Consider two functions $g_1,g_2 \in\C^1([0,1],\R)$, twice differentiable over $(0,1)$, with $g_1(0) = g_2(0) = \dot g_1(0) =  \dot g_2(0) = 0$. Let $\dot \gamma_t = m_0$ be constant for some $m_0 \in \R \setminus \{0 \}$. The initial and endpoint condition implies $m_0 = t_1-t_0$ and therefore
\begin{equation*}
  \gamma_t(r) = t_0 + r(t_1-t_0).
\end{equation*}
Finally, our ansatz for the forcing is
\begin{equation*}
  \begin{cases}
    \dot \gamma_v(r) = (t_1-t_0)^{-1} \ddot \gamma_x (r) = \ddot g_1(r)  m_1 + \ddot g_2(r)  m_2 \\[2mm]
    \gamma_v(r) = (t_1-t_0)^{-1} \dot \gamma_x(r) = \dot g_1(r)  m_1 + \dot g_2(r)  m_2 + v_0 \\[2mm]
    \gamma_x(r) = (t_1-t_0) g_1(r)  m_1 +(t_1-t_0) g_2(r)  m_2 +(t_1-t_0)rv_0 +x_0 
  \end{cases}
\end{equation*}
for vectorial parameters $m_1,m_2 \in \R^n$ to be chosen later. We define the Wronskian matrix
\begin{equation*}
  \forall \, r \in [0,1], \quad \cW(r) := \begin{pmatrix}
    g_1(r) \, \id_n & g_2(r) \, \id_n \\
    \dot g_1(r) \, \id_n & \dot g_2(r) \, \id_n 
  \end{pmatrix}
\end{equation*} 
of the two forcings, which measures their independence. To ensure that the trajectory indeed connects $(t_0,x_0,v_0)$ to $(t_1,x_1,v_1)$ we must impose the endpoint conditions
\begin{equation*}
  \begin{cases}
    (t_1-t_0) g_1(1)  m_1 +(t_1-t_0) g_2(1)  m_2 +(t_1-t_0)v_0 +x_0 =x_1 \\[2mm]
    \dot g_1(1)  m_1 + \dot g_2(1)  m_2 + v_0 =v_1,
  \end{cases}
\end{equation*}
which can be equivalently written as
\begin{equation*}
  \D_{t_1-t_0} \cW(1) \begin{pmatrix}  m_1 \\  m_2 \end{pmatrix} + \E_{t_1-t_0}(1) \begin{pmatrix} x_0 \\ v_0 \end{pmatrix} = \begin{pmatrix} x_1 \\ v_1 \end{pmatrix}, 
\end{equation*}
where we have denoted, for $\delta \in \R$ and $r \in [0,1]$,
\begin{equation} \label{eq:defDandE}
  \D_{\delta} := \begin{pmatrix}
    \delta \, \id_n & 0 \\
    0 & \id_n
  \end{pmatrix} \mbox{ and } \E_\delta(r) = \begin{pmatrix}
    \id_n & \delta r \, \id_n \\
    0 & \id_n 
  \end{pmatrix} = \exp\left(\delta r \begin{pmatrix}
 0 & \id_n \\ 0 & 0	
 \end{pmatrix} \right).
\end{equation}
With this notation, the trajectory $\gamma$ therefore writes
\begin{equation} 
  \gamma(r)
  = \begin{pmatrix}
    \gamma_t(r) \\
    \gamma_x(r) \\
    \gamma_v(r)
  \end{pmatrix}
  := \begin{pmatrix}
      t_0+ (t_1-t_0)r \\
      \D_{t_1-t_0} \W(r)\begin{pmatrix}  m_1 \\  m_2
      \end{pmatrix} 
      + \E_{t_1-t_0}(r) \begin{pmatrix} x_0 \\ v_0 \end{pmatrix} 
    \end{pmatrix}.  \nonumber
\end{equation}

If we assume that $\cW(1)$ is invertible, then we can determine the parameters $m_1$ and $m_2$ as follows
\begin{equation} \label{eq:mconnect}
  \begin{pmatrix}  m_1 \\  m_2 \end{pmatrix} = \cW(1)^{-1} \D_{t_1-t_0}^{-1} \begin{pmatrix} x_1 \\ v_1 \end{pmatrix}- \cW(1)^{-1} \D_{t_1-t_0}^{-1} \E_{t_1-t_0}(1) \begin{pmatrix} x_0 \\ v_0 \end{pmatrix}.
\end{equation}

This leads to the following final expression for the trajectory $\gamma \colon [0,1] \to \R^{1+2n}$:
\begin{equation*}
  \gamma(r)
  = \begin{pmatrix}
    \gamma_t(r) \\
    \gamma_x(r) \\
    \gamma_v(r)
  \end{pmatrix}
  := \begin{pmatrix}
      t_0+ (t_1-t_0)r \\
      \A_{t_1-t_0}(r) \begin{pmatrix} x_1 \\ v_1
      \end{pmatrix} 
      + \B_{t_1-t_0}(r) \begin{pmatrix} x_0 \\ v_0 \end{pmatrix} 
    \end{pmatrix},  \nonumber
\end{equation*}
where we have defined the two following $2n \times 2n$ matrices:
\begin{equation*}
  \begin{cases}
  \A_{t_1-t_0}(r) := \D_{t_1-t_0}\cW(r)\cW(1)^{-1}\D_{t_1-t_0}^{-1} \\[2mm]
  \B_{t_1-t_0}(r) := \E_{t_1-t_0}(r) - \D_{t_1-t_0}\cW(r)\cW(1)^{-1}\D_{t_1-t_0}^{-1} \E_{t_1-t_0}(1).
\end{cases}
\end{equation*}

We now choose the forcings. The intuition is that we want to combine the optimal scaling on $\ddot g_1(r)$ and $\ddot g_2(r)$ with the independence of the trajectories created by each forcing, which is measured by the invertibility of the Wronskian. Since the initial point of our trajectory is fixed, the Wronskian has to vanish near $r = 0$ and the dependency of the velocity in terms of the endpoint is necessarily singular at $r=0$; we want however this latter singularity to exactly match that of the forcing and this imposes $\ddot g_1(r) \sim \ddot g_2(r) \sim r^{-\frac{1}{2}}$ at zero. We would like to choose $r^{\frac{3}{2}}$ for both $g_1$ and $g_2$, but this choice does not have enough independence for $\W(1)$ to be invertible. Modifications of the form $r^{\frac{3}{2}+\epsilon_0}$ and $r^{\frac{3}{2}+\epsilon_1}$ for some $\epsilon_0,\epsilon_1 \in \R$, restore the independence required for the Wronskian to be invertible, but fail to reach criticality, as discussed in Appendix~\ref{sec:smoothnotwork}. Based on this ansatz, almost critical trajectories are achieved in \cite{anceschi2024poincare} while the method of \cite{niebel_kinetic_2022-1} allows to reach almost criticality, too.

It suggests to perturb the forcings along the \emph{imaginary} variable, i.e.\ to take forcings of the form $r^{\frac{3}{2}-i}$ and $r^{\frac{3}{2}+i}$: it is easy to verify that they satisfy the desired properties, with the obvious shortcoming that it would lead to complex-valued trajectories. Taking the real part of these forcings, we arrive at a neat solution to  our problem --- to reach the optimal scaling and to maintain independence at the same time ---, relying on \emph{desynchronised logarithmic oscillations}:
\begin{equation*}
  g_1(r) := r^{\frac{3}{2}} \cos(\log(r)) \quad \mbox{ and } \quad g_2(r) := r^{\frac{3}{2}} \sin(\log(r)).
\end{equation*}
This yields $\dot \gamma_v(r) \sim r^{-\frac{1}{2}}$ and we have
\begin{align*}
  \det \A_{t_1-t_0} (r)
  &= \det \D_{t_1-t_0} \cW(r) \cW(1)^{-1} \D_{t_1-t_0}^{-1} = \det \cW(r) \\
  &= \det \begin{pmatrix}
    r^{\frac{3}{2}} \cos(\log(r)) \, \id_n &  r^{\frac{3}{2}} \sin(\log(r)) \, \id_n \\
    r^{\frac{1}{2}} \left( \frac{3}{2} \cos(\log(r)) - \sin (\log(r)) \right) \, \id_n  & r^{\frac{1}{2}} \left( \frac{3}{2} \sin(\log(r)) + \cos (\log(r)) \right) \, \id_n \end{pmatrix} \\
  &= r^{2n}\det \begin{pmatrix}
    \cos(\log(r)) \, \id_n &  \sin(\log(r)) \, \id_n \\
    - \sin (\log(r) ) \, \id_n &  \cos (\log(r) ) \, \id_n
  \end{pmatrix} \\
  & = r^{2n} \left( \cos^2(\log(r))+\sin^2(\log(r)) \right)^n = r^{2n}.
\end{align*}
In the fourth equality, we used the multilinearity of the determinant, removed a multiple of the rows of the first block row in the rows of the second block, and employed standard formulas for the determinant of a square block matrix with blocks that are multiples of the identity. We can also compute the inverse of $\cW(r)$ by standard block matrix calculations as
\begin{equation}
  \label{eq:inverse-W}
  \begin{cases}
  \left( \cW(r)^{-1} \right)_{1;1} =
  r^{-\frac32} \left( \frac32 \sin (\log (r)) + \cos (\log (r)) \right) \id_n \\[2mm]
    \left( \cW(r)^{-1} \right)_{2;1} =  -r^{-\frac32} \left( \frac32 \cos(\log (r)) - \sin (\log (r)) \right) \id_n \\[2mm]
  \left( \cW(r)^{-1} \right)_{1;2} = - r^{-\frac12} \left( \sin (\log (r)) \right) \id_n \\[2mm]
  \left( \cW(r)^{-1} \right)_{2;2} = r^{-\frac12} \left( \cos (\log (r)) \right ) \id_n,
\end{cases}
\end{equation}
where we recall that the indices ``$1$'' and ``$2$'' listed with semi-colon label the blocks.

We summarise a list of properties of the matrices $\A_{t_1-t_0}$ and $\B_{t_1-t_0}$ and the trajectory $\gamma$ that includes the previous calculations plus others that are easy to verify: 
\begin{enumerate}[itemsep=0.2cm]
\item[\hypertarget{link:1}{\textbf{(1)}}] $\gamma$ is a kinetic trajectory in the sense of Definition \ref{def:kintraj}.
\item[\hypertarget{link:2}{\textbf{(2)}}] $\A_{t_1-t_0} \colon [0,1] \to \R^{2n \times 2n}$ satisfies
  \begin{enumerate}[topsep=0.5em,itemsep=0.2cm]
  \item[\hypertarget{link:2a}{\textbf{(a)}}] $\A_{t_1-t_0}(0) = 0$, $\A_{t_1-t_0}(1) = \id_{2n}$,
  \item[\hypertarget{link:2b}{\textbf{(b)}}] $\det \A_{t_1-t_0}(r) = r^{2n}$,
  \item[\hypertarget{link:2c}{\textbf{(c)}}] $\abs{(\A_{t_1-t_0}(r)^{-1})_{i;2}} \lesssim (1+\abs{t_1-t_0}) r^{-\frac{1}{2}}$ for $i=1,2$ and $r \in (0,1]$,
  \end{enumerate}
\item[\hypertarget{link:3}{\textbf{(3)}}] $\B_{t_1-t_0} \colon [0,1] \to \R^{2n \times 2n}$ satisfies
  \begin{enumerate}[topsep=0.5em,itemsep=0.2cm]
  \item[\hypertarget{link:3a}{\textbf{(a)}}] $\B_{t_1-t_0}(0) = \id_{2n}$, $\B_{t_1-t_0}(1) = 0$,
  \item[\hypertarget{link:3b}{\textbf{(b)}}] $\det \B_{t_1-t_0}(r) \gtrsim 1$ near $r = 0$.
  \end{enumerate}
\item[\hypertarget{link:4}{\textbf{(4)}}] The trajectory satisfies the following bounds for some universal constants 
  \begin{equation*}
  \begin{cases}
    \abs{\gamma_x(r)-x_0-r(t_1-t_0)v_0} \lesssim   \left(\abs{x_0}+\abs{x_1}\right) r^{\frac32}+ \abs{t_1-t_0} r^{\frac32} \left( \abs{v_0}+\abs{v_1} \right), \\[2mm]
    \abs{\gamma_v(r)-v_0} \lesssim  \abs{t_1-t_0}^{-1} \left( \abs{x_0}+\abs{x_1}\right)r^{\frac12}+ \left( \abs{v_0}+\abs{v_1} \right)r^{\frac12}  , \\[2mm]
    \abs{\dot \gamma_v(r)} \lesssim  \abs{t_1-t_0}^{-1}  \left(\abs{x_0}+\abs{x_1}\right)r^{-\frac{1}{2}}+  \left( \abs{v_0}+\abs{v_1} \right)r^{-\frac12}.
  \end{cases}
\end{equation*}
\end{enumerate}

The other properties are needed in the proofs of Theorem~\ref{thm:weakl1poin} and Theorem~\ref{thm:kinemb}. The property \hyperlink{link:2}{\textbf{(2)}} \hyperlink{link:2c}{\textbf{(c)}} follows from~\eqref{eq:inverse-W}. This property is crucial for the rest of the article as it ensures the criticality condition~\eqref{eq:criticality}. Compare Remark~\ref{rem:critical} for more explanations. The property \hyperlink{link:3}{\textbf{(3)}} \hyperlink{link:3b}{\textbf{(b)}} is a trivial consequence of the continuity of $\B$; it can be improved to $\det \B(r)>0$ for $r \in [0,1)$, but this is cumbersome to verify and not needed in our work. Bounds can be estimated explicitly on a small subinterval.

\begin{remark}
  We believe the best possible behaviour in \hyperlink{link:2}{\textbf{(2)}} \hyperlink{link:2c}{\textbf{(c)}}, under the assumption that $\dot{\gamma}_t$ behaves like $1$, and $\dot{\gamma}_v$ behaves like $r^{-\frac{1}{2}}$, to be $r^{-\frac{1}{2}}$, which is why we call it \emph{critical}. In the parabolic setting, the critical behaviour is attained, see~\cite{niebel_trajectorial_2022}. Using forcings made by power-type functions which behave at worst like $r^{-\frac{1}{2}}$ one can construct $\A_{t_1-t_0}(r)$ such that $\abs{(\A_{t_1-t_0}(r)^{-1} )_{\cdot; 2}} \lesssim r^{-\frac{1}{2}-\epsilon}$ for any $\epsilon>0$, see \cite{niebel_kinetic_2022-1,anceschi2024poincare}. In Appendix~\ref{sec:smoothnotwork}, we explain why the critical behaviour can never be reached in the class of \emph{suitably smooth} forcings at $r \sim 0$. By ``suitably smooth'', we mean that there exists an expansion close to zero in terms of monomials of any real-valued degree (even negative), see Appendix~\ref{sec:smoothnotwork} for precise details. In the present construction the limit $\lim_{r \to 0}r^{\frac{1}{2}} \dot{\gamma}_v(r)$ does not exist.
\end{remark}

\begin{remark} \label{rem:action}
  Let us compare our kinetic trajectories with curves obtained by minimisation of the action functional. Consider the following set of admissible curves
  \begin{align*}
    &\mathcal{C}((t_0,x_0,v_0),(t_1,x_1,v_1)) \\
    &= \left\{ \gamma \in \C^\infty([0,1];\R^{1+2n}) : \gamma(0) = (t_0,x_0,v_0), \, \gamma(1) = (t_1,x_1,v_1), \, \dot{\gamma}_x = \dot{\gamma}_t \gamma_v, \, \dot{\gamma}_t = t_1-t_0 \right\}
  \end{align*}
  as in~\cite{pascucci_harnack_2004}. Then, the minimiser of the action functional
  \begin{equation*}
    \inf_{\gamma \in \mathcal{C}} \int_0^1 \abs{\dot{\gamma}_v(r)}^2 \dx r = \frac{1}{(t_1-t_0)^2} \inf_{\gamma \in \mathcal{C}} \int_0^1 \abs{\ddot{\gamma}_x(r)}^2 \dx r
  \end{equation*}
  seems to be a good choice for a kinetic trajectory. A minimiser exists, see~\cite{pascucci_harnack_2004}, and the minimisation property readily implies that $\gamma_x$ must be a cubic polynomial in $r$. Hence,
  \begin{equation*}
    \gamma_v(r) = m_1 r^2+m_2r + v
  \end{equation*}
  for some $m_1,m_2 \in \R^n$, i.e.
  \begin{equation*}
    \gamma_x(r) = (t_1-t_0) \left(  \frac{1}{3}m_1 r^3+\frac{1}{2}m_2r^2 + vr \right)+x.
  \end{equation*}

  We can now interpret this ``energy-minimising'' trajectory from~\cite{pascucci_harnack_2004} in the framework of our control problem. It corresponds to the two forcings
  \begin{equation*}
    g_1(r) = \frac{1}{3}r^3 \quad \text{ and } \quad g_2(r) = \frac{1}{2}r^2
  \end{equation*}
  which gives a Wronskian with determinant $\frac{1}{6}r^{4n}$. It is thus invertible for $r>0$; we can invert the endpoint condition at $r=1$ to obtain $m_1,m_2$. We calculate
  \begin{equation*}
    \cW(r)^{-1}\cW(1) =  \left(
      \begin{array}{cc}
        \frac{3 r-2}{r^3} & \frac{3 (r-1)}{r^3} \\
        \frac{2-2 r}{r^2} & \frac{3-2 r}{r^2} \\
      \end{array}
    \right),
  \end{equation*}
  which is far from the critical scaling. This is a kinetic trajectory, but it neither has the Wronskian with the correct scaling nor do we obtain the property \hyperlink{link:2}{\textbf{(2)}} \hyperlink{link:2c}{\textbf{(c)}} needed to perform the integration by parts at the right scale.

  Nonetheless, this kinetic trajectory has the interesting property that the minimum is attained at
  \begin{equation*}
    \inf_{\gamma \in \mathcal{C}} \int_0^1 \abs{\dot{\gamma}_v(r)}^2 \dx r = \abs{v_1-v_0}^2+3\abs{(v_0+v_1)-2\frac{x_1-x_0}{t_1-t_0}}^2,
  \end{equation*}
  which is the exponent in the fundamental solution of the constant coefficient
  Kolmogorov equation with $\fra = \id$ multiplied by $-4(t_1-t_0)$.  This is
  noteworthy because it is due to this property that the authors of
  \cite{pascucci_harnack_2004} obtain the sharpest possible Harnack inequality
  for the Kolmogorov equation with constant coefficients, as a consequence of the
  Li-Yau inequality (note that their proof makes use of the fundamental
  solution).

  The constant coefficient Kolmogorov equation can stochastically be understood
  as the trajectory forced by \(\dx B_t\), where \(B_t\) is a Brownian motion.  The
  final distribution in \(v\) after time \(t\) is then given by \(B_t\)
  and the distribution in \(x\) is obtained by a further integration of \(B_t\) in time.  Note, however,
  that the oscillations of the Brownian motion imply that the standard deviation
  grows like \(\sqrt t\), which yields the required spreading for short times.
  Translating this back to a corresponding forcing, this yields the scaling of
  \(\dot \gamma_{v}\) as \(r^{-1/2}\).

  An increased spreading for short times can be obtained by a weighted action
  functional as
  \begin{equation*}
    \inf_{\gamma \in \mathcal{C}} \int_0^1 r^{\alpha}\abs{\dot{\gamma}_v(r)}^2 \dx r
  \end{equation*}
  with the parameter \(\alpha=1/2\).  For the standard parabolic case, this yields
  the critical trajectories.  However, in the kinetic setting, this does not
  yield sufficient spreading.
\end{remark}

\begin{remark}
  The construction of the family can also be understood on a dyadic level.
  Starting from two forcings \(\rho_1=\rho_1(r)\) and \(\rho_2=\rho_2(r)\), we see
  that after time \(r=1\) we can connect any two points if
  \begin{equation*}
  	\left(\int_0^r \rho_1(s) \dx s,\int_0^r \int_0^s \rho_1(\tau) \dx \tau \dx r\right) \quad \mbox{ and } \quad \left(\int_0^r \rho_1(s) \dx s,\int_0^r \int_0^s \rho_2(\tau) \dx \tau \dx r\right)
  \end{equation*}
  are independent.

  Adding a factor \(r^{-1/2}\) does not yield a family of critical trajectories
  as they are not ``spreading'' enough, like in the previous remark for the paths
  minimising the weighted action functional.  For sufficient spreading at
  every level, we need the different directions at every (small) timescale.

  This motivates a dyadic construction where we express the artificial time
  \(r\) as \(r=a^{-k} (1+\tau)\) with \(k \in \Z\), \(\tau \in [0,a-1)\) and some fixed $a >1$. We make
  the ansatz
  \begin{equation*}
    \ddot g_i(r) = a^{-\frac{1}{2}} \rho_i(\tau), \quad \text{for} \; i=1,2.
  \end{equation*}
  In our above construction this is expressed through the oscillations \(\sin(\log(r))\) and
  \(\cos(\log (r))\) which yield very elegant algebraic computations.  A
  construction based on the dyadic ansatz is possible, but leads to a very difficult
  algebra. 

  Compared to a stochastic forcing, we see that our construction of critical
  kinetic trajectories only keeps the noisy behaviour of the Brownian motion
  around the origin.  This corresponds to only sampling the Brownian motion at
  the discrete points \(r=e^{-n}, n \in \Z\) and smoothing the forcing in
  between.
\end{remark}

\addtocontents{toc}{\SkipTocEntry}
\subsection{The kinetic almost exponential map}
\label{sec:exponential}
In this subsection, we use the critical kinetic trajectories to construct a kinetic almost exponential map around a point that allows us to mollify a given function $f \colon \R^{1+2n} \to \R$ with optimal scaling when applying the kinetic vector fields $\partial_t +v \cdot \nabla_x$ and $\nabla_v$. The difference with the previous subSection~is that we do not prescribe the endpoint but instead let the forcing parameters $(m_0,m_1,m_2) \in \R^{1+2n}$ loose. Let us emphasise that, while the scaling properties we are after would also be satisfied by a more standard convolution with the rescaled fundamental solution, our method of mollification has the advantage of being \emph{local}.

\smallskip

Given $(t,x,v) \in \R^{1+2n}$, $m_0 \in \R \neq 0$ and $( m_1, m_2) \in \R^{2n}$ we consider $\gamma^{\mathbf m} \colon [0,\infty) \times \R^{1+2n}\times \R^{1+2n} \to \R^{1+2n}$ with $\mathbf m = (m_0,m_1,m_2)$ defined by
\begin{align}
  \nonumber
  \gamma^{\mathbf m}(r;(t,x,v))
  & := \gamma^{(m_0, m_1, m_2)}(r;(t,x,v)) \\
  \label{eq:gammam}
  & := \begin{pmatrix} 
    t+m_0 r \\
    \E_{m_0}(r)\begin{pmatrix}
      x \\
      v
    \end{pmatrix} + \D_{m_0} \cW(r)\D_{m_0}^{-1}
    \begin{pmatrix}
       m_1 \\  m_2
    \end{pmatrix}
  \end{pmatrix}
\end{align}
where $\E_{m_0}$ and $\D_{m_0}$ are defined in equation \eqref{eq:defDandE}.

\medskip

We have the following properties of $\gamma^{\mathbf m}$.
\begin{enumerate}[itemsep=0.2cm]
\item[\hypertarget{link:M1}{\textbf{(M1)}}] It is an \emph{open-ended kinetic trajectory}:  $\dot \gamma_x ^{\mathbf m} = \dot \gamma_t ^{\mathbf m} \gamma_v ^{\mathbf m} $ (note that $\dot \gamma_t^{\mathbf m} = m_0$).
\item[\hypertarget{link:M2}{\textbf{(M2)}}] We have $ \det(\mathcal{D}_{m_0} \W(r)\D_{m_0}^{-1}) = r^{2n}$ for all $r \in [0,\infty)$.
\item[\hypertarget{link:M3}{\textbf{(M3)}}] We have $\abs{((\D_{m_0} \cW(r)\D_{m_0}^{-1})^{-1})_{i;2}} \lesssim (1+\abs{{m_0}}) r^{-\frac{1}{2}}$ for $i=1,2$ and $r \in (0,\infty)$.
\item[\hypertarget{link:M4}{\textbf{(M4)}}] The following bounds hold: 
\begin{equation*}
	\begin{cases}
		\abs{\dot{\gamma}_v^{\mathbf m}(r)} \lesssim \left( \frac{\abs{m_1}}{\abs{m_0}} + \abs{m_2} \right) r^{-\frac{1}{2}}, \\
		\abs{\gamma_v^{\mathbf m}(r)-v} \lesssim \left( \frac{\abs{m_1}}{\abs{m_0}} + \abs{m_2} \right)r^{\frac{1}{2}}, \\
		\abs{\gamma_x^{\mathbf m}(r)-x-m_0vr} \lesssim \left( {\abs{m_1}} + \abs{m_0}\abs{m_2} \right)r^{\frac{3}{2}}
	\end{cases}
\end{equation*}
for all $r \in (0,\infty)$.
\end{enumerate}

\addtocontents{toc}{\SkipTocEntry}
\subsection{Kinetic mollification}
\label{sec:smoothing}
Given $m_0 = \pm 1$, and a smooth function $\chi \colon \R^{2n} \to [0,\infty)$ we introduce the corresponding kinetic mollification $S_r^{\chi}(f)$ of a function $f \colon \R^{1+2n} \to \R$ at the artifical time $r \in [0,\infty)$ as
\begin{equation} \label{eq:kinsmoothing}
\left[S_r^{\chi} (f)\right](t,x,v) := \frac{1}{c_\chi} \int_{\R^{2n}}f(\gamma^{\bold m}(r;(t,x,v)) \chi(m_1,m_2) \dx (m_1,m_2), 
\end{equation}
where $c_\chi := \int_{\R^{2n}} \chi(m_1,m_2) \dx (m_1,m_2)$. 

\begin{remark}
	We could also average over $(m_0,m_1,m_2)$ and set
	\begin{equation*}
		\left[\tilde{S}_r^{\chi} (f)\right](t,x,v) := \frac{1}{c_{\tilde{\chi}}} \int_{ \R^{1+2n}}f(\gamma^{\bold m}(r;(t,x,v)) \tilde{\chi}(m_0,m_1,m_2) \dx (m_0,m_1,m_2)
	\end{equation*}
	for a function $\tilde{\chi} = \tilde{\chi}(m_0,m_1,m_2) \in \C_c^\infty(\R^{1+2n})$. This is a generalisation of the definition in equation \eqref{eq:kinsmoothing} and includes an average in the time variable. It is not needed for our applications.
\end{remark}

Observe that the kinetic mollification commutes with the free transport $\partial_t+v\cdot \nabla_x$, i.e. $(\partial_t +v \cdot \nabla_x)S_r^{\chi} (f) = S_r^{\chi} (m_0(\partial_t +v \cdot \nabla_x)f)$. Furthermore, we provide two estimates on the kinetic mollification, which are of independent interest. We will use them to prove the kinetic Sobolev inequality in Section~\ref{sec:sobolev}.

For the proof of these lemmas, we need the following Young-type inequality (generalised Schur test). We denote by $d \in \N$ the dimension. Let $\emptyset \neq X,Y \subset \R^d$ be Borel measurable. Given a measurable function $K \colon X \times Y \to \R$ we study the operator 
\begin{equation} \label{eq:TK}
	[T_Kf](y) = \int_{X} K(x,y) f(x) \dx x. 
\end{equation}

\begin{theorem} \label{thm:young}
Suppose that $K\colon  X \times Y \rightarrow \R$ satisfies the bounds
$$
\|K(\cdot, y)\|_{\L^{r, \infty}(X)} \le A \; \text { for almost every } y \in Y,
$$
and
$$
\|K(x, \cdot)\|_{\L^{r, \infty}(Y)} \le A \; \text { for almost every } x \in X
$$
for some $1<r<\infty$ and a constant $A>0$.  Let $1<p<q<\infty$ with $\frac{1}{p}+\frac{1}{r}=\frac{1}{q}+1$, then the operator $T_K$ defined in \eqref{eq:TK} is of strong-type $(p, q)$, i.e.\ there exists a constant $C= C(d,p,q)>0$ with 
\begin{equation*}
	\norm{T_Kf}_{\L^q(Y)} \le C A \norm{f}_{\L^p(X)},
\end{equation*}
for all $f \in \L^p(X)$.  
\end{theorem}

\begin{proof}
	A proof can be found in \cite[Proposition 6.1]{tao_247A_2018}.
\end{proof}

In the next two lemmas, we use a nonnegative cutoff function $\tilde{\chi} \in \C_c^\infty(\R^{2n})$, which for $\sigma>0$ we rescale as $\chi_\sigma  = \tilde{\chi}(\sigma^{-1}\cdot)$. We do not track the dependency of the constants on the choice of $\tilde{\chi}$. 

\begin{lemma} \label{lem:SrL2Lq}
	Let $\sigma>0$, $q \in (2,\infty)$ and $\theta \in (1,\infty)$ such that 
	\begin{equation*}
		\frac{1}{2}+\frac{1}{\theta}= \frac{1}{q}+1.
	\end{equation*}
	Then, there exists a constant $C = C(n,\theta) >0$ such that for almost all $r >0$ the kinetic mollification, as defined in \eqref{eq:kinsmoothing} with $m_0 = \pm 1$, satisfies the inequality
	\begin{equation*}
		\norm{S_r^{\chi_\sigma}(f)(t,\cdot)}_{\L^{q}(\R^{2n})} \le C \sigma^{2n\left( \frac{1}{\theta}-1\right)}  r^{2n\left( \frac{1}{\theta}-1 \right) } \norm{f(t+m_0r,\cdot)}_{\L^2(\R^{2n})}
	\end{equation*}
	for all measurable $f = f(t,x,v) \colon \R \to \L^2(\R^{2n})$ and all $t \in \R$.
\end{lemma}

\begin{proof}
	Recalling the formula in \eqref{eq:gammam}, we first use the change of variables 
	\begin{equation*}
	\begin{pmatrix}
			\tilde{m}_1 \\ \tilde{ m}_2
		\end{pmatrix} =\Phi_{r,m_0,x,v}(m_1,m_2)=\gamma^{\bold m}_{x,v}(r;(t,x,v))= \E_{m_0}(r)\begin{pmatrix}
      x \\
      v
    \end{pmatrix} + \D_{m_0} \cW(r)\D_{m_0}^{-1}
    \begin{pmatrix}
       m_1 \\  m_2 
    \end{pmatrix}.
	\end{equation*}
    With this notation, we write
	\begin{align*}
		 S_r^{\chi_\sigma}(f) &= \frac{1}{c_{\chi_\sigma}}  \int_{\R^{2n}}f(\gamma^{\bold  m}(r;(t,x,v))) \chi_\sigma( m_1, m_2) \dx ( m_1, m_2)  \\
		 &=\frac{1}{c_{\chi_\sigma}}\int_{\R^{2n}}f(t+m_0r,\tilde{m}_1,\tilde{m}_2) \chi_\sigma(\Phi^{-1}_{r,x,v}(\tilde{m}_1,\tilde{m}_2)) r^{-2n} \dx (\tilde{m}_1,\tilde{m}_2).
	\end{align*}
	Our aim is to apply Theorem~\ref{thm:young} with the kernel 
	\begin{equation*}
		K(x,v, \tilde{m}_1, \tilde{m}_2 ) = \frac{1}{c_{\chi_\sigma}} \chi_\sigma(\Phi^{-1}_{r,m_0,x,v}( \tilde{m}_1,  \tilde{m}_2)) \ r^{-2n} 
	\end{equation*}
	and 
    \begin{equation*}
		\frac{1}{2}+\frac{1}{\theta}= \frac{1}{q}+1.
	\end{equation*}
	For that we observe that \hyperlink{link:M2}{\textbf{(M2)}} implies the size estimate
	\begin{equation*}
		\norm{K(x,v,\cdot)}_{\L^\theta(\R^{2n})} \approx \norm{K(\cdot,\tilde{m}_1, \tilde{m}_2)}_{\L^\theta(\R^{2n})} \lesssim_{n,\theta} \sigma^{2n\left( \frac{1}{\theta}-1\right)} r^{2n\left( \frac{1}{\theta}-1 \right) },
	\end{equation*}
    for almost all $(x,v) \in \R^{2n}$, $(\tilde m_1, \tilde m_2) \in \R^{2n}$, which yields the estimate. 
\end{proof}

\begin{lemma} \label{lem:intSrL2Lpq}
	Let $\sigma >0$, $m_0 = \pm 1$. Moreover, let $k \ge -\frac{1}{2}$, $p,q \in (2,\infty)$ and $\theta,h \in (1,\infty)$ be such that 
	\begin{equation*}
		\frac{1}{2}+\frac{1}{\theta}= \frac{1}{q}+1 \mbox{ and } \frac{1}{2}+\frac{1}{h}=\frac{1}{p}+1
	\end{equation*}
    with 
    \begin{equation} \label{eq:condk}
       k+\left(\frac{1}{\theta}-1\right)2n\ge -\frac{1}{h}.
    \end{equation}
	Then, there exists a constant $C = C(h,k,n,\theta)>0$ such that for all $\tau >0$ the inequality
	\begin{equation*}
		\norm{\int_0^\tau r^{k} S_r^{\chi_\sigma}(f) \dx r}_{\L^{p}(\R; \L^q(\R^{2n}))} \le C\sigma^{2n\left( \frac{1}{\theta}-1\right)}   \tau^{k+2n\left( \frac{1}{\theta}-1 \right)+ \frac{1}{h}}\norm{f}_{\L^2(\R^{1+2n})}
	\end{equation*}
	holds for all $f \in \L^2(\R^{1+2n})$.
\end{lemma}

\begin{proof}
	We use Minkowski's inequality and Lemma~\ref{lem:SrL2Lq} to obtain 
	\begin{align*}
		\norm{\int_0^\tau r^k [S_r^{\chi_\sigma}(f)](t,\cdot) \dx r}_{\L^{q}(\R^{2n})} \le C \int_0^\tau r^{k+2n\left( \frac{1}{\theta}-1 \right)} \sigma^{2n\left( \frac{1}{\theta}-1\right)}  \norm{f(t+m_0 r,\cdot)}_{\L^2(\R^{2n})} \dx r
	\end{align*}
	for some constant $C = C(n,\theta)>0$. For the convolution in the time variable, we use the Young inequality of Theorem~\ref{thm:young} with 
	\begin{equation*}
		\frac{1}{2}+\frac{1}{h} = \frac{1}{p}+1
	\end{equation*}
	to deduce
	\begin{align*}
		&\norm{t \mapsto \norm{(x,v) \mapsto \int_0^\tau r^k [S_r^{\chi_\sigma}(f)](t,x,v) \dx r}_{\L^{q}(\R^{2n})}}_{\L^p(\R)} \\
		&\le \sigma^{2n\left( \frac{1}{\theta}-1\right)}  \norm{\int_0^\tau r^{k+2n\left( \frac{1}{\theta}-1 \right)} \norm{f(\cdot+m_0 r,\cdot)}_{\L^2(\R^{2n})} \dx r}_{\L^p(\R)} \\
		&\le \sigma^{2n\left( \frac{1}{\theta}-1\right)}  \norm{\mathds{1}_{(0,\tau)} r^{k+2n\left( \frac{1}{\theta}-1 \right)} }_{\L^{h,\infty}(\R)} \norm{f}_{\L^2(\R^{1+2n})}.
	\end{align*}
	Note that 
	\begin{equation*}
		\norm{\mathds{1}_{(0,\tau)} r^{k+2n\left( \frac{1}{\theta}-1 \right)} }_{\L^{h,\infty}(\R)} \approx \tau^{k+2n\left( \frac{1}{\theta}-1 \right)+ \frac{1}{h}} < \infty
	\end{equation*}
	if \eqref{eq:condk} is satisfied.
\end{proof}

\section{Critical trajectories for higher-order kinetic vector fields}
\label{sec:hypo}

We now explain how our methodology can be extended to a class of hypoelliptic equations where more than one commutator is needed to satisfy H\"ormander's condition. For $k \ge 1$ and spatial dimensions $d_1,\dots,d_k \in \N$ with $1 \le d_1 \le d_2 \le \dots \le d_k$ we want to study solutions $f  = f(t,x^{1},x^{2},\dots,x^{k}) \colon \R^{1+d_1+\dots+d_k} \mapsto \R$, with $x^j \in \R^{d_j}$ for $j = 1,\dots,k$, to the equation 
\begin{equation} \label{eq:kolk}
  \left( \partial_t +  \sum_{j = 1}^{k-1} \left( \mathfrak b_{j} x^{j+1} \right) \cdot \nabla_{x^{j}} \right) f = \nabla_{x^k} \cdot \left( \mathfrak a \nabla_{x^k} f \right),
\end{equation}
where $\mathfrak a = \mathfrak a(t,x^{1},x^{2},\dots,x^{k})$ is a rough diffusion coefficient, i.e.\ measurable, bounded and elliptic, and $\mathfrak b_j$ are constant $d_{j} \times d_{j+1}$-matrices with maximal rank $d_{j}$ for $j=1,\dots,k-1$. Note that the kinetic setting studied in the previous subSection~corresponds to $k=2$ and $d_1=d_2=n$ and $\mathfrak b_2 = \id_n$, with $x_1=x$ and $x_2=v$. The parabolic setting corresponds to $k = 1$, $d_1 = n$ and $x_1 = v$. The notation adopted highlights the cascade of H\"ormander's commutators. We refer to $x^k$ as the ``diffusive'' variable. We set $D = d_1+\dots+d_k$ and 
\begin{equation} \label{eq:frab}
	\mathfrak b = \left(\begin{array}{cccccc}
0 & \mathfrak b_1 & 0 & \cdots & \cdots & 0 \\
\vdots & 0 & \mathfrak b_{2} & 0 & \cdots & 0 \\
\vdots & \vdots & 0 & \ddots & 0 & 0 \\
\vdots & \vdots & \vdots & \ddots & \mathfrak{~b}_{k-2} & 0 \\
\vdots & \vdots & \vdots & \vdots & 0 & \mathfrak{~b}_{k-1} \\
0 & 0 & \cdots & \cdots & \cdots & 0
\end{array}\right),
\end{equation}
which allows to write the left-hand side of \eqref{eq:kolk} as $\partial_t +  (\mathfrak{b} x) \cdot \nabla$. In view of this cascading structure, we call $\partial_t +  (\mathfrak{b} x) \cdot \nabla$ the higher-order kinetic vector field. This naming can also be justified from the stochastic viewpoint. The equation in \eqref{eq:kolk} describes the evolution of the joint law of a $k-1$-times integrated Wiener process. 

Changes need to be made for the underlying geometry, i.e.\ the kinetic cylinders need to be replaced by their higher-order hypoelliptic analogues, defined by the higher-order Galilean transformation along the characteristics of $\partial_t +  (\mathfrak{b} x) \cdot \nabla$, and dilations
\begin{equation*}
  \delta_r = \left(r^2,r^{1+2(k-1)} \id_{d_1},r^{1+2(k-2)} \id_{d_2},\dots,r^{1+2} \id_{d_{k-1}},r^{1}\id_{d_k}\right),
\end{equation*}
which means that the homogeneous dimension $2+4n$ needs to be replaced by 
\begin{equation} \label{eq:homdim}
  \homdim := 2+d_1(1+2(k-1))+d_2(1+2(k-2))+\dots+d_{k-1}(1+2) +d_k(1+0),
\end{equation}
which is simply $2+ k^2n$ when $d_1= \dots =d_k =n$. Recall that the homogeneous dimensions describe the scaling of the volume of cylinders with respect to this geometry. We refer to \cite{pascucci_mosers_2004,anceschi_note_2022,anceschi2024poincare} for more discussion about the underlying structure. 

\bigskip
\label{rem:changebj}
	Without loss of generality, we may assume that the matrices $\br_j$ are of the form
\begin{equation*}
	\br_j = \begin{pmatrix}
		\id_{d_{j}} & 0_{d_j \times (d_{j+1}-d_j)}
	\end{pmatrix}
\end{equation*}
for $j = 1,\dots,k-1$. Here, $0_{d_j \times (d_{j+1}-d_j)}$ denotes the $d_j \times (d_{j+1}-d_j)$ matrix with zero entries.

Indeed, consider the operator
\begin{equation*}
	 \left( \partial_t  +  \sum_{j = 1}^{k-1} \left( \mathfrak b_{j} x^{j+1} \right) \cdot \nabla_{x^j} \right) f 
\end{equation*}
evaluated at $(x^j = A_j y^j)_{j = 1,\dots,k}$ for some invertible matrices $A_j \in \R^{d_j \times d_j}$ to be fixed later, $j = 1,\dots,k$. Writing $g(t,y^1,\dots,y^k) = f(t,A_1y^1,\dots,A_ky^k)$ we obtain 
\begin{align*}
\left.	\partial_t f +  \sum_{j = 1}^{k-1} \left( \mathfrak b_{j} x^{j+1} \right) \cdot \nabla_{x^j} f \right|_{\substack{x^j = A_jy^j\\ j = 1,\dots ,k}} & = \partial_t g + \sum_{j = 1}^{k-1} \left( \br_{j}A_{j+1}y^{j+1} \right) \cdot \left( (A_{j}^{-1})^T \nabla_{y^j} g \right) \\
& = \partial_t g + \sum_{j = 1}^{k-1} \left( A_{j}^{-1} \br_{j}A_{j+1}y^{j+1} \right) \cdot  \nabla_{y^j} g.
\end{align*}

For each $j$, we set $C_j$ to be a matrix with columns consisting of a basis of $d_{j+1}-d_j$ vectors spanning $\ker \br_j$. Moreover, we choose a right inverse $\br_j^{-1}$ mapping to the orthogonal complement of $\ker \br_j$. We set $A_1 = \id_{d_1}$ and 
\begin{equation*}
 A_{2} = \begin{pmatrix}
 	\br_1^{-1} & C_1
 \end{pmatrix}
\end{equation*}
as well as
\begin{equation*}
	A_{j+1} =  \begin{pmatrix}
		\br^{-1}_{j} A_{j} & C_{j}
	\end{pmatrix}
\end{equation*}
for $j = 2,\dots,k-1$. The matrices are invertible as the columns of $\br^{-1}_{j} A_j$ span the vector space $(\ker \br_{j})^\perp$. Consequently,
\begin{equation*}
	  \br_{j}A_{j+1} = \begin{pmatrix}
		A_j & 0_{d_j \times (d_{j+1}-d_j)}
	\end{pmatrix}
\end{equation*}
and thus
\begin{equation*}
	A_{j}^{-1} \br_{j}A_{j+1} = \begin{pmatrix}
		\id_{d_{j}} & 0_{d_j \times (d_{j+1}-d_j)}
	\end{pmatrix}.
\end{equation*}

This change of variables simplifies the calculations for the construction of our critical trajectories. We mention the work \cite{pascucci_harnack_2004} where non-critical trajectories are constructed while skipping this change of coordinates, but based on the introduction of a vector space decomposition for $\R^{d_k}$.

\begin{remark}
	In terms of the geometry of the vector fields, this class is not as general as the whole hypoellipticity class of H\"ormander~\cite{hormander_hypoelliptic_1967}. 
	
	For smooth vector fields $X_0,X_1,\dots,X_m$ invariant under an underlying homogeneous Lie group and satisfying the H\"ormander condition, we consider
		\begin{equation*}
			X_0 f + \sum_{i = 1}^m X_i^*(\mathfrak{a}_{ij}X_jf)=0.
		\end{equation*}
		
	The Rothschild-Stein theorem \cite{rothschild_hypoelliptic_1976} explains how, in the general case (without an underlying homogeneous Lie group structure), to reduce locally to such a structural setting by a change of variable and an approximation procedure. 		
		
	Another restriction in our class is that we assume, in this canonical geometric setting, that all the coefficients' roughness is concentrated on the last diffusive term. It is an interesting question how much roughness can be allowed in the drift term or the vector fields themselves. An example are kinetic relations of the form $\partial_t+\mathfrak{b}(v) \cdot \nabla_x$, where the function $\mathfrak{b}(v)$ satisfies some nondegeneracy assumption, see \cite{zhu_velocity_2021}.
\end{remark}

\subsection{Construction of critical trajectories} 
$ $ \\[1em]
\textbf{The case $d_1 = \dots = d_k$.} 
As explained in Section~\ref{rem:changebj}, we may assume that the matrices are of the form $\br_j = \id_{d_j} = \id_{d_k}$, $j = 1,\dots,k-1$. 

Let $(t_0,x^{1}_0,\dots,x^{k}_0)$ and $(t_1,x^{1}_1,\dots,x^{k}_1) \in \R^{1+d_1+\dots+d_k}$ with $t_0 \neq t_1$. Consider $k$ functions $g_1,\dots,g_k \in\C^{k-1}([0,1],\R)$, $k$-times differentiable on $(0,1)$, with $g_j^{(i)}(0)=0$ for all $i = 1,\dots,k-1$ and $j=1,\dots,k$. We now construct a trajectory $\gamma = (\gamma_t,\gamma_1,\dots, \gamma_k)$ between these points. Again we impose $\dot \gamma_t = m_0$ to be constant equal to $m_0 \in \R\setminus \{0 \}$. The start and endpoint condition yields $m_0 = t_1-t_0$, i.e.\ $\gamma_t(r) = t_0 + r(t_1-t_0)$, and we solve the ``higher-order Newton's laws of motion'' for a forcing that combines linearly the $g^{(k)}_j$ for $j=1,\dots,k$: 
\begin{equation*} \displaystyle
  \begin{cases}
    \dot \gamma_k(r) = g_1^{(k)}(r)  m_1 + \dots + g_k^{(k)}(r)  m_k \\[2mm]
    \gamma_k(r) = g_1^{(k-1)}(r)  m_1 + \dots + g_k ^{(k-1)}(r)  m_k + x^k_0 \\[2mm]
    \gamma_{k-1}(r) = m_0 g_1^{(k-2)}(r)   m_1 + \dots + m_0 g_k ^{(k-2)}(r)   m_k + m_0 r x^k_0 + x^{k-1}_0 \\[2mm]
    \quad \vdots \\[2mm]
    \gamma_1(r) = m_0^{k-1} g_1(r) m_1 + \dots + m_0^{k-1} g_k (r)   m_k + \sum\limits_{j=2}^{k}\frac{(m_0 r)^{j-1} }{(j-1)!}  x^j_0 + x^1_0 
  \end{cases}
\end{equation*}
for vectorial parameters $ m_1,\dots, m_k \in \R^{d_k}$ to be chosen later. 

Observe that $\gamma$ is continuous on $r \in [0,1]$ (and in particular bounded), differentiable on $r \in (0,1)$, and it satisfies the constraints $\dot \gamma_{j-1}(r) = \dot \gamma_t(r) \gamma_j(r)$ for $r \in (0,1)$ and $j=2,\dots,k$. Moreover we have $\gamma(0) = (t_0,x_0^1,\dots,x^k_0)$ by construction. If the trajectory attains the endpoint, i.e. $\gamma(1) = (t_1,x^1_1,\dots,x^k_1)$, then we say that, in line with our previous definition in the case of only one commutator, $\gamma$ is a \emph{higher-order kinetic trajectory}. 

We define the (higher-order) Wronskian
\begin{equation*}
  \cW(r) := \cW(g_1(r),\dots,g_k(r)) = \begin{pmatrix}
    g_1^{(0)}(r) \; \id_{d_k} & g_2^{(0)} (r) \; \id_{d_k} & \dots & g_k ^{(0)}(r) \; \id_{d_k} \\
    g_1^{(1)}(r)  \; \id_{d_k} & g_2^{(1)}(r) \; \id_{d_k} & \dots & g_k^{(1)}(r) \; \id_{d_k} \\
    \vdots & \dots & \dots & \vdots \\
    g_1^{(k-1)}(r) \; \id_{d_k} & g_2^{(k-1)}(r) \; \id_{d_k} & \dots & g_k^{(k-1)}(r) \; \id_{d_k}
  \end{pmatrix}
\end{equation*} 
for all $r \in [0,1]$ of the $k$ forcings, which measures their independence. We impose the final endpoint condition
\begin{equation*}
  \D_{t_1-t_0} \cW(1) \begin{pmatrix}  m_1 \\  m_2 \\ \vdots \\  m_k \end{pmatrix} + \E_{t_1-t_0}(1) \begin{pmatrix} x_0 ^1 \\ x_0 ^2 \\ \vdots \\ x^k_0 \end{pmatrix} = \begin{pmatrix} x_1^1 \\ x^2_1 \\ \vdots \\ x^k _1 \end{pmatrix} 
\end{equation*}
with the new definitions of
\begin{equation*}
  \D_{\delta} := \begin{pmatrix}
    \delta^{k-1} \id_{d_k} & 0 & \dots & \dots & 0\\
    0 & \delta^{k-2} \id_{d_k} & \dots & \dots & 0 \\
    \vdots & \dots & \dots & \vdots & \vdots \\
    0 & \dots & 0 & \delta \id_{d_k} & 0 \\
    0 & \dots & \dots & 0 & \id_{d_k}
  \end{pmatrix}
\end{equation*}
and
\begin{equation*}
	\E_\delta(r) := \exp(\delta r \mathfrak b) =  \begin{pmatrix}
    \id_{d_k} & \frac{(\delta r)^1}{1!} \id_{d_k} & \frac{(\delta r)^{2}}{2!} \id_{d_k} & \dots & \frac{(\delta r)^{k-1}}{(k-1)!} \id_{d_k} \\
    0 & \id_{d_k} & \frac{(\delta r)^1}{1!} \id_{d_k} & \dots & \frac{(\delta r)^{k-2}}{(k-2)!} \id_{d_k} \\
    \vdots & 0 & \id_{d_k} & \dots & \vdots \\
    \vdots & \vdots & 0 & \dots & \vdots \\
    0 & \vdots & \dots & 0 & \id_{d_k}
  \end{pmatrix}
\end{equation*}
for $\delta \in \R$. 

With this notation, the trajectory is
\begin{equation*}
  \gamma(r)
  = \begin{pmatrix}
    \gamma_t(r) \\
    \gamma_{1}(r) \\
    \vdots \\
    \gamma_k(r)
  \end{pmatrix}
  := \begin{pmatrix}
      t_0+ (t_1-t_0)r \\
      \D_{t_1-t_0} \W(r)\begin{pmatrix}  m_1 \\ \vdots \\  m_k
      \end{pmatrix} 
      + \E_{t_1-t_0}(r) \begin{pmatrix} x_0^1 \\ \vdots \\ x_0^k \end{pmatrix} 
    \end{pmatrix}.
\end{equation*}

Assuming that $\cW(1)$ is invertible we obtain
\begin{equation} \label{eq:hot:m}
  \begin{pmatrix}  m_1 \\  m_2 \\ \vdots \\  m_k \end{pmatrix} = \cW(1)^{-1} \D_{t_1-t_0}^{-1} \begin{pmatrix} x_1 ^1 \\ x_1^2 \\ \vdots \\ x_1^k \end{pmatrix}- \cW(1)^{-1} \D_{t_1-t_0}^{-1} \E_{t_1-t_0}(1) \begin{pmatrix} x_0^1 \\ x^2_0 \\ \vdots \\ x_0 ^k  \end{pmatrix}.
\end{equation}

This leads to the following final expression for the trajectory $\gamma \colon [0,1] \to \R^{1+2n}$:
\begin{equation} \label{eq:kintrajAB}
  \gamma(r)
  = \begin{pmatrix}
    \gamma_t(r) \\
    \gamma_1(r) \\ 
    \vdots \\
    \gamma_k(r)
  \end{pmatrix}
  =: \begin{pmatrix}
      t_0+ (t_1-t_0)r \\
      \A_{t_1-t_0}(r) \begin{pmatrix} x_1^1 \\ \vdots \\ x^k_1
      \end{pmatrix} 
      + \B_{t_1-t_0}(r) \begin{pmatrix} x_0^1\\ \vdots \\ x^k_0 \end{pmatrix} 
    \end{pmatrix},  \nonumber
\end{equation}
where we have defined the two following $D \times D$ square matrices ($D = kd_k$):
\begin{equation*}
  \begin{cases}
  \A_{t_1-t_0}(r) := \D_{t_1-t_0}\cW(r)\cW(1)^{-1}\D_{t_1-t_0}^{-1}, \\[2mm]
  \B_{t_1-t_0}(r) := \E_{t_1-t_0}(r) - \D_{t_1-t_0}\cW(r)\cW(1)^{-1}\D_{t_1-t_0}^{-1} \E_{t_1-t_0}(1).
\end{cases}
\end{equation*}

We now choose the forcings $g_1,\dots, g_k$, following the same intuition as in the previous case with only one commutator, i.e.\ we introduce \emph{desynchronised logarithmic oscillations} with the help of trigonometric functions. We first notice that if we choose $g_i$ of the form $g_i(r) := r^{k-\frac{1}{2}} \tilde g_i(\log(r))$, with $\tilde g_i \colon \R \to \R $, for $i = 1,\dots,k$, then the Wronskian determinant is given by 
\begin{align*}
  \bold  w(r)
  &:= \det \W(r^{\frac{2k-1}{2}} \tilde g_1(\log r),\dots,r^{\frac{2k-1}{2}} \tilde g_k( \log r)) \\
  &= r^{d_k \sum\limits_{i=0}^{k-1} \left( \frac{2k-1}{2}-i\right)} \det \W(\tilde g_1(s),\dots, \tilde g_k(s))|_{s = \log r}\\
  &= r^{\frac{d_k k^2}{2}} \det \W(\tilde g_1(s),\dots,\tilde g_k(s))|_{s = \log r}.
\end{align*}
In the second line, we made use of the Fa\`a di Bruno formula for the derivatives of a composite function and properties of the determinant. In fact for any $(i,j) \in \{1,\dots,k\}^2$ the prefactor of the block matrix at entry $(i,j)$ is the sum of terms
\begin{equation*}
	r^{\frac{2k-1}{2}-(i-1)} [\partial_s^l \tilde g_j](\log r)
\end{equation*}
for $l \in \{0,\dots,i-1\}$. We can pull out the $r^{\frac{2k-1}{2}-(i-1)} $ factor by multilinearity and get rid of any other term except $[\partial^{i-1}_s g_j](\log r)$ by elementary row operations. The $d_k$ factor is due to the dimension of the block matrix.

When $k=2\ell$ is even, which includes the kinetic case $k=2$, we define 
\begin{equation*}
  \left\{ \tilde g_1(s), \dots, \tilde g_k(s) \right\} = \left\{ \cos(s), \dots, \cos(\ell s), \sin(s), \dots, \sin(\ell s) \right\}.
\end{equation*}
This yields $\dot \gamma_k(r) \sim r^{-\frac{1}{2}}$ and
\begin{equation} \label{eq:wronskianhypo}
  \bold  w(r)=r^{\frac{d_k k^2}{2}} \left(\frac{(1! \cdot 3! \cdot \dots \cdot (2\ell-1)!)^2}{\ell!}\right)^{d_k}.
\end{equation}
(see for instance \cite[A192081]{oeis} for this standard formula). When $k=2\ell+1$ is odd, we define 
\begin{equation*}
  \left\{ \tilde g_1(s), \dots, \tilde g_k(s) \right\} = \left\{ 1, \cos(s), \dots, \cos(\ell s), \sin(s), \dots, \sin(\ell s) \right\}.
\end{equation*}
Applying Laplace expansion in the first column, we obtain the same formula for the determinant as in \eqref{eq:wronskianhypo} and $\dot \gamma_k(r) \sim r^{-\frac{1}{2}}$.

We can also compute the inverse of $\cW(r)$ by block matrix calculations and find $(\cW(r)^{-1})_{i;k} = O(r^{-\frac{1}{2}})$ for all $i=1,\dots,k$. We end up with the properties
\begin{enumerate}
\item[\textbf{(1)}] $\gamma$ is a higher-order kinetic trajectory in the sense of Defintion \ref{def:kintraj} with the obvious adaptions.
\item[\textbf{(2)}] $\A_{t_1-t_0} \colon [0,1] \to \R^{D \times D}$ satisfies
  \begin{enumerate}[topsep=0.5em,itemsep=0.2cm]
  \item[\textbf{(a)}] $\A_{t_1-t_0}(0) = 0$, $\A_{t_1-t_0}(1) = \id_D$,
  \item[\textbf{(b)}] $\det \A_{t_1-t_0}(r) = C r^{\frac{\homdim-2}{2}}$ for some explicit constant $C>0$,
  \item[\textbf{(c)}] $\abs{(\A_{t_1-t_0}(r)^{-1})_{i;k}} \lesssim (1+(t_1-t_0)) r^{-\frac{1}{2}}$ for $i=1,\dots,k$,
  \end{enumerate}
\item[\textbf{(3)}] $\B_{t_1-t_0} \colon [0,1] \to \R^{D \times D}$ satisfies
  \begin{enumerate}[topsep=0.5em,itemsep=0.2cm]
  \item[\textbf{(a)}] $\B_{t_1-t_0}(0) = \id_D$, $\B_{t_1-t_0}(1) = 0$,
  \item[\textbf{(b)}] $\det \B_{t_1-t_0}(r) \gtrsim 1$ near $r \sim 0$.
  \end{enumerate}
\item[\textbf{(4)}] The coefficients $ m_i$ given by \eqref{eq:hot:m} satisfy the bound 
\begin{equation*}
	| m_i| \lesssim \sum\limits_{j=1}^{k} (t_1-t_0)^{j-k}  \left(\abs{x^j_0}+\abs{x^j_1}\right)
\end{equation*}
which implies the following bounds on the trajectory
  \begin{equation*}
  \begin{cases}
    \abs{\dot \gamma_k(r)} \lesssim r^{-\frac{1}{2}} \sum\limits_{i=1}^k (t_1-t_0)^{i-k}  \left(\abs{x^i_0}+\abs{x^i_1}\right)\\[2mm]
    \abs{\gamma_k(r)-x_0^k} \lesssim  r^{\frac{1}{2}} \sum\limits_{i=1}^k (t_1-t_0)^{i-k}  \left(\abs{x^i_0}+\abs{x^i_1}\right) \\[2mm]
    \quad \vdots \\[2mm]
    \abs{\gamma_j(r)-\sum\limits_{i=j}^{k} \frac{((t_1-t_0) r)^{i-1+(j-1)}}{(i-1+(j-1))!} x^i_0} \lesssim  r^{k-\frac{1}{2}-(j-1)} \sum\limits_{i=1} ^k (t_1-t_0)^{i-j}  \left(\abs{x^i_0}+\abs{x^i_1}\right) \\[2mm]
    \quad \vdots \\[2mm]
    \abs{\gamma_1(r)-\sum\limits_{i=1}^{k} \frac{((t_1-t_0) r)^{i-1}}{(i-1)!} x^i_0} \lesssim  r^{k-\frac{1}{2}} \sum\limits_{i=1}^k (t_1-t_0)^{i-1}  \left(\abs{x^i_0}+\abs{x^i_1}\right).
  \end{cases}
\end{equation*}
\end{enumerate}
The constants in all of the estimates are universal.

\bigskip

\noindent \textbf{The general case $d_1 \le \dots \le d_k$.}
By the change of coordinates in Section~\ref{rem:changebj}, we may assume
\begin{equation*}
	\br_j = \begin{pmatrix}
		\id_{d_{j}} & 0_{d_{j+1}-d_j}
	\end{pmatrix}
\end{equation*}
for $j = 1,\dots,k-1$.

\begin{figure}[!h] 
	\hspace*{-3cm}
	\centering
  \begin{tikzpicture}[xscale=1.8,yscale=2]
    \draw[fill=grey1!30!white,dotted]
    (0.6,-0.2) rectangle (2.55,6.2);
    \node[anchor=south] at (1.5,6.3) {\vbox{Building with\\ \(k-1\) commutators\\ of dim 2}};
    \draw[fill=grey2!30!white,dotted]
    (2.6,-0.2) rectangle (5.55,4.2);
    \node[anchor=south] at (4,4.3) {\vbox{Building with\\ \(k-3\) commutators\\ of dim 3}};
    \draw[fill=grey3!30!white,dotted]
    (5.6,-0.2) rectangle (6.55,1.2);
    \node[anchor=south] at (6.3,1.3) {\vbox{Building with\\ \(1\) commutator\\ of dim 1}};

    \node at (-0.5,0) {\(d_k=6\)};
    \node[shape=rectangle,draw] (ki1) at (1,0) {\(x_1^k\)};
    \node[shape=rectangle,draw] (ki2) at (2,0) {\(x_2^k\)};
    \node[shape=rectangle,draw] (ki3) at (3,0) {\(x^k_3\)};
    \node[shape=rectangle,draw] (ki4) at (4,0) {\(x^k_4\)};
    \node[shape=rectangle,draw] (ki5) at (5,0) {\(x^k_5\)};
    \node[shape=rectangle,draw] (ki6) at (6,0) {\(x^k_6\)};

    \node at (-0.5,1) {\(d_{k-1}=6\)};
    \node[shape=rectangle,draw] (k1i1) at (1,1) {\(x^{k-1}_1\)};
    \node[shape=rectangle,draw] (k1i2) at (2,1) {\(x^{k-1}_2\)};
    \node[shape=rectangle,draw] (k1i3) at (3,1) {\(x^{k-1}_3\)};
    \node[shape=rectangle,draw] (k1i4) at (4,1) {\(x^{k-1}_4\)};
    \node[shape=rectangle,draw] (k1i5) at (5,1) {\(x^{k-1}_5\)};
    \node[shape=rectangle,draw] (k1i6) at (6,1) {\(x^{k-1}_6\)};

    \node at (-0.5,2){\(\vdots\)};

    \node at (-0.5,3) {\(d_{4}=5\)};
    \node[shape=rectangle,draw] (3i1) at (1,3) {\(x^{4}_1\)};
    \node[shape=rectangle,draw] (3i2) at (2,3) {\(x^{4}_2\)};
    \node[shape=rectangle,draw] (3i3) at (3,3) {\(x^{4}_3\)};
    \node[shape=rectangle,draw] (3i4) at (4,3) {\(x^{4}_4\)};
    \node[shape=rectangle,draw] (3i5) at (5,3) {\(x^{4}_5\)};

    \node at (-0.5,4) {\(d_{3}=5\)};
    \node[shape=rectangle,draw] (2i1) at (1,4) {\(x^{3}_1\)};
    \node[shape=rectangle,draw] (2i2) at (2,4) {\(x^{3}_2\)};
    \node[shape=rectangle,draw] (2i3) at (3,4) {\(x^{3}_3\)};
    \node[shape=rectangle,draw] (2i4) at (4,4) {\(x^{3}_4\)};
    \node[shape=rectangle,draw] (2i5) at (5,4) {\(x^{3}_5\)};

    \node at (-0.5,5) {\(d_{2}=2\)};
    \node[shape=rectangle,draw] (1i1) at (1,5) {\(x^{2}_1\)};
    \node[shape=rectangle,draw] (1i2) at (2,5) {\(x^{2}_2\)};

    \node at (-0.5,6) {\(d_{1}=2\)};
    \node[shape=rectangle,draw] (0i1) at (1,6) {\(x^{1}_1\)};
    \node[shape=rectangle,draw] (0i2) at (2,6) {\(x^{1}_2\)};

    \draw[->] (ki1) to node[anchor=west] {\(\br_{k-1}\)} (k1i1);
    \draw[->] (ki2) to node[anchor=west] {\(\br_{k-1}\)} (k1i2);
    \draw[->] (ki3) to node[anchor=west] {\(\br_{k-1}\)} (k1i3);
    \draw[->] (ki4) to node[anchor=west] {\(\br_{k-1}\)} (k1i4);
    \draw[->] (ki5) to node[anchor=west] {\(\br_{k-1}\)} (k1i5);
    \draw[->] (ki6) to node[anchor=west] {\(\br_{k-1}\)} (k1i6);

    \draw[dotted,->] (k1i1) to (3i1);
    \draw[dotted,->] (k1i2) to (3i2);
    \draw[dotted,->] (k1i3) to (3i3);
    \draw[dotted,->] (k1i4) to (3i4);
    \draw[dotted,->] (k1i5) to (3i5);

    \draw[->] (3i1) to node[anchor=west] {\(\br_3\)} (2i1);
    \draw[->] (3i2) to node[anchor=west] {\(\br_3\)} (2i2);
    \draw[->] (3i3) to node[anchor=west] {\(\br_3\)} (2i3);
    \draw[->] (3i4) to node[anchor=west] {\(\br_3\)} (2i4);
    \draw[->] (3i5) to node[anchor=west] {\(\br_3\)} (2i5);
    
    \draw[->] (2i1) to node[anchor=west] {\(\br_2\)} (1i1);
    \draw[->] (2i2) to node[anchor=west] {\(\br_2 \)} (1i2);

    \draw[->] (1i1) to node[anchor=west] {\(\br_1\)} (0i1);
    \draw[->] (1i2) to node[anchor=west] {\(\br_1\)} (0i2);
  \end{tikzpicture}
  \caption{Illustration of the relabeling and the decomposition in several buildings in the case $d_1 =d_2= 2$, $d_3=d_4 = \dots = d_{k-2} = 5$, $d_{k-1} = d_{k} = 6$ for an arbitrary $k \ge 6$.}
  \label{fig:build}
\end{figure}
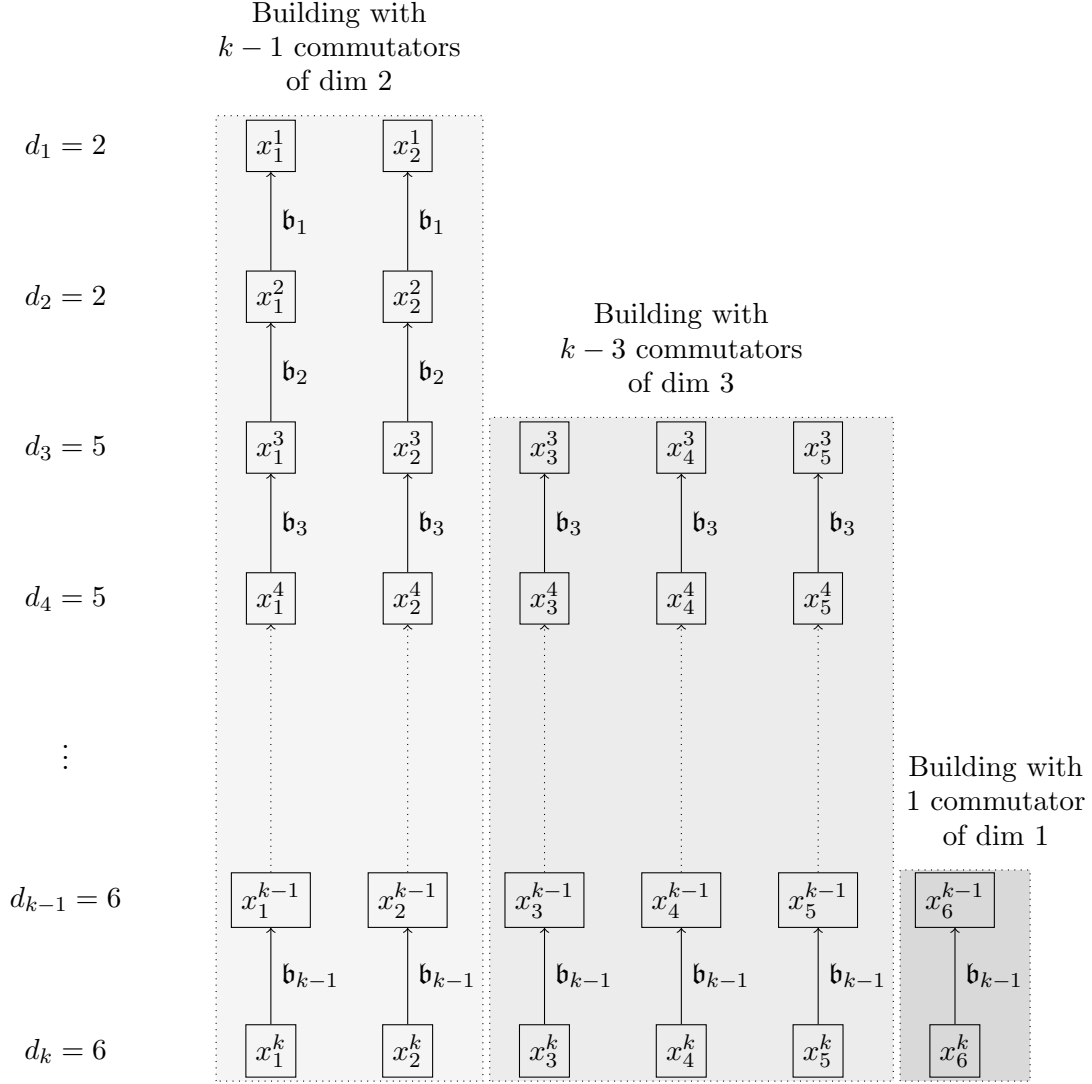

Due to the assumption $d_1 \le \dots \le d_k$ we possibly have fewer subelliptic variables than elliptic/parabolic variables. We split up the directions and relabel our variables in so-called \emph{buildings}. The first $d_1$ directions of each of the variables $x^j \in \R^{d_j}$ are needed to reach the subelliptic variables in $x^1 \in \R^{d_1}$. We call this the first building with $k-1$ commutators of dimension $d_1$. Next, there are $d_2-d_1$ directions remaining in $\R^{d_2}$ and to reach them we need $d_2-d_1$ variables in each of the $x^j \in \R^{d_j}$, $j = 2,\dots,k$ which are not used in this first cascade. Iterating this, we will have $d_{k}-d_{k-1}$ elliptic/parabolic directions remaining, i.e.\ a building where no commutators are needed. 

Based on this we introduce $y^1 \in \R^{k d_1}$, $y^j \in \R^{(k-(j-1))(d_{j}-d_{j-1})}$ for $j = 2,\dots,k$, where $y^1 = (y^1_1,\dots,y^1_k) = (x^1_{1},\dots,x^1_{d_1},\dots,x^k_{1},\dots,x^k_{d_1}) \in \R^{kd_1}$ and 
\begin{equation*}
	y^j = (x^j_{d_{j-1}+1},\dots,x^j_{d_j},\dots,x^{k}_{d_{j-1}+1},\dots,x^{k}_{d_j}) = (y^j_j,\dots,y^j_k) \in \R^{(k-(j-1))(d_j-{d}_{j-1})}
\end{equation*}
for $j = 2,\dots,k$. In total, we have 
\begin{equation*}
	\sum_{j = 2}^{k} (k-(j-1))(d_j-{d}_{j-1}) + kd_1 = d_k+\dots d_1 = D 
\end{equation*}
variables. We visualise this relabeling in Figure~\ref{fig:build}.

After relabeling the differential operator $\partial_t +(\mathfrak{b} x) \cdot \nabla$ can be rewritten as 
\begin{equation*}
	\left( \partial_t + \sum_{j = 1}^k \sum_{l = j}^{k-1} y^{j}_{l+1} \cdot \nabla_{y^j_{l}} \right) f.
\end{equation*}
In particular, the sets of variables $(y^1,\dots,y^k)$ are independent of each other. Hence, to find a trajectory $\gamma = (\gamma_t,\gamma_0,\dots,\gamma_k)$ connecting any two points we can fix $\gamma_t(r) = t_0+r(t_1-t_0)$ and find the trajectories $(\gamma_t,\gamma_j)$ for $j = 1,\dots,k$ as constructed in the previous Section~where we imposed equality of dimensions. Furthermore, similar properties to those in \textbf{(1)} -- \textbf{(3)} are satisfied. Moreover, we can also define the analogue to the kinetic mollification and verify similar estimates to those in Lemma \ref{lem:SrL2Lq} and Lemma \ref{lem:intSrL2Lpq}. Further details are left to the reader.

\section{Weak solutions and kinetic cylinders}
\label{sec:mr}

\subsection{The notion of weak solutions used}

Given $T>0$ and $\Omega_x, \Omega_v \subset \R^n$ open and bounded sets, we denote $\Omega_T := (0,T) \times \Omega_x \times \Omega_v$. We consider the Kolmogorov equation 
\begin{equation} \label{eq:kolharnack}
  (\partial_t  + v \cdot \nabla_x) f = \nabla_v \cdot (\mathfrak a \nabla_v f)+S
\end{equation}
in $\Omega_T$. Here, $\mathfrak a = \mathfrak a (t,x,v)\in \L^\infty(\Omega_T ; \R^{n \times n})$, $S = S(t,x,v) \in \L^1_{\mathrm{loc}}(\Omega_T)$ and we impose the following assumptions on the diffusion matrix $\mathfrak a$:
\begin{enumerate}
\item[\hypertarget{link:H1}{\textbf{(H1)}}] \begin{equation*}
	0< \lambda := \inf\limits_{\substack{0 \neq \xi \in \R^n  \\(t,x,v) \in \Omega_T}} \frac{\langle \fra(t,x,v) \xi , \xi \rangle}{\abs{\xi}^2}, 
\end{equation*}
\item[\hypertarget{link:H2}{\textbf{(H2)}}] \begin{equation*}
	\Lambda := \sup\limits_{\substack{0 \neq \xi \in \R^n\\(t,x,v) \in \Omega_T}} \frac{\abs{ \fra (t,x,v) \xi}^2}{\langle\fra (t,x,v) \xi, \xi \rangle}<\infty.
\end{equation*}
\end{enumerate}
It is easy to see that $\lambda \le \Lambda$. In view of the discussion in \cite{bs_nonuniform_2021}, these quantities are natural to measure the ellipticity of the diffusion coefficient.  We abbreviate $\abs{\xi}_\fra:=\sqrt{\langle \fra(t,x,v) \xi,\xi \rangle}$ and we set $\mu := \frac{1}{\lambda}+\Lambda$. If $\fra$ is assumed to be symmetric, then 
\begin{equation*}
	\Lambda = \sup\limits_{\substack{0 \neq \xi \in \R^n\\(t,x,v) \in \Omega_T}} \frac{\langle\fra (t,x,v) \xi, \xi \rangle}{\abs{\xi}^2}
\end{equation*}
so that $\lambda$ corresponds to a uniform bound on the smallest eigenvalue and $\Lambda $ to a uniform bound on the largest eigenvalue. 
	
\begin{definition} \label{def:weaksol}
  A function $f \in  \L^\infty((0,T) ;\L^2(\Omega_x \times \Omega_v))\cap \L^2((0,T) \times \Omega_x ; \Hdot^1(\Omega_v))$ is a weak (sub-, super-) solution to the Kolmogorov equation \eqref{eq:kolharnack} on $\Omega_T$ if for all $\varphi \in\C^\infty_c(\Omega_T)$ with $\varphi \ge 0$ we have
  \begin{equation} \label{eq:weaksol}
    \int_{\Omega_T} \left[ -f (\partial_t + v \cdot \nabla_x ) \varphi+  \langle \mathfrak a \nabla_v f, \nabla_v \varphi \rangle \right] \dx (t,x,v) =  \, (\le, \, \ge) \, \int_{\Omega_T} S \varphi \dx (t,x,v).
  \end{equation}
\end{definition}

\begin{remark} \label{rem:formal} 
There is a regularisation procedure which allows to reduce the rigorous derivation of estimates for weak (sub-, super-) solutions to working with the formal relation
  \begin{equation} \label{eq:weaksolf}
    \int_{\Omega_v} \varphi (\partial_t + v \cdot \nabla_x ) f +  \langle \mathfrak a \nabla_v f, \nabla_v \varphi \rangle  \dx v =  \, (\le, \, \ge) \, \int_{\Omega_v} S \varphi \dx v,
  \end{equation}
  for a.e.\ $(t,x) \in (0,T) \times \Omega_x$, and where $\varphi \in\C^\infty_c(\Omega_v)$ with $\varphi \ge 0$, see Appendix~\ref{sec:rigorous} for a rigorous justification.
\end{remark}

\begin{remark} \label{rem:weaksol}
  The regularity assumption on these weak (sub-, super-) solutions is natural and optimal since it fits the energy estimate. It is equivalent to the assumptions made in all previous works for weak \emph{solutions}, but it is weaker than them when considering sub- or supersolutions. In \cite{anceschi_mosers_2019,pascucci_mosers_2004, wang_calpha_2009, wang_calpha_2011,wang_calpha_2019,zhang_calpha_2011}, weak (sub-) supersolutions are assumed to additionally satisfy $(\partial_t +v \cdot \nabla_x) f \in \L^2(\Omega_T)$. In \cite{anceschi_note_2022,golse_harnack_2019,guerand_log-transform_2022}, they are assumed to additionally satisfy the slightly weaker (but still stronger than ours) assumption $(\partial_t  + v \cdot \nabla_x) f \in \L^2((0,T) \times \Omega_x ; \H^{-1}(\Omega_v)))$. We refer to~\cite{auscher_weak_2024} for a more detailed comparison between these regularity assumptions. In \cite{guerand_quantitative_2022}, sub- and supersolutions are assumed to additionally satisfy a renormalised formulation of \eqref{eq:weaksol}.
\end{remark}

\begin{remark}
	One may also consider lower-order terms of the form $S_1 \cdot \nabla_v f$ in \eqref{eq:kolharnack} with $S_1 \in \L^1_{\mathrm{loc}}(\Omega_T;\R^n)$. We leave the details to the reader.
\end{remark}

\subsection{Kinetic translation, scaling and cylinders}

The kinetic transport term $\partial_t + v \cdot \nabla_x$ and the derivatives in velocity $\nabla_v$ dictate the well-known geometric non-commutative structure of the Kolmogorov equation \eqref{eq:kolharnack}. Given $  (t_1,x_1,v_1), (t_2,x_2,v_2) \in \R^{1+2n}$ we write
\begin{equation*}
	(t_1,x_1,v_1) \circ (t_2,x_2,v_2) = (t_1+t_2,x_1+x_2+t_2v_1,v_1+v_2).
\end{equation*}
This is a non-commutative group structure on $\R^{1+2n}$, and every $(t,x,v)$ admits an right inverse element given by $(t,x,v)^{-1} = (-t,-x+tv,-v)$. Given $r>0$, we define the scaling
\begin{equation*}
  [\delta_r f](t,x,v) = f(r^2t,r^3x,rv).
\end{equation*}

 The Kolmogorov equation \eqref{eq:kolharnack} is \emph{structurally invariant} with respect to the translation along $\circ$ and the scaling along $\delta_r$. This means that if $f$ is a weak (sub-, super-) solution to~\eqref{eq:kolharnack} then $(t,x,v) \mapsto f((t_0,x_0,v_0)\circ (t,x,v))$ and $(t,x,v) \mapsto [\delta_r f](t,x,v)$ are also weak (sub-, super-) solutions to~\eqref{eq:kolharnack} with respective diffusion matrices $\mathfrak a ((t_0,x_0,v_0) \circ \cdot)$ and $\delta_r \mathfrak a$, which both satisfy \hyperlink{link:H1}{\textbf{(H1)}}--\hyperlink{link:H2}{\textbf{(H2)}} with the same constants $\lambda$ and $\Lambda$. 
 
 These invariances, and the fact that the regularity theory we are interested in only depends on \hyperlink{link:H1}{\textbf{(H1)}}--\hyperlink{link:H2}{\textbf{(H2)}}, naturally lead to defining \emph{kinetic cylinders} as follows. Given $(t_0,x_0,v_0) \in \R^{1+2n}$ and $\mathbf{r} = (r_1,r_2,r_3) \in (0,\infty)^3$ we define kinetic cylinders \emph{backward in time} as 
\begin{align*}
  Q_{\mathbf{r}}(t_0,x_0,v_0) & = \left\{(t,x,v) \in \R^{1+2n} : (t_0,x_0,v_0)^{-1} \circ (t,x,v) \in Q_{\mathbf{r}}(0) \right\} \\ & = \left\{ (t,x,v) \in \R^{1+2n} : -r_1^2 < t-t_0 < 0, \; \abs{x-x_0-(t-t_0)v_0} < r_2^3, \; \abs{v-v_0}<r_3 \right\},
\end{align*}
and kinetic cylinders \emph{forward in time} as
\begin{align*}
  Q_{\mathbf{r}}'(t_0,x_0,v_0) & = \left\{ (t,x,v) \in \R^{1+2n} : (t_0,x_0,v_0)^{-1} \circ (t,x,v) \in Q_{\mathbf{r}}'(0) \right\} \\ & = \left\{ (t,x,v) \in \R^{1+2n} : 0 < t-t_0 < r_1^2, \; \abs{x-x_0-(t-t_0)v_0} < r_2^3, \; \abs{v-v_0}<r_3 \right\}.
\end{align*}
Kinetic cylinders \emph{in both directions of time} are denoted by
\begin{align*}
  \bar{Q}_{\mathbf{r}}(t_0,x_0,v_0) & = \left\{(t,x,v) \in \R^{1+2n} : (t_0,x_0,v_0)^{-1} \circ (t,x,v) \in \bar{Q}_{\mathbf{r}}(0) \right\} \\ & = \left\{ (t,x,v) \in \R^{1+2n} : \abs{ t-t_0} < r_1^2, \; \abs{x-x_0-(t-t_0)v_0} < r_2^3, \; \abs{v-v_0}<r_3 \right\}.
\end{align*}
If $\mathbf{r}=(r,r,r)$ for some $r>0$, we abbreviate 
\begin{equation*}
	Q_r(t_0,x_0,v_0) = Q_{\mathbf{r}}(t_0,x_0,v_0), \; Q_r'(t_0,x_0,v_0) = Q_{\mathbf{r}}'(t_0,x_0,v_0), \; \mbox{and }\bar{Q}_r(t_0,x_0,v_0) = \bar{Q}_{\mathbf{r}}(t_0,x_0,v_0).
\end{equation*}

\section{The kinetic Sobolev inequality}
\label{sec:sobolev}

In this section, we prove the kinetic Sobolev inequality with the optimal exponent and the presumably optimal dependency on the ellipticity constants as an application of the kinetic mollification. 

\begin{theorem} \label{thm:kinemb}
	Let $0<\delta<R_1<R_2<\rho$, $(t_0,x_0,v_0) \in \R^{1+2n}$ with $Q_{R_2}(t_0,x_0,v_0)  \ssubset \Omega_T$. Let $f$ be a nonnegative weak subsolution of equation \eqref{eq:kolharnack} assuming only \hyperlink{link:H2}{\textbf{(H2)}}. Then, for ${\kappa} =  1+\frac{1}{2n}$ we have the following gain of integrability
	\begin{align*}
	&\norm{f}_{\L^{2\kappa}(Q_{R_1}(t_0,x_0,v_0))} \\&\le \frac{C}{(R_2-R_1)^{\frac{3}{2}}} \left( \sqrt{\Lambda}\norm{\abs{\nabla_v f}_\fra}_{\L^2(Q_{R_2}(t_0,x_0,v_0))} + \norm{\nabla_v f}_{\L^2(Q_{R_2}(t_0,x_0,v_0))} + \norm{f}_{\L^2(Q_{R_2}(t_0,x_0,v_0))} \right), 
	\end{align*}
	where $C = C(n,\delta,\rho)>0$. The same statement holds true for nonnegative weak supersolutions in the family of cylinders $(Q_{R_1}'(t_0,x_0,v_0))_{R_1 \in (\delta,R_2)}$ (stacked forward in time) and for nonnegative weak (sub-, super-) solutions in the cylinders $(\bar{Q}_{R_1}(t_0,x_0,v_0))_{R_1 \in (\delta,R_2)}$ (in both directions of time).
\end{theorem}

\begin{proof}
We give a detailed proof in the case that $f$ is a weak subsolution to \eqref{eq:kolharnack} in the family of cylinders $(Q_{R_1}(0))_{r \in (\delta,R_2)}$. Treating the case $(t_0,x_0,v_0) = 0$ suffices by invariance of the statement with respect to kinetic translation. 

Let $0<\delta<R_1<R_2<\rho$ and fix $\tau = 2R^2_2$. Choose any smooth cutoff function in negative direction of time $\psi \colon \R^{1+2n} \to \R$ such that $0 \le \psi \le 1$ with $\psi = 1$ in $Q_{R_1}(0) = (-R_1^2,0) \times B_{R_1^3}(0) \times B_{R_1}(0) $ and support in $Q_{R_2}(0)$, where 
\begin{equation*}
	\norm{\nabla_v \psi}_\infty \le \frac{2}{R_2-R_1} \text{ and } \norm{(\partial_t +v \cdot \nabla_x)\psi}_\infty \le \frac{2}{R^2_2-R^2_1}+\frac{2R_2}{R^3_2-R^3_1} \le  \frac{C}{R_2-R_1} 
\end{equation*}
with $C = C(\delta)>0$. 

Let $\sigma >0$, which will be set at a later point. We write $w = f \psi^2$ and let $\tilde{\chi} \in \C_c^\infty( \R^{2n})$ with $\supp \tilde{\chi} \subset B_1(0) \times B_1(0) \subset \R^{2n} $ and $0 \le \tilde{\chi} \le 1$ such that $\tilde{\chi} = 1$ on $B_{\frac{1}{2}}(0) \times B_{\frac{1}{2}}(0)$ with $\norm{\nabla \tilde{\chi}}_\infty \le 4$ and set $\chi_\sigma = \tilde{\chi}(\sigma^{-1} \cdot)$.

We use the kinetic mollification of $w$ at $(t,x,v) \in \R^{1+2n}$ defined by
\begin{equation*}
	[S_r^{\chi_\sigma}(w)](t,x,v) = \frac{1}{c_{\chi_\sigma}} \int_{\R^{2n}} w\left(\gamma^{(-1,m_1,m_2)}(r;(t,x,v))\right) \chi_\sigma^2(m_1,m_2) \dx (m_1,m_2),
\end{equation*}
in negative direction of time $(m_0 = -1)$ for $r \in [0,\infty)$. We write $\gamma^{\bold m}(r)$ for the kinetic trajectory  $\gamma^{(-1,m_1,m_2)}(r;(t,x,v))$ constructed in \eqref{eq:gammam}. 

Note that for $\tau = 2R^2_2$ we have $\gamma_t^{\bold m}(\tau) = t-\tau = t-2R^2_2$ and hence $w\left(\gamma^{\bold m}(\tau;(t,x,v)) \right)= 0$ and $[S_\tau^{\chi_\sigma}(w)](t,x,v) = 0$. 

Next, we have to bound the trajectory $\gamma^{\bold m}(r;(t,x,v))$ to stay in the cylinder $Q_{R_2}(0)$ for all times $0 \le r \le \tau$ uniformly in $(t,x,v) \in Q_{R_1}(0)$. For that, we need to make sure that $\sigma$ is of the correct size. By property \hyperlink{link:M4}{\textbf{(M4)}} of the kinetic trajectory, we can choose a small universal constant $C>0$ (depending only on the constant in \hyperlink{link:M4}{\textbf{(M4)}}) such that for $\sigma$ with
\begin{equation*}
	0< \frac{C}{2} \min \left\{ \frac{R_2-R_1}{\tau^{\frac{1}{2}}}, \frac{R_2^3-R_1^3}{\tau^{\frac{3}{2}}} \right\}<\sigma < C \max \left\{ \frac{R_2-R_1}{\tau^{\frac{1}{2}}}, \frac{R_2^3-R_1^3}{\tau^{\frac{3}{2}}} \right\} \le C,
\end{equation*}
 we have $(\gamma^{\bold m}_x(r),\gamma^{\bold m}_v(r)) \in B_{R^3_2}(0) \times B_{R_2}(0)$ for all $0 \le r\le \tau$ and any $(m_1,m_2) \in \supp {\chi_\sigma}$ uniformly in $(t,x,v) \in Q_{R_1}(0)$.

Consequently for $(t,x,v) \in Q_{R_1}(0)$ we have 
\begin{align*}
	0& \le f(t,x,v) = f(t,x,v) \psi^2(t,x,v) = w(t,x,v) = w(t,x,v) - [S_\tau^{\chi_\sigma}(w)](t,x,v) \\
	&= - \frac{1}{c_{\chi_\sigma}} \int_{\R^{2n}} (w(\gamma^{\bold m}(\tau))-w(t,x,v)) \chi_\sigma^2(m_1,m_2) \dx (m_1,m_2) \\
	& = -\frac{1}{c_{\chi_\sigma}}\int_{\R^{2n}} \int_0^\tau \frac{\dx}{\dx r} w(\gamma^{\bold m}(r))\dx r  \chi_\sigma^2(m_1,m_2) \dx (m_1,m_2) \\
	&= \frac{1}{c_{\chi_\sigma}}\int_{\R^{2n}} \int_0^\tau \Big( [(\partial_t +v \cdot \nabla_x )w](\gamma^{\bold m}(r)) \chi_\sigma^2(m_1,m_2)  \\
	&\hspace{6cm}   - \dot{\gamma}_v^{\bold m}(r) \cdot [\nabla_v w](\gamma^{\bold m}(r)) \chi_\sigma^2(m_1,m_2) \Big) \dx r\dx (m_1,m_2) \\
	&= \frac{1}{c_{\chi_\sigma}}\int_{\R^{2n}} \int_0^\tau [(\partial_t +v \cdot \nabla_x )f](\gamma^{\bold m}(r)) \psi^2(\gamma^{\bold m}(r)) \chi_\sigma^2(m_1,m_2)   \dx r\dx (m_1,m_2) \\
	&\hphantom{=}+ \frac{2}{c_{\chi_\sigma}}\int_{\R^{2n}} \int_0^\tau [(\partial_t +v \cdot \nabla_x) \psi](\gamma^{\bold m}(r)) \psi(\gamma^{\bold m}(r)) f(\gamma^{\bold m}(r)) \chi_\sigma^2(m_1,m_2) \dx r\dx (m_1,m_2)\\
	&\hphantom{=}- \frac{1}{c_{\chi_\sigma}}\int_{\R^{2n}} \int_0^\tau  \dot{\gamma}_v^{\bold m}(r) \cdot [\nabla_v f](\gamma^{\bold m}(r))  \psi^2(\gamma^{\bold m}(r)) \chi_\sigma^2(m_1,m_2)    \dx r\dx (m_1,m_2) \\
	&\hphantom{=}- \frac{2}{c_{\chi_\sigma}}\int_{\R^{2n}} \int_0^\tau  \dot{\gamma}_v^{\bold m}(r) \cdot [\nabla_v \psi](\gamma^{\bold m}(r)) \psi(\gamma^{\bold m}(r)) f(\gamma^{\bold m}(r)) \chi_\sigma^2(m_1,m_2)  \dx r\dx (m_1,m_2)\\
	&=: I_1 + I_2 + I_3 + I_4.
\end{align*}

For the first term, we use the subsolution property and then perform an integration by parts with respect to $\tilde{m}_2$ after the change of variables
\begin{equation*}
	\begin{pmatrix}
			\tilde{m}_1 \\ \tilde{ m}_2
		\end{pmatrix} =\Phi_{r,t,x,v}(m_1,m_2)=\gamma^{\bold m}_{x,v}(r;(t,x,v))= \E_{-1}(r)\begin{pmatrix}
      x \\
      v
    \end{pmatrix} + \D_{-1} \cW(r)\D_{-1}^{-1}
    \begin{pmatrix}
       m_1 \\  m_2 
    \end{pmatrix}.
	\end{equation*}

We have
\begin{align*}
	I_1 &=\frac{1}{c_{\chi_\sigma}}\int_{\R^{2n}} \int_0^\tau [(\partial_t +v \cdot \nabla_x )f](\gamma^{\bold m}(r)) \psi^2(\gamma^{\bold m}(r)) \chi_\sigma^2(m_1,m_2)   \dx r\dx (m_1,m_2) \\
	&\le  \frac{1}{ c_{\chi_\sigma}} \int_0^\tau \int_{\R^{2n}} [\nabla_v \cdot (\fra \nabla_v f)](\gamma^{\bold m}(r)) \psi^2(\gamma^{\bold m}(r))\chi_\sigma^2(m_1,m_2) \dx (m_1,m_2) \dx r\\
	&=  \frac{1}{ c_{\chi_\sigma}} \int_0^\tau \int_{\R^{2n}} [\nabla_v \cdot (\fra \nabla_v f)](t-r,\tilde{m}_1,\tilde{m}_2) \psi^2(t-r,\tilde{m}_1,\tilde{m}_2)\chi_\sigma^2(\Phi^{-1}(\tilde{m}_1,\tilde{m}_2)) r^{-2n}\dx (\tilde{m}_1,\tilde{m}_2) \dx r \\
	&=  \frac{2}{ c_{\chi_\sigma}} \int_0^\tau \int_{\R^{2n}} \langle [\fra \nabla_v f](t-r,\tilde{m}_1,\tilde{m}_2) , [\nabla_v\psi](t-r,\tilde{m}_1,\tilde{m}_2) \rangle \psi(t-r,\tilde{m}_1,\tilde{m}_2)\\
	&\hspace{8cm}\cdot \chi_\sigma^2(\Phi^{-1}(\tilde{m}_1,\tilde{m}_2)) r^{-2n}\dx (\tilde{m}_1,\tilde{m}_2) \dx r \\
	&\hphantom{=}+ \frac{2}{ c_{\chi_\sigma}} \int_0^\tau \int_{\R^{2n}} \langle [\fra \nabla_v f](t-r,\tilde{m}_1,\tilde{m}_2) , [\nabla {\chi_\sigma}]^T(\Phi^{-1}(\tilde{m}_1,\tilde{m}_2)) ((\D_{-1} \cW(r)\D_{-1}^{-1})^{-1})_{\cdot;2} \rangle \\
	&\hspace{6cm}\cdot \phi(\Phi^{-1}(\tilde{m}_1,\tilde{m}_2)) \psi^2(t-r,\tilde{m}_1,\tilde{m}_2) r^{-2n}\dx (\tilde{m}_1,\tilde{m}_2) \dx r \\
	&=  \frac{2}{ c_{\chi_\sigma}} \int_0^\tau \int_{\R^{2n}} \langle [\fra \nabla_v f](\gamma^{\bold m}(r)), [\nabla_v\psi](\gamma^{\bold m}(r)) \rangle \psi(\gamma^{\bold m}(r))\chi_\sigma^2(m_1,m_2) \dx ({m}_1,{m}_2) \dx r \\
	&\hphantom{=}+ \frac{2}{ c_{\chi_\sigma}} \int_0^\tau \int_{\R^{2n}} \langle [\fra \nabla_v f](\gamma^{\bold m}(r)) , [\nabla {\chi_\sigma}]^T(m_1,m_2) ((\D_{-1} \cW(r)\D_{-1}^{-1})^{-1})_{\cdot;2} \rangle \\
	&\hspace{8cm}\cdot \phi(m_1,m_2) \psi^2(\gamma^{\bold m}(r))\dx ({m}_1,{m}_2) \dx r. 
\end{align*}
Consequently,
\begin{equation*}
	I_1 \le \frac{2\sqrt{\Lambda}}{c_{\chi_\sigma}}\int_0^\tau\int_{\R^{2n}} \left(\frac{2}{R_2-R_1}+\frac{4C}{\sigma}r^{-\frac{1}{2}} \right) [\abs{\nabla_v f}_\fra \psi](\gamma^{\bold m}(r)) {\chi_\sigma}(m_1,m_2)  \dx (m_1,m_2) \dx r,
\end{equation*}
where we use $\abs{((\D_{-1} \cW(r)\D_{-1}^{-1})^{-1})_{\cdot;2}} \le C r^{-\frac{1}{2}}$ for some universal constant $C>0$ as in property \hyperlink{link:M3}{\textbf{(M3)}} of Section~\ref{sec:smoothing}. The derivative of $\psi$ yields a factor of order $1$ (in $r$). Furthermore, we used the assumption \hyperlink{link:H2}{\textbf{(H2)}}.

For the other integrals we estimate, 
\begin{equation*}
	I_2 \le \frac{2C}{c_{\chi_\sigma}(R_2-R_1)}\int_0^\tau \int_{\R^{2n}}  [f\psi](\gamma^{\bold m}(r)) {\chi_\sigma}(m_1,m_2) \dx (m_1,m_2)\dx r,
\end{equation*}
where $C = C(\delta)>0$ and
\begin{equation*}
	I_3 \le \frac{2C\sigma}{c_{\chi_\sigma}}\int_0^\tau \int_{\R^{2n}}    r^{-\frac{1}{2}}[\abs{\nabla_v f}\psi](\gamma^{\bold m}(r)) {\chi_\sigma}(m_1,m_2) \dx (m_1,m_2) \dx r
\end{equation*}
for some universal constant $C>0$ in $\abs{\dot{\gamma}^{\bold m}_v(r)} \le C (\abs{m_1}+\abs{m_2}) r^{-\frac{1}{2}}$, see property \hyperlink{link:M4}{\textbf{(M4)}} in Section~\ref{sec:smoothing}. Similarly,
\begin{equation*}
	I_4 \le \frac{4C}{c_{\chi_\sigma}(R_2-R_1)}\int_0^\tau \int_{\R^{2n}} r^{-\frac{1}{2}}   [f\psi](\gamma^{\bold m}(r)) {\chi_\sigma}(m_1,m_2) \dx (m_1,m_2) \dx r.
\end{equation*}

\medskip

We need to estimate functions of the form
\begin{equation*}
	(t,x,v) \mapsto \frac{1}{c_{\chi_\sigma}}\int_{0}^\tau \int_{\R^{2n}} r^k F(\gamma^{\bold m}(r)) {\chi_\sigma}(m_1,m_2)  \dx (m_1,m_2) \dx r
\end{equation*}
in $\L^{2\kappa}(\R^{1+2n})$ for some nonnegative function $F \in \L^{2}(\R^{1+2n})$ and $k = 0$ or $k = -\frac{1}{2}$. This is an integrated form of the kinetic mollification. We use the estimate of Lemma~\ref{lem:intSrL2Lpq} with $ p = q = 2\kappa = 2\left( 1+ \frac{1}{2n} \right)$, $h= \theta = \frac{2+4n}{1 + 4 n}$ and the condition on $k$ is $k \ge -\frac{1}{2}$ so that we obtain the desired estimate in both cases. 

Recalling that each of the two terms for $f,\nabla f$ comes with a localising factor of $\psi$, and tracking the dependency on $\sigma$, we conclude
\begin{align*}
	&\norm{f}_{\L^{2\kappa}(Q_{R_1}(0))} \\
	&\le C(n,\rho) \sigma^{-\frac{n}{2n+1}} \left[ 2\sqrt{\Lambda} \left( \frac{2}{R_2-R_1} +\frac{4}{\sigma} \right) \norm{\abs{\nabla_v f}_\fra}_{\L^2(Q_{R_2}(0))}+2C(\delta)\sigma \norm{\nabla_v f}_{\L^2(Q_{R_2}(0))}\right]  \\
	&\hphantom{=} +\frac{C(n,\rho)}{(R_2-R_1)}\sigma^{-\frac{n}{2n+1}} \left(4+2C(\delta) \right) \norm{f}_{\L^2(Q_{R_2}(0))} \\
	&\le  \frac{C(\delta,n,\rho)}{(R_2-R_1)^{\frac{3}{2}}} \left( \sqrt{\Lambda} \norm{\abs{\nabla_v f}_\fra}_{\L^2(Q_{R_2}(0))}+\norm{{\nabla_v f}}_{\L^2(Q_{R_2}(0))} + \norm{f}_{\L^2(Q_{R_2}(0))} \right),
\end{align*}
where we use that $1+\frac{n}{2n+1}  \le \frac{3}{2}$. If $f$ is a weak supersolution, we want the statement in cylinders of the form $(Q_{R_1}'(t_0,x_0,v_0))_{R_1 \in (\delta,R_2)}$. Here, we reverse time, and this case follows from the already proven statement. The statement in cylinders $(\bar{Q}_{R_1}(t_0,x_0,v_0))_{R_1 \in (\delta,R_2)}$ follows similarly.
\end{proof}

Note that we only used the nonnegativity when dealing with (sub-) supersolutions. We may prove a Sobolev inequality on the full spaces as follows. The same argument without the localisation but for $f$ compactly supported allows us to prove the estimate on the full space without the $\L^2$-norm on the right-hand side. Approximating suitably, we obtain the following kinetic Sobolev inequality.
 
\begin{theorem} \label{thm:sobolev}
	Let $f \in  \L^2(\R^{1+n};\H^1(\R^n))$ such that $(\partial_t+v \cdot \nabla_x) f = \nabla_v \cdot S$ for some $S \in \L^2(\R^{1+2n};\R^n)$, then 
	\begin{equation*}
		\norm{f}_{\L^{2\kappa}(\R^{1+2n})} \le C \left(\norm{\nabla_v f}_{\L^{2}(\R^{1+2n})} +\norm{S}_{\L^{2}(\R^{1+2n})}  \right)
	\end{equation*}
	with $\kappa = 1+\frac{1}{2n}$ and $C = C(n)>0$.
\end{theorem}

Writing $\L^2(\R^{1+2n})$ as an interpolation space betweeen $\L^1(\R^{1+2n})$ and $\L^{2\kappa}(\R^{1+2n})$ together with the Sobolev inequality of Theorem~\ref{thm:sobolev} yields the following kinetic Nash inequality.
 
\begin{theorem} \label{thm:nash}
    Let $f \in \L^2(\R^{1+n};\H^1(\R^n))$ such that $(\partial_t +v \cdot \nabla_x) f = \nabla_v \cdot S$ for some $S \in \L^2(\R^{1+2n};\R^n)$, then 
    \begin{equation*}
		\norm{f}_{\L^2(\R^{1+2n})}^{1+\frac{2}{2+4n}} \le  C\sqrt{\norm{\nabla_v f}_{\L^2(\R^{1+2n})}^2+\norm{S}_{\L^2(\R^{1+2n})}^2} \norm{f}_{\L^1(\R^{1+2n})}^{\frac{2}{2+4n}}
	\end{equation*}
 for some $C = C(n)>0$.
\end{theorem}

\begin{remark}
	Our proof works for $\L^p$ spaces with $p \in (1,\infty)$, too. Then, in the kinetic Sobolev inequalities of Theorem \ref{thm:kinemb} and \ref{thm:sobolev}, the gain of integrability $2\kappa$ needs to be replaced by $\frac{p}{Q-p}$. Note that for the proof one needs to make adaptations to Lemma \ref{lem:SrL2Lq} and Lemma \ref{lem:intSrL2Lpq}. We leave further details to the reader. 
	
	We propose, as an interesting research problem, the study of kinetic logarithmic Sobolev inequalities using the method presented here. We would like to highlight a very elegant argument due to Stroock \cite[Lecture II]{MR1292280} (see also \cite[Lemma 1]{df_exponential_2014}), which shows how the Poincar\'e inequality, combined with the Sobolev inequality, implies a logarithmic Sobolev inequality. Disregarding certain technical refinements, this approach makes it possible to establish a kinetic Sobolev inequality.
\end{remark}

\begin{remark}
	We emphasise that the constants in the estimate in Theorem~\ref{thm:sobolev} do not depend on $\lambda$, which shows that the theorem also holds for $\fra = 0$ or even negative $\fra$. At further inspection of the proof, one notices that $\nabla_v \cdot (\fra \nabla_v \cdot)$ can be replaced by $\nabla_v \cdot S$ for some function $S \in \L^2$. The kinetic Sobolev inequality is indeed not solely tailored to the Kolmogorov equation but to kinetic equations with local diffusion in general. Moreover, Theorem~\ref{thm:kinemb} extends to the case with source term $S_0 \in \L^q(\Omega_T)$ or $S_1 \cdot \nabla_v f$ with $S_1 \in \L^q(\Omega_T;\R^n)$ for some ranges of $q$ (in particular the ranges $q>2n+1$ and $q>4n+2$ considered in the De~Giorgi-Moser iteration are included). Note that one needs to write Lemma~\ref{lem:SrL2Lq} and Lemma~\ref{lem:intSrL2Lpq}, replacing the $\L^2$-norm on the right-hand side of the inequality. For simplicity of the presentation, we only consider the case $S_0 = 0$ and $S_1 = 0$. 
\end{remark}

\begin{remark}
	The exponent $\frac{3}{2}$ in the Sobolev inequality of Theorem~\ref{thm:kinemb} is most certainly not optimal. We emphasise that the value of the exponent does not play any role in our application to obtain the Harnack inequality. The exponent obtained in \cite{pascucci_mosers_2004} is equal to $1$. Following the argumentation in \cite[Remark 4.1 and Lemma 4.2]{pascucci_mosers_2004}, we may recover this scaling in Theorem~\ref{thm:kinemb}.
\end{remark}

\begin{remark}
	The precise dependency of the right-hand side in the Sobolev inequality of Theorem~\ref{thm:sobolev} on the diffusion constant and matrix can also be obtained by carefully rewriting the proof of \cite{pascucci_mosers_2004}.
\end{remark}

\begin{remark} \label{rem:cutoff}
	We briefly describe how to construct cutoff functions with respect to kinetic cylinders. Let $(t_0,x_0,v_0) \in \R^{1+2n}$, $\sigma,\sigma' \in (0,\infty)^3$, $\sigma_i< \sigma'_i $, $i = 1,2,3$. Choose a function $\tilde{\psi} \in\C^\infty_c(\R^{1+2n})$ such that 
\begin{align*}
	\tilde{\psi}(t,x,v) &= 1 \mbox{ if } (t,x,v) \in (-\sigma^2_1,0]\times B_{\sigma^3_2}(0) \times B_{\sigma_3}(0), \\
	\tilde{\psi}(t,x,v) &= 0  \mbox{ for } (t,x,v)  \in ((-\sigma'^2_1,\sigma'^2_1-\sigma^2_1]\times B_{\sigma'^3_2}(0) \times B_{\sigma'_3}(0))^c,
\end{align*}
with 
\begin{equation*}
	 \abs{\partial_t \tilde{\psi}(t,x,v) }  \le \frac{2}{\sigma'^2_1-\sigma^2_1} , \,  \abs{\nabla_x \tilde{\psi}(t,x,v) }  \le \frac{2}{\sigma'^3_2-\sigma^3_2} \mbox{ and } \abs{\nabla_v \tilde{\psi}(t,x,v) } \le \frac{2}{\sigma'_3-\sigma_3}
\end{equation*}
	  for all $(t,x,v) \in \R^{1+2n}$. Consider now $\psi \in\C^\infty(\R^{1+2n})$ defined as
\begin{equation*}
	\psi(t,x,v) = \tilde{\psi}(t-t_0,x-x_0-(t-t_0)v_0,v-v_0),
\end{equation*}
then $\psi = 1$ in $Q_{\sigma}(t_0,x_0,v_0)$ and $\psi = 0$ in $(Q_{\sigma'}(t_0,x_0,v_0))^c$. Moreover, 
\begin{equation*}
	\norm{\nabla_v \psi}_\infty \le \frac{2}{\sigma_3'-\sigma_3} \text{ and } \norm{ (\partial_t +v \cdot \nabla_x)\psi}_\infty \le \frac{2}{\sigma'^2_1-\sigma^2_1}+\frac{2(\sigma'_3+\abs{v_0})}{\sigma'^3_2-\sigma^3_2}. 
\end{equation*}
Cutoff functions for $Q_{r}'(t_0,x_0,v_0)$, $\bar{Q}_r(t_0,x_0,v_0)$ and in only one direction of time can be constructed similarly. 
\end{remark}

\section{An estimate for the logarithm of supersolutions to the Kolmogorov equation}
\label{sec:weakL1log}

\begin{theorem} \label{thm:weakl1poin}
Let $\delta,\eta \in (0,1)$, $\tau >0$, $\iota < \min \{ \eta,1- \eta \}$. There exists a universal constant $R = R(\iota,\tau)>0$ such that the following holds. Let $r>0$ and $(t_0,x_0,v_0) \in \Omega_T $ such that $\tilde{Q} = Q_{(\sqrt{\tau}\,r,Rr,R r)}'(t_0,x_0,v_0) \ssubset \Omega_T$, then for any weak supersolution $f >0$ of the Kolmogorov equation \eqref{eq:kolharnack} in $\tilde{Q}$, there is a constant $c = c(f) =  c(f,\delta,\eta,\tau)$ such that 
\begin{equation*}
	\int_{K_-} \Big(\log f(t,x,v)-c(f)\Big)_+ \dx (t,x,v) \le C \abs{K_-} \mu, \; s>0
\end{equation*}	
and 
\begin{equation*}
	\int_{K_+} \Big(c(f)-\log f(t,x,v)\Big)_+ \dx (t,x,v) \le C \abs{K_+} \mu, \; s>0,
\end{equation*}	
where 
\begin{align*}
	K_- &= Q_{(\sqrt{\tau(\eta-\iota)} \,r,\delta^{\frac{1}{3}} r,\delta r)}'(t_0,x_0,v_0), \\
	K_+ &= Q_{(\sqrt{\tau(1-\eta-\iota)} \, r,\delta^{\frac{1}{3}} r,\delta r)}(t_0+\tau r^2,x_0,v_0)
\end{align*}
and $C = C(\delta,\eta,\iota,n,\tau)>0$.
\end{theorem}

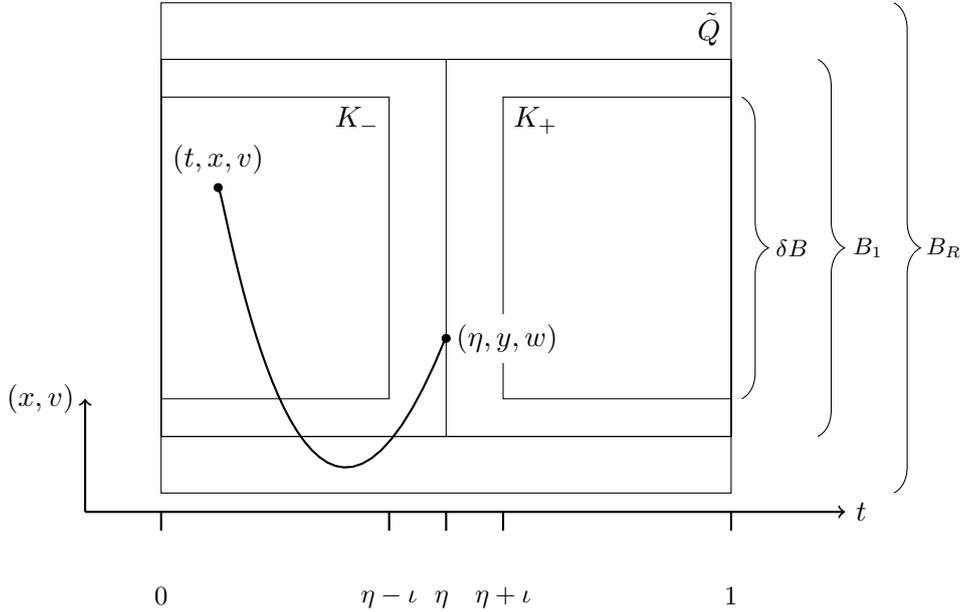
\begin{figure}[H]
\centering
\tikzmath{\x = 0.2; \y = 0.2;}
\begin{tikzpicture}[scale=5]
  \draw (1.15-\x,0.4-\y) -- (1.15-\x,1.4-\y);
  \draw[thick,->] (0, 0) -- (2,0) node[right] {$t$};
  \draw[thick,->] (0, 0) -- (0,0.3) node[left] {$(x,v)$};
  \draw[draw=black] (0.4-\x,0.25-\y) rectangle ++(1.5,1.3);
  \draw[draw=black] (0.4-\x,0.4-\y) rectangle ++(1.5,1);
  \draw[draw=black] (0.4-\x,0.4-\y) rectangle ++(1.5,1);
  \filldraw (0.55-\x,1.06-\y) circle[radius=0.3pt];
  \draw (0.55-\x,1.07-\y) node[anchor= south] {$(t,x,v)$};
  \draw[draw=black] (0.4-\x,0.5-\y) rectangle ++(0.6,0.8);
  \draw[draw=black] (1.3-\x,0.5-\y) rectangle ++(0.6,0.8);
  \draw (1.0-\x,1.3-\y) node[anchor=north east] {$K_-$};
  \draw (1.47-\x,1.3-\y) node[anchor=north east] {$K_+$}; 
  \draw (1.9-\x,1.55-\y) node[anchor=north east] {$\tilde{Q}$};
  \draw [thick] (0.4-\x, 0) -- ++(0, -.05) ++(0, -.15) node [below, outer sep=0pt, inner sep=0pt] {\small\(0\)};
  \draw [thick] (1.3-\x, 0) -- ++(0, -.05) ++(0, -.15) node [below, outer sep=0pt, inner sep=0pt] {\small\(\eta+\iota \)};
  \draw [thick] (1.15-\x, 0) -- ++(0, -.05) ++(0, -.15) node [below, outer sep=0pt, inner sep=0pt] {\small\(\eta \) \vphantom{\small\(\eta+\iota \)}};
  \draw [thick] (1.0-\x, 0) -- ++(0, -.05) ++(0, -.15) node [below, outer sep=0pt, inner sep=0pt] {\small\(\eta-\iota \)};
  \draw [thick] (1.9-\x, 0) -- ++(0, -.05) ++(0, -.15) node [below, outer sep=0pt, inner sep=0pt] {\small\(1 \)};
	\draw [decorate,decoration={brace,amplitude=10pt,mirror,raise=4pt},yshift=0pt] (1.9-\x,0.5-\y) -- (1.9-\x,1.3-\y) node [black,midway,xshift=0.8cm] {\footnotesize $\delta{B}$};
	\draw [decorate,decoration={brace,amplitude=10pt,mirror,raise=4pt},yshift=0pt] (2.3-\x,0.25-\y) -- (2.3-\x,1.55-\y) node [black,midway,xshift=0.8cm] {\footnotesize $B_R$};
	\draw [decorate,decoration={brace,amplitude=10pt,mirror,raise=4pt},yshift=0pt] (2.1-\x,0.4-\y) -- (2.1-\x,1.4-\y) node [black,midway,xshift=0.8cm] {\footnotesize ${B}_{1}$};
	\draw [thick,  domain=0:1, samples=40] plot ({0.35+\x^2*(0.95-0.35)},{0.85 - (4.5*(-\x + 2*\x^2) - 3.8*(-\x + 2*\x^3))*5.5/10});
  	\draw (0.95,0.46) node[anchor= west,fill = white] {$(\eta,y,w)$};
  	\filldraw (0.95,0.46) circle[radius=0.3pt];
\end{tikzpicture}

\caption{The cylinders $K_-,K_+, \tilde{Q}$ and the spatial balls $\delta B,B_1,B_R$ in the proof of Theorem~\ref{thm:weakl1poin} for $\tau = 1$, $(t_0,x_0,v_0) = (0,0,0)$ and $r = 1$. The curve depicts the kinetic trajectory connecting $(t,x,v)$ with $(\eta,y,w)$.}
\end{figure}

\begin{remark} \label{rem:L1log}
	\begin{enumerate}
	\item  The constant $c(f)$ will be chosen as a weighted mean of $\log f$ in the position and velocity variables over balls $B_{r}((x_0,v_0))$ at time $t_0+\eta \tau r^2$.
	\item Strictly speaking we prove an $\L^1$-estimate, which trivially implies the weak $\L^1$-estimate. This improvement in integrability, compared to Moser's original work \cite{moser_pointwise_1971}, comes from the gap in time between $K_-$ and $K_+$, which is not present in Moser's work but is not an obstacle to deduce the (weak) Harnack inequality, see also \cite{niebel_trajectorial_2022}. In the notation of Theorem \ref{thm:weakl1poin}, we obtain the following estimates 
	\begin{equation} \label{eq:weakL1logA}
s |\{ (t,x,v) \in K_- \colon \log f-c(f)>s \}| \le  C \abs{K_-} \mu, \; s>0
\end{equation}	
and 
\begin{equation} \label{eq:weakL1logB}
	s |\{ (t,x,v) \in K_+ \colon c(f)-\log f>s \}| \le C \abs{K_+} \mu, \; s>0
\end{equation}
as a consequence of the Chebyshev inequality.

It is an interesting research question to understand whether the weak $\L^1$-estimate without the time gap can be proven in the kinetic setting or by a trajectorial argument.
	\item One could call this result an $\L^1$-Poincar\'e inequality for the logarithm of positive supersolutions, too. There is no gradient term on the right-hand side due to the quadratic term appearing in the equation for the logarithm, which allows us to absorb any appearing gradient term of order less than two, see Section~\ref{sec:mbg} for more explanations. 
	\item The constant $R = R(\iota,\tau)$ needs to be thought of as an indication of how much space (in $x$ and $v$) the trajectories need in order to connect two given points. It can be estimated precisely in terms of the gap in time of size $2 \tau \iota $, see Section~\ref{sec:connecting}. Details are left to the reader. 
	\item The assumption that the solution is satisfied in a larger cylinder is merely qualitative. Values of $f$ outside of $Q_{(\sqrt{\tau} r,r,r)}(t_0,x_0,v_0)$ do not appear in the estimate.
	\end{enumerate}
\end{remark}

\begin{proof}
By scaling and transformation invariance, we can reduce to the case $r = 1$ and $(t_0,x_0,v_0) = (0,0,0)$. For simplicity of the presentation, we only give a proof in the case $\tau = 1$. Moreover, we may assume that $f \ge \epsilon $ for some $\epsilon>0$ and let $\epsilon \to 0^+$ in the end. 

Let $\delta, \eta \in (0,1)$. In this case $Q_1'(0) = (0,1) \times B_1(0) \times B_1(0)$. Set $B = B_1(0) \times B_1(0)$ and $\delta B = B_\delta(0) \times B_\delta(0)$. Then $K_- = (0,\eta-\iota) \times \delta B$ and $K_+ = (\eta+\iota,1) \times \delta B$. 

In the following, we use formal calculations for which we think of $f$ as a smooth function, see Appendix~\ref{sec:rigorous} for more details on how to make the argument rigorous. 

Let us define the mean of $\log f$, i.e.\ the constant $c(f)$, with which we want to compare with. To this end let $\varphi \in \C_c^\infty(\R^{2n})$ be such that $\varphi = 1$ in $\delta B$ and $\varphi = 0$ in $B^c$. We define $c_\varphi = \int_{B}\varphi^2(y,w) \dx (y,w)$. Moreover, we may assume $\norm{\nabla \varphi}_\infty \le \frac{2}{1-\delta}$. We set \begin{equation*}
	c(f) = \frac{1}{c_\varphi}\int_{B} \log f(\eta,y,w) \varphi^2(y,w) \dx (y,w).
\end{equation*}

As $f$ is a positive supersolution to \eqref{eq:kolharnack}, we deduce that $g = \log f$ is a supersolution to the Kolmogorov equation with a quadratic nonlinearity
\begin{equation*}
	(\partial_t + v \cdot \nabla_x) \log f  \ge  \nabla_v \cdot (\fra  \nabla_v \log f) + \langle \fra  \nabla_v \log f , \nabla_v \log f\rangle.
\end{equation*}
Our goal is to use this quadratic gradient term (which comes with a minus sign in our estimate) to absorb any appearing gradient term.

Let $(t,x,v) \in K_-$. We aim to prove an $\L^1$ estimate on the positive part of $\log f(t,x,v) - c(f)$. Given $(\eta,y,w)$ with $(y,w) \in B$ we choose a critical kinetic trajectory $\gamma$ connecting $(t,x,v)$ with $(\eta,y,w)$ as constructed in Section~\ref{sec:connecting}. We choose $R = R(\iota)>0$ accordingly such that $(\gamma_x(r),\gamma_v(r)) \in B_R(0) \subset \R^{2n}$ for all $r \in [0,1]$ and any choice of $(t,x,v) \in K_-$ and $(y,w) \in B$, see property \hyperlink{link:3}{\textbf{(4)}} in Section~\ref{sec:connecting}.

We deduce
\begin{align*}
	I&:=\log f(t,x,v) - c(f)\\
	&= \frac{1}{c_\varphi} \int_{B} \left( \log f(t,x,v) - \log f(\eta,y,w)) \right) \varphi^2(y,w) \dx (y,w) \\
	&= - \frac{1}{c_\varphi} \int_{B} \int_0^{1}\frac{\dx}{\dx r}\log f(\gamma(r))\dx r \ \varphi^2(y,w) \dx (y,w) \\
	&= - \frac{1}{c_\varphi} \int_{B}\int_0^{1} \Big( \dot{\gamma}_t(r) [(\partial_t +v\cdot \nabla_x) \log f](\gamma(r)) + \dot{\gamma}_v(r) \cdot [\nabla_v \log f](\gamma(r)) \dx r \ \varphi^2(y,w) \Big) \dx (y,w)\\
	&\le  - \frac{\eta-t}{c_\varphi} \int_{B}\int_0^{1}\Big( [\nabla_v \cdot (\fra \nabla_v \log f)](\gamma(r))  + \langle \fra \nabla_v \log f, \nabla_v \log f \rangle(\gamma(r))  \dx r \ \varphi^2(y,w) \Big) \dx (y,w)\\
&	\hphantom{\le}\,  - \frac{1}{c_\varphi} \int_{B}\int_0^{1} \dot{\gamma}_v(r) \cdot [\nabla_v \log f](\gamma(r)) \dx r \ \varphi^2(y,w) \dx (y,w) 
\end{align*}
after employing the supersolution property. 

Next, we use the change of variables 
\begin{equation*}
	\begin{pmatrix}
		\tilde{y} \\\tilde{w}
	\end{pmatrix} = \Phi(y,w) = \Phi_{r,t,x,v,\eta}(y,w) = \begin{pmatrix}
		\gamma_x(r) \\
		\gamma_v(r)
	\end{pmatrix}  = \A_{\eta-t}(r) \begin{pmatrix}
		y \\ w
	\end{pmatrix} + b(r,t,x,v,\eta),
\end{equation*}
where $b$ is some $\R^{2n}$-valued function with the noted dependencies, to perform the integration by parts of the first term in the above inequality with respect to $\tilde{w}$. Let $r \in (0,1)$ be fixed, then 
\begin{align*}
	&-\int_{B} [\nabla_v \cdot (\fra \nabla_v \log f)](\gamma(r)) \varphi^2(y,w) \dx (y,w) \\
	&=-\int_{\R^{2n}} [\nabla_v \cdot (\fra \nabla_v \log f)](\gamma(r)) \varphi^2(y,w) \dx (y,w) \\
	&= -\int_{\R^{2n}} [\nabla_v \cdot (\fra \nabla_v \log f)](\gamma_t(r),\tilde{y},\tilde{w})    \varphi^2(\Phi^{-1}(\tilde{y},\tilde{w})) \abs{\det{\A_{\eta-t}(r)}}^{-1} \dx (\tilde{y},\tilde{w})  \\
	&=  \int_{\R^{2n}} \left\langle [\fra \nabla_v \log f](\gamma_t(r),\tilde{y},\tilde{w}), \nabla_{\tilde{w}}[ \varphi^2(\Phi^{-1}(\tilde{y},\tilde{w}))] \right\rangle r^{-2n} \dx (\tilde{y},\tilde{w}) \\
	&=  2\int_{\R^{2n}} \left\langle[\fra \nabla_v \log f](\gamma_t(r),\tilde{y},\tilde{w}),[\nabla \varphi]^T(\Phi^{-1}(\tilde{y},\tilde{w}))(\A_{\eta-t}(r)^{-1})_{\cdot; 2}\right\rangle \varphi(\Phi^{-1}(\tilde{y},\tilde{w}))  r^{-2n} \dx (\tilde{y},\tilde{w}) \\
	&=  2\int_{\R^{2n}} \left\langle[\fra \nabla_v \log f](\gamma(r)),[\nabla \varphi]^T(y,w)(\A_{\eta-t}(r)^{-1})_{\cdot; 2}\right\rangle \varphi(y,w) \dx (y,w) \\
	&\le 2c_{1d} r^{-\frac{1}{2}}\sqrt{\Lambda} \norm{\nabla \varphi}_\infty \int_{B} \abs{\nabla_v \log f}_\fra(\gamma(r))  \varphi(y,w) \dx (y,w),
\end{align*}
where $\abs{\cdot}_\fra^2 = \langle \fra \cdot, \cdot \rangle$ and $c_{2c}$ is such that $\abs{(\A_{\eta-t}(r)^{-1})_{\cdot; 2}} \le c_{2c}r^{-\frac{1}{2}}$ for all $r \in (0,1]$, see Section~\ref{sec:connecting} property \hyperlink{link:2}{\textbf{(2)}} \hyperlink{link:2c}{\textbf{(c)}}. Moreover, we used property \hyperlink{link:2}{\textbf{(2)}} \hyperlink{link:2b}{\textbf{(b)}}. Note that $\Phi(B)$ is contained in a compact set by property \hyperlink{link:4}{\textbf{(4)}} so that there is no boundary term in the integration by parts. 

From here, we obtain for the positive part
\begin{align*}
	I_+ & \le  \frac{\eta}{2c_\varphi} \int_0^1 \int_{B} \left(4c_{1d}  \norm{\nabla \varphi}_\infty \sqrt{\Lambda} r^{-\frac{1}{2}}\abs{\nabla_v \log f}_\fra(\gamma(r)) \varphi(y,w)-   \abs{\nabla_v \log f}_\fra^2(\gamma(r))  \varphi^2(y,w) \right)_+\\
	& \hspace{12.5cm} \dx (y,w) \dx r \\ 
	&\hphantom{\le} +\frac{\eta}{2c_\varphi} \int_0^1\int_{B} \left( \frac{2c_{3}}{\sqrt{\lambda}\, \iota }r^{-\frac{1}{2}}\abs{\nabla_v \log f}_\fra(\gamma(r))  \varphi(y,w) -  \abs{\nabla_v \log f}^2_\fra(\gamma(r)) \varphi^2(y,w) \right)_+ \hspace{-0.3cm}\dx (y,w) \dx r,
\end{align*}
where $c_{4}>0$ is such that $\abs{\dot{\gamma}_v(r)} \le c_{4} r^{-\frac{1}{2}}$ for all $r \in (0,1]$, see Section~\ref{sec:connecting} property \hyperlink{link:4}{\textbf{(4)}}.
 
To conclude the proof of the desired estimate, we integrate on $K_-$ and then estimate two terms of the form 
\begin{align*}
	&\int_0^{\eta-\iota} \int_{\delta B} \int_0^1 \int_{B} \left(Mr^{-\frac{1}{2}}\abs{\nabla_v \log f}_\fra(\gamma(r)) \varphi(y,w)-   \abs{\nabla_v \log f}_\fra^2(\gamma(r)) \varphi^2(y,w)\right)_+ \\
	&\hspace{12cm} \dx (y,w) \dx r \dx(x,v) \dx t 
\end{align*}
for $M \in \left\{4c_{2c}  \norm{\nabla \varphi}_\infty \sqrt{\Lambda},\frac{2c_{4}}{\sqrt{\lambda}\, \iota } \right\}$.

Next, we split up the $r$-integral. First, we treat the case of small $r \in (0,r_0)$ with $r_0 \in (0,1)$ given by property \hyperlink{link:3}{\textbf{(3)}} \hyperlink{link:3b}{\textbf{(b)}} in Section~\ref{sec:connecting}. We substitute $\tilde{t} = \gamma_t(r) = t+r(\eta-t)$ and then 
\begin{equation*}
	(\tilde{x},\tilde{v}) = \Psi(x,v) = \Psi_{r,\tilde{t},\eta,y,w}(x,v) = \begin{pmatrix}
		\gamma_x(r) \\
		\gamma_v(r)
	\end{pmatrix}  = \B(r) \begin{pmatrix}
		x \\ v
	\end{pmatrix} + b(r,\tilde{t},\eta,y,w)
\end{equation*}
for some $\R^{2n}$-valued function $b$ with the noted dependencies. By the choice of $R$, we have $\Psi(\delta B) \subset B_R$. 

Applying Fubini and this change of variables gives
\begin{align*}
	&\int_0^{\eta-\iota} \int_{\delta B} \int_0^{r_0} \int_{B} \left(Mr^{-\frac{1}{2}}\abs{\nabla_v \log f}_\fra(\gamma(r)) \varphi(y,w)-   \abs{\nabla_v \log f}_\fra^2(\gamma(r)) \varphi^2(y,w)\right)_+   \\
	& \hspace{10cm} \dx (y,w) \dx r \dx(x,v) \dx t  \\
	&\le  \int_0^{r_0} \int_{B} \int_0^{\eta} \int_{\Psi(\delta B)} \left(Mr^{-\frac{1}{2}}\abs{\nabla_v \log f}_\fra(\tilde{t},\tilde{x},\tilde{v}) \varphi(y,w)-   \abs{\nabla_v \log f}_\fra^2(\tilde{t},\tilde{x},\tilde{v}) \varphi^2(y,w)\right)_+ \\
	& \hspace{8cm} \frac{1}{1-r} \abs{\det \B(r)}^{-1} \dx(\tilde{x},\tilde{v}) \dx \tilde{t}  \dx (y,w) \dx r   \\
	&\le \int_{B} \int_0^{\eta} \int_{B_R} \int_0^{r_0} \left(Mr^{-\frac{1}{2}}\abs{\nabla_v \log f}_\fra(\tilde{t},\tilde{x},\tilde{v}) \varphi(y,w)-   \abs{\nabla_v \log f}_\fra^2(\tilde{t},\tilde{x},\tilde{v}) \varphi^2(y,w)\right)_+ \\
	& \hspace{8cm} \sup_{s \in (0,r_0)} \frac{\abs{\det \B(s)}^{-1}}{1-s}  \dx r\dx(\tilde{x},\tilde{v}) \dx \tilde{t}  \dx (y,w) .
\end{align*}
Writing $p := p(\tilde{t},\tilde{x},\tilde{v},y,w) := \abs{\nabla_v \log f}_\fra(\tilde{t},\tilde{x},\tilde{v}) \varphi(y,w)$, which does not depend on $r$, we estimate the inner integral
	\begin{equation*}
		\int_{0}^{r_0}  \left(  Mr^{-\frac{1}{2}} p-p^2 \right)_+ \dx r  = 2\theta^{\frac{1}{2}}Mp-\theta p^2 \le M^2\le C(\delta,\eta,\iota,n)\mu,
	\end{equation*}
where $\theta = \min \{r_0,M^2/p^2 \}$ whatever the value of $p >0$. This shows the universal bound for the first term.

For the integral for large $r \in (r_0,1)$, we employ Young's inequality as
\begin{align*}
	&\int_0^{\eta-\iota} \int_{\delta B} \int_{r_0}^1 \int_{B} \left(Mr^{-\frac{1}{2}}\abs{\nabla_v \log f}_\fra(\gamma(r)) \varphi(y,w)-   \abs{\nabla_v \log f}_\fra^2(\gamma(r)) \varphi^2(y,w)\right)_+ \\
	&\hspace{11cm}  \dx (y,w) \dx r \dx(x,v) \dx t \\
	&\le \int_0^{\eta-\iota} \int_{\delta B} \int_{r_0}^1 \int_{B} \frac{1}{4}M^2r^{-1}   \dx (y,w) \dx r \dx(x,v) \dx t \\
	&\le C(\delta,\eta,\iota,n)M^2 \le C(\delta,\eta,\iota,n)\mu.
\end{align*}

The estimate for $K_+$ follows similarly. Note that the sign of the difference and the direction of the kinetic trajectory are reversed in this case. 
\end{proof}

\begin{remark} \label{rem:difficulty}
		Let us comment on the main difficulty of the proof. In the end, it all boils down to simultaneously controlling
		\begin{equation*}
			\int_0^{1} (\abs{\dot{\gamma}_t(r)} \abs{(\A_{\eta-t}^{-1}(r))_{\cdot ;2}} p- \abs{\dot{\gamma}_t(r)}p^2)_+ \dx r \quad \text{ and } \quad \int_0^{1} (\abs{\dot{\gamma}_v(r)} p- \abs{\dot{\gamma}_t(r)}p^2)_+ \dx r,
		\end{equation*}
		where $\gamma$ is a kinetic trajectory as in equation \eqref{eq:kintrajAB}.
		Assuming power-type behaviour, the integral 
		\begin{equation*}
			\int_0^{1} (r^j p- r^k p^2)_+ \dx r
		\end{equation*}
		for $j,k>-1$, can be bounded independently of $p$ if and only if $\frac{k-1}{2} \le j $. The calculations above show that there is a trade-off between the scaling of the forcing and the scaling of the inverse matrix. 
		
		In the parabolic case one can use the trajectory $\gamma(r) = (t+r^a(\eta-t),x+r^{b}(y-w))$ with $a,b>0$, where both integrals are controllable if and only if $a = 2b$, i.e. $\gamma$ is a parabolic trajectory, see \cite{niebel_trajectorial_2022}. What complicates the trajectories in the (higher-order) kinetic setting is the higher order relation between $ \abs{(\A^{-1}_{\eta-t}(r))_{\cdot; 2}}$ and $\abs{\dot{\gamma}_v(r)}$, see also Appendix~\ref{sec:smoothnotwork}. 
\end{remark}

\begin{remark} \label{rem:critical}
	Instead of connecting $(t,x,v)$ and $(\eta,y,w) $ directly, we could also first use the kinetic mollification to connect $(t,x,v)$ to $\gamma^{\bold  m}(\tau;(t,x,v))$ for some small $\tau$. In this step, we do the integration by parts, where the critical behaviour of the Jacobian comes into play. After that, we use simple, i.e.\ non-critical (such as the ones explained in Remark \ref{rem:action}), kinetic trajectories to connect the smoothened value $\gamma^{\bold  m}(\tau;(t,x,v))$ with $(\eta,y,w)$. At this level, we have more freedom when doing the integration by parts, similar to the argument in \cite{anceschi2024poincare}. We can either do the change of variables corresponding to $(y,w)$ or to $(m_1,m_2)$. This separation of difficulties may help to treat more involved hypoelliptic problems. This remark also applies to the proof of the Poincar\'e inequality as given in \cite{niebel_kinetic_2022-1,anceschi2024poincare}.   
\end{remark}

\section{Weak and strong Harnack inequality for the Kolmogorov equation}
\label{sec:harnack}

The following statements are invariant under the kinetic scaling, and the parameter is denoted by $r \in (0,\infty)$. Moreover, the center of the kinetic cylinders $(t_0,x_0,v_0) \in \R^{1+2n}$ can be shifted by using the translation invariance of the Kolmogorov equation. The unit scale corresponds to $r = 1$ and $(t_0,x_0,v_0) = (0,0,0)$. Kinetic cylinders are defined in Section~\ref{sec:mr}.

In all of the following results, one could also consider lower order terms of the form $S_1 \cdot \nabla_v f$ in the Kolmogorov equation \eqref{eq:kolharnack} with $S_1 \in \L^q(\Omega_T;\R^n)$ for $q>4n+2$, see~\cite{anceschi_note_2022}.

\subsection{Local boundedness of weak subsolutions}

\begin{theorem} \label{thm:loclinf}
  Given $\delta \in (0,1)$, $p >0$, $q>2n+1$ and $\tau>0$, there exists a constant $C = C(\delta,n,p,q,\tau)>0$ such that for any nonnegative weak subsolution $f$ to equation \eqref{eq:kolharnack} (with the assumptions \hyperlink{link:H1}{\textbf{(H1)}}--\hyperlink{link:H2}{\textbf{(H2)}}) on the kinetic cylinder $\tilde{Q}:=Q_{(\sqrt{\tau}\,r,r,r)}(t_0,x_0,v_0)$, the inequality
  \begin{equation*}
     \sup_{\delta \tilde{Q}} f \le C \mu^{\frac{4n+2}{p}} \left[ \left( \frac{1}{|\tilde{Q}|} \int_{\tilde{Q}} f^p \dx (t,x,v) \right)^{\frac1p}+r^2 \left( \frac{1}{|\tilde Q|} \int_{\tilde Q} \abs{S}^q \dx (t,x,v) \right)^{\frac1q} \right]
  \end{equation*}  
  holds with 
  \begin{align*}
  	 \delta \tilde{Q} &:= Q_{(\sqrt{\delta \tau}\, r, \delta^{\frac{1}{3}} r, \delta r)}(t_0,x_0,v_0) \\
  	 &= \left\{ z \in \R^{1+2n} : -\delta \tau r^2 < t-t_0 < 0, \; \abs{x-x_0-(t-t_0)v_0} < \delta r^3, \; \abs{v-v_0}< \delta r \right\}.
  \end{align*}
\end{theorem}   

\begin{figure}[H]
\centering 
\tikzmath{\x = 0.2; \y = 0.2;}
	\begin{tikzpicture}[scale=4]
  \draw[thick,->] (0, 0) -- (2,0) node[right] {$t$};
  \draw[thick,->] (0, 0) -- (0,0.3) node[left] {$(x,v)$};
  \draw[draw=black] (0.4-\x,0.4-\y) rectangle ++(1.5,1);
  \draw[draw=black] (1.2-\x,0.6-\y) rectangle ++(0.7,0.6);
  \draw (0.4-\x,1.4-\y) node[anchor=north west] {$\tilde{Q}$};
  \draw (1.2-\x,1.2-\y) node[anchor=north west] {$\delta \tilde{Q}$};
  \filldraw (1.9-\x,0.9-\y) circle[radius=0.5pt];
  \draw (1.9-\x,0.9-\y) node[anchor= west] {$(t_0,x_0,v_0)$};
   \draw [thick] (1.9-\x, 0) -- ++(0, -.05) ++(0, -.15) node [below, outer sep=0pt, inner sep=0pt] {\small\(t_0\)};
   \draw [thick] (1.2-\x, 0) -- ++(0, -.05) ++(0, -.15) node [below, outer sep=0pt, inner sep=0pt] {\small\(t_0-\delta \tau r^2 \)};
   \draw [thick] (0.4-\x, 0) -- ++(0, -.05) ++(0, -.15) node [below, outer sep=0pt, inner sep=0pt] {\small\(t_0-\tau r^2 \)};
 \end{tikzpicture}
 \caption{The sets $\tilde{Q}$ and $\delta \tilde{Q}$ centered at $(t_0,x_0,v_0)$ in Theorem~\ref{thm:loclinf}.\label{fig:lplinf}}
\end{figure}
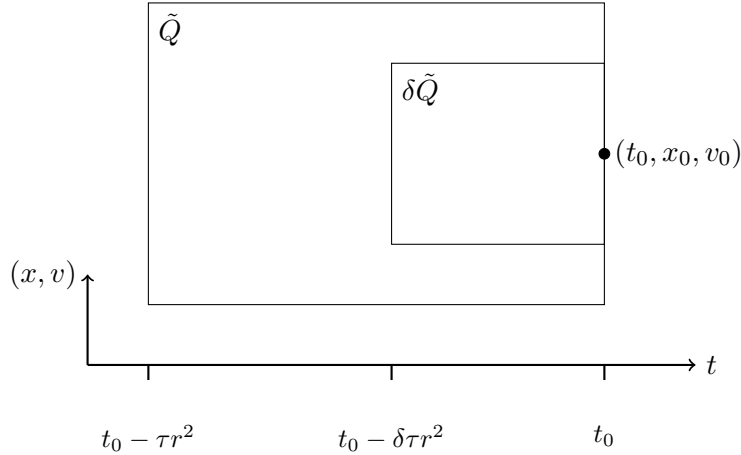

\begin{remark}
  \begin{enumerate}
  \item In the elliptic and parabolic setting, this $\L^p-\L^\infty$ estimate is referred to as ``first lemma of De~Giorgi'', ``iterated gain of integrability'', ``De~Giorgi-Moser iteration'' or ``Moser iteration'' because of the iteration procedure at the core of the proof.
  \item If one does not consider a source term, i.e. $S=0$, the extension to the kinetic setting was first obtained in~\cite{pascucci_mosers_2004}, for $p \ge 1$, with extensions to more general equations and a slightly different proof later obtained in~\cite{golse_harnack_2019,zhu_velocity_2021}. In~\cite{anceschi_note_2022,anceschi_mosers_2019}, it was proved for $p>\frac{1}{2}$ and finally in~\cite{guerand_quantitative_2022}, it was proved for all $p>0$.
  \item The case when the source $S \in \L^q$ has been considered in~\cite{golse_harnack_2019} for $q>12n+6$, in~\cite{zhu_velocity_2021} for $q>\frac{5}{2}(4+n)(1+2n)$, in~\cite{anceschi_mosers_2019} for $q>\frac{3}{2}(2n+1)$ and in~\cite{anceschi_note_2022} for $q>2n+1$. By analogy with the elliptic and parabolic theories where the natural critical Lebesgue exponent on the source term is half the homogeneous dimension, i.e.\ the sum of dimensions on each variable weighted by their natural scaling, it is expected for the Kolmogorov equation, which has homogeneous dimension $2+4n$, that the critical Lebesgue exponent for the source is indeed $1+2n$ as in our statement; see~\cite{niebel_kinetic_2022} for a discussion.
  \item If one wants to combine the local boundedness of solutions and the weak Harnack inequality to deduce the Harnack inequality, it is important to have the full range $p>0$ in this inequality. But a standard iterative argument adapted from the elliptic and parabolic theories easily shows that the inequality for some $p >0$ implies the inequality for all $p \in (0,\infty)$, see for instance~\cite[Theorem 5.2.9]{saloff-coste_aspects_2002} in the parabolic setting, and see the last eight lines of~\cite[Proof of Proposition~12]{guerand_quantitative_2022} in the kinetic setting.
  \item The statement is invariant with respect to the kinetic scaling and translation. 
  \end{enumerate}
\end{remark}

\subsection{The weak Harnack inequality}

\begin{theorem} \label{thm:weakH}
  Given $\delta \in (0,1)$, $\tau >0$, $p \in \left(0,1+\frac{1}{2n}\right)$ and $q>2n+1$, there exist $C = C(\delta,\mu,n,p,q,\tau)>0$ and $R = R(\delta,\tau)>0$ such that for any nonnegative weak supersolution $f$ to the Kolmogorov equation~\eqref{eq:kolharnack} (with the assumptions \hyperlink{link:H1}{\textbf{(H1)}}--\hyperlink{link:H2}{\textbf{(H2)}}) on the kinetic cylinder $\tilde{Q} := Q_{(\sqrt{2\tau}  \, r,R^{\frac{1}{3}}r,Rr)}(t_0,x_0,v_0)$, the inequality
  \begin{equation*}
    \left( \frac{1}{\abs{Q_-}} \int_{Q_-} f^p \dx (t,x,v) \right)^{1/p} \le C \left[ \inf_{Q_+} f + r^2 \left( \frac{1}{|\tilde Q|} \int_{\tilde{Q}} \abs{S}^q \dx (t,x,v) \right)^{\frac1q} \right],
  \end{equation*}
  holds, where the past and future cylinders are given by
  \begin{equation}
    \label{eq:past-fut-cyl}
    \begin{cases}
      Q_- &:= Q_{(\sqrt{\delta \tau} \, r,\delta^{\frac{1}{3}}r,\delta r)}'(t_0-2\tau r^2,x_0,v_0), \\[2mm]
      Q_+ &:= Q_{(\sqrt{\delta \tau} \, r,\delta^{\frac{1}{3}}r,\delta r)}(t_0,x_0,v_0).
    \end{cases}
  \end{equation}
  Moreover, if $p \in \left(0,\frac{1}{\mu}\right)$, then $C = \tilde{C}^\mu$ for some $\tilde{C} = \tilde{C}(\delta,n,\mu p,q,\tau)>0$. 
\end{theorem}

\begin{remark} \label{rem:wHarnack}
  \begin{enumerate}
  \item In the previous articles~\cite{guerand_quantitative_2022,guerand_log-transform_2022,anceschi_note_2022}, the weak Harnack inequality is proved for \emph{some} $p>0$ depending on $\lambda,\Lambda$ and other parameters corresponding to our $\delta$, $\tau$ and $q$. In the present work, we obtain the weak Harnack inequality for the optimal range $p \in \left(0,1+\frac{1}{2n}\right)$ (see Appendix~\ref{sec:opt}) by a quantitative method. Moreover, the exponential dependence in $\mu$ of the constant for small $p$ allows us to deduce the strong Harnack inequality with optimal dependency on the ellipticity constants. 
  \item In~\cite{guerand_log-transform_2022,guerand_quantitative_2022} the source $S$ is assumed to be $\L^\infty$ and  in~\cite{anceschi_note_2022} it is assumed to be $\L^q$ with $q>1+2n$ as in our theorem, which is the good range in view of the scaling, see also the embeddings in \cite{niebel_kinetic_2022}. 
  \item We need to assume that $f$ is a supersolution in a slightly larger cylinder in the spatial variables, by a factor of $R = R(\delta,\tau)$ which can be calculated in terms of the size of the gap in time between the two cylinders, i.e. $2(1-\delta)\tau$. This is required to ensure that we can connect any two points in $Q_-$ and $Q_+$ respectively, by a kinetic trajectory that does not exit $\tilde{Q}$: we refer to the proof of Theorem~\ref{thm:weakl1poin} for more details. We have not tried to optimise the kinetic trajectories so as to minimise this increase factor $R$. A similar assumption is made in all previous works deriving Harnack inequalities:  \cite{golse_harnack_2019,guerand_quantitative_2022,guerand_log-transform_2022,niebel_kinetic_2022-1}.
  \item In comparison to the weak Harnack inequality in \cite{guerand_log-transform_2022}, we do not need to know that the supersolution property is satisfied at times before the past cylinder.
  \item A positive gap must be left between the past and future cylinders $Q_-$ and $Q_+$. It is easy to see why by considering, for instance, $S=0$ and $\mathfrak a=\text{Id}$, and the weak nonnegative supersolution $f=f(t,v)$ solving $\partial_t f \ge \Delta_v f$ with initial data that is $1$ in some proper subset of the $v$-domain and zero in the rest of the $v$-domain; at this initial time the infimum is zero, and a gap in time is necessary for the heat flow to spread the mass and induce a strictly positive lower bound. In fact, for $f = f(t,x,v)$, this can also be seen by looking at the fundamental solution to $(\partial_t  + v\cdot \nabla_x )f  = \Delta_v f$. 
  \end{enumerate}
\end{remark}

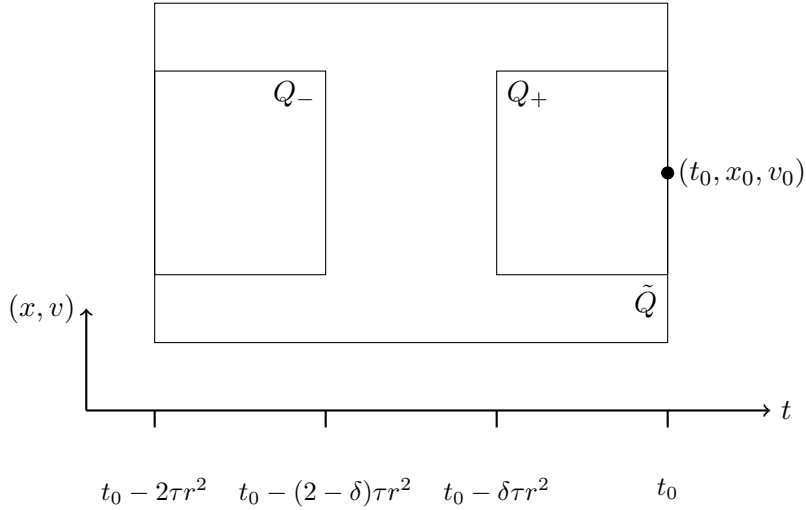
\begin{figure}[H]
\centering 
\tikzmath{\x = 0.2; \y = 0.2;}
	\begin{tikzpicture}[scale=4.5]
  \draw[thick,->] (0, 0) -- (2,0) node[right] {$t$};
  \draw[thick,->] (0, 0) -- (0,0.3) node[left] {$(x,v)$};
  \draw[draw=black] (0.4-\x,0.4-\y) rectangle ++(1.5,1);
  \draw[draw=black] (1.4-\x,0.6-\y) rectangle ++(0.5,0.6);
  \draw (1.4-\x,1.2-\y) node[anchor=north west] {$Q_+$};
  \filldraw (1.9-\x,0.9-\y) circle[radius=0.5pt];
  \draw (1.9-\x,0.9-\y) node[anchor= west] {$(t_0,x_0,v_0)$};
  \draw[draw=black] (0.4-\x,0.6-\y) rectangle ++(0.5,0.6);
   \draw (.9-\x,1.2-\y) node[anchor=north east] {$Q_-$};
      \draw (1.9-\x,0.6-\y) node[anchor=north east] {$\tilde{Q}$};
       \draw [thick] (0.4-\x, 0) -- ++(0, -.05) ++(0, -.15) node [below, outer sep=0pt, inner sep=0pt] {\small\(t_0-2\tau r^2\)};
       \draw [thick] (0.9-\x, 0) -- ++(0, -.05) ++(0, -.15) node [below, outer sep=0pt, inner sep=0pt] {\small\(t_0-(2-\delta)\tau r^2\)};
       \draw [thick] (1.4-\x, 0) -- ++(0, -.05) ++(0, -.15) node [below, outer sep=0pt, inner sep=0pt] {\small\(t_0-\delta \tau r^2\)};
        \draw [thick] (1.9-\x, 0) -- ++(0, -.05) ++(0, -.15) node [below, outer sep=0pt, inner sep=0pt] {\small\(t_0\)};
 \end{tikzpicture}
 \caption{The sets $\tilde{Q},Q_-,Q_+$ in Theorem~\ref{thm:weakH}.\label{fig:weakH}}
\end{figure}

\subsection{The (strong) Harnack inequality}

\begin{theorem} \label{thm:harnack}
  Given $\delta \in (0,1)$, $\tau >0$ and $q>2n+1$, there exist constants $C = C(\delta,n,q,\tau)>0 $ and $R=R(\delta,\tau)>0$ such that any nonnegative weak solution $f$ to the Kolmogorov equation~\eqref{eq:kolharnack} (with the assumptions \hyperlink{link:H1}{\textbf{(H1)}}--\hyperlink{link:H2}{\textbf{(H2)}}) on $\tilde{Q} := Q_{(\sqrt{(5-\delta)\tau/2} \, r,R^{\frac{1}{3}}r,Rr)}(t_0,x_0,v_0)$ satisfies
  \begin{equation} \label{eq:mr:harnack}
    \sup_{Q_-} f \le C^\mu \left[ \inf_{Q_+} f + \frac{r^2}{\lambda} \left( \frac{1}{|\tilde Q|} \int_{\tilde Q} |S|^q \dx (t,x,v) \right)^{\frac1q} \right]
  \end{equation}
  with the same $Q_\pm$ as in~\eqref{eq:past-fut-cyl} in the previous statement.
\end{theorem}

\begin{remark}
  \begin{enumerate}
  \item The exponential dependence of the Harnack constant in Theorem~\ref{thm:harnack} on $\mu$ is optimal: we give an argument in Appendix~\ref{sec:optC}.
  \item The Harnack inequality was first proved in the kinetic setting in~\cite{golse_harnack_2019} with a source $S \in \L^\infty$, by a non-quantitative argument. Constructive proofs were later obtained, still for a source $S \in \L^\infty$, in~\cite{guerand_log-transform_2022,guerand_quantitative_2022}. The proof of~\cite{guerand_log-transform_2022} (of Moser-Kruzhkov type) has been extended to a more general class of Kolmogorov operators in~\cite{anceschi_note_2022}, for a source $S \in \L^q$ with $q>2n+1$. The proof of~\cite{guerand_quantitative_2022} (trajectory approach) was extended to non-local and higher-order Kolmogorov operators in~\cite{MR4688651,anceschi2024poincare}.
  \item The same remarks about the importance of the larger cylinder $\tilde Q$ and a gap in time between $Q_-$ and $Q_+$, as in the previous statement, apply (see Remark~\ref{rem:wHarnack}-(3)-(4)). Again, we have not tried to optimise the proof in order to minimise this gap in time and the increase factor $R$. Note, however, the presence in this statement of an additional gap in time \emph{before} the start of the past cylinder $Q^-$. It is again easy to see why such a gap is needed by considering simply $S=0$, $\mathfrak a = \text{Id}$ and $f=f(t,v)$ solution to the heat equation: then without such gap one could take $f$ converging to a fundamental solution diverging at the common initial time of $Q^-$ and $\tilde Q$. This can also be observed at the level of the fundamental solution to the constant-coefficient Kolmogorov equation.
  \end{enumerate}
\end{remark}

\begin{figure}[H]
\centering 
\tikzmath{\x = 0.2; \y = 0.2;}
	\begin{tikzpicture}[scale=4.1]
  \draw[thick,->] (0, 0) -- (2.6,0) node[right] {$t$};
  \draw[thick,->] (0, 0) -- (0,0.3) node[left] {$(x,v)$};
  \draw[draw=black] (0.3-\x,0.4-\y) rectangle ++(2.2,1);
  \draw[draw=black] (2-\x,0.6-\y) rectangle ++(0.5,0.6);
  \draw (2-\x,1.2-\y) node[anchor=north west] {$Q_+$};
  \draw[draw=black] (0.8-\x,0.6-\y) rectangle ++(0.5,0.6);
   \draw (1-\x,1.2-\y) node[anchor=north east] {$Q_-$};
     \filldraw (2.5-\x,0.9-\y) circle[radius=0.5pt];
  \draw (2.5-\x,0.9-\y) node[anchor= west] {$(t_0,x_0,v_0)$};
      \draw (2.5-\x,0.6-\y) node[anchor=north east] {$\tilde{Q}$};
      \draw [thick] (0.3-\x, 0) -- ++(0, -.05) ++(0, -.15) node [below, outer sep=0pt, inner sep=0pt] {\small\(t_0-\frac{5-\delta}{2}\tau r^2\)};
      \draw [thick] (0.8-\x, 0) -- ++(0, -.05) ++(0, -.15) node [below, outer sep=0pt, inner sep=0pt] {\small\(t_0-2\tau r^2\)};
      \draw [thick] (1.3-\x, 0) -- ++(0, -.05) ++(0, -.15) node [below, outer sep=0pt, inner sep=0pt] {\small\(t_0-(2-\delta)\tau r^2\)};
      \draw [thick] (2-\x, 0) -- ++(0, -.05) ++(0, -.15) node [below, outer sep=0pt, inner sep=0pt] {\small\(t_0-\delta\tau r^2\)};
        \draw [thick] (2.5-\x, 0) -- ++(0, -.05) ++(0, -.15) node [below, outer sep=0pt, inner sep=0pt] {\small\(t_0\)};
 \end{tikzpicture}
 \caption{The sets $\tilde{Q},Q_-,Q_+$ in Theorem~\ref{thm:harnack}.\label{fig:harnack}}
\end{figure}
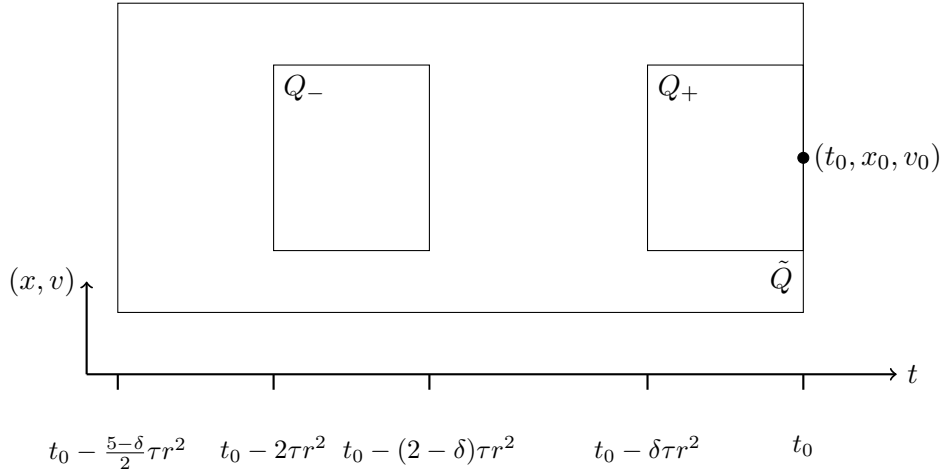

\subsection{H\"older continuity}

It is classical that Harnack inequality implies H\"older continuity in the elliptic and parabolic setting. The same holds in the kinetic setting, see~\cite{golse_harnack_2019,guerand_log-transform_2022,guerand_quantitative_2022}, with only minor modifications. The argument does not depend on the details of the equation but only on the underlying kinetic geometry. Indeed, it is important to define a notion of H\"older continuity that respects the kinetic invariances. Consider a set $Q \subset \R^{1+2n}$ and $\alpha \in (0,1)$. We say that $f \in \Cdot^\alpha_{\mathrm{kin}}(Q)$ if for all $(t,x,v),(s,y,w) \in Q$ we have
\begin{equation} \label{eq:kinholder}
  \abs{f(t,x,v)-f(s,y,w)} \le C \left( \abs{x-y-(t-s)v}^\frac{1}{3}+\abs{v-w}+\abs{t-s}^\frac{1}{2} \right)^\alpha 
\end{equation}
for a constant $C>0$. For $f \in  \Cdot^\alpha_{\mathrm{kin}}(Q)$ we define the semi-norm $[f]_{ \Cdot^\alpha_{\mathrm{kin}}}$ as the smallest constant $C$ such that the above inequality holds.

\begin{theorem}[H\"older continuity] \label{thm:holder}
  Given $q>2n+1$ and $C>0$ the constant in the Harnack inequality of Theorem~\ref{thm:harnack}, and given $f$ a weak solution to~\eqref{eq:kolharnack} on a domain $\Omega_T=(0,T) \times \Omega_x \times \Omega_v$, and $Q \ssubset \Omega_T$. Then, there exist $\alpha \ge C^{-\mu}/\log 4$ and $M = M(C,\mathrm{dist}(Q,\Omega_T))$ such that
  \begin{equation*}
    [f]_{\Cdot^\alpha_{\kin}(Q)} \le M \left( \norm{f}_{\L^\infty(\Omega_T)}+\frac{1}{\lambda }\norm{S}_{\L^q(\Omega_T)} \right).
  \end{equation*}
\end{theorem}

\begin{remark}
  \begin{enumerate}
  \item The H\"older regularity of weak solutions to~\eqref{eq:kolharnack} was first proved in~\cite{wang_calpha_2009,wang_calpha_2011,zhang_calpha_2011}, see also~\cite{wang_calpha_2019} by the same authors and~\cite{guerand_log-transform_2022} for a generalisation and simplification of this (Moser-Kruzhkov type) method. A new non-constructive proof was then obtained along the De~Giorgi method in~\cite{golse_harnack_2019}, and it was later made constructive in~\cite{guerand_quantitative_2022}.
  \item The dependence of $\alpha$ on $\mu $ might not be optimal, but it seems to be the best we can do with a method that is strong enough to prove the Harnack inequality. We refer to~\cite{mosconi_optimal_2018} for a discussion in the elliptic setting and how to obtain a better result in the elliptic non-local case. 
  \item The explicit dependency of the constants on the ellipticity bounds may be useful to study the global existence of quasilinear kinetic equations in $\L^p$-spaces, see e.g.~\cite{niebel_kinetic_2022}.
  \item The weak Harnack inequality is sufficient to derive the  H\"older continuity, see~\cite{zacher_weak_2013} for instance.
  \item Note that the kinetic notion of H\"older continuity $\C^\alpha_{\mathrm{kin}} = \L^\infty \cap \  \Cdot^\alpha_{\mathrm{kin}}$ implies the standard Euclidean notion of H\"older continuity on bounded sets (with smaller exponents due to the scaling factors). 
  \end{enumerate}
\end{remark}

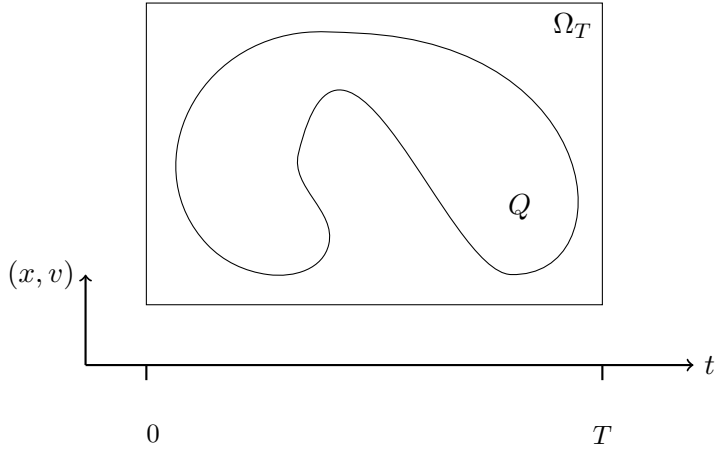
\begin{figure}[H]
\centering 
\tikzmath{\x = 0.2; \y = 0.2;}
	\begin{tikzpicture}[scale=4]
  \draw[thick,->] (0, 0) -- (2,0) node[right] {$t$};
  \draw[thick,->] (0, 0) -- (0,0.3) node[left] {$(x,v)$};
  \draw[draw=black] (0.4-\x,0.4-\y) rectangle ++(1.5,1);
      \draw (1.9-\x,1.4-\y) node[anchor=north east] {$\Omega_T$};
      \draw (1.7-\x,0.8-\y) node[anchor=north east] {$Q$};
       \draw [thick] (0.4-\x, 0) -- ++(0, -.05) ++(0, -.15) node [below right, outer sep=0pt, inner sep=0pt] {\small\(0\)};
        \draw [thick] (1.9-\x, 0) -- ++(0, -.05) ++(0, -.15) node [below, outer sep=0pt, inner sep=0pt] {\small\(T\)};
        \path[draw,use Hobby shortcut,closed=true]
(0.6,0.3) .. (0.8,0.4) .. (0.7,0.7) .. (1.4,0.3) .. (0.9,1.1) .. (0.7,1.1) .. (0.4,0.4);
 \end{tikzpicture}
 \caption{The sets $Q$ and $\Omega_T$ in Theorem~\ref{thm:holder}.\label{fig:holder}}
\end{figure}

\subsection{On the higher-order Kolmogorov equation}
We consider the higher-order Kolmogorov equation
\begin{equation} \label{eq:mr:kolk}
 (\partial_t+(\mathfrak{b} x) \cdot \nabla) f = \nabla_{x^k} \cdot \left( \mathfrak a \nabla_{x^k} f \right),
\end{equation}
as introduced in Section~\ref{sec:hypo}. Here, $\br$ is a matrix of the cascading structure as defined in \eqref{eq:frab}. Let us briefly explain how all of the main results in the kinetic setting transfer to weak solutions of \eqref{eq:mr:kolk}. We assume analogous conditions to \hyperlink{link:H1}{\textbf{(H1)}} and \hyperlink{link:H2}{\textbf{(H2)}} and introduce $\mu$ in the same way. Weak solutions to \eqref{eq:mr:kolk} are defined analogously to Definition~\ref{def:weaksol}, i.e.\ in the energy sense: $\L^\infty$ in time and $\L^2$ in the spatial variables together with $\L^2$-integrability of the $\Hdot^1$-norm in the diffusive variable $x^k$ suffice to make all the arguments rigorous, see Appendix~\ref{sec:rigorous}. In particular, we provide the a priori H\"older continuity for a larger class of solutions than considered in previous works (such as~\cite{wang_calpha_2011,wang_calpha_2009,wang_calpha_2019,zhang_calpha_2011}) since our definition of weak solutions is more general, see also \cite{anceschi_note_2022,anceschi2024poincare}.

Analogous statements to all the main results presented in Section~\ref{sec:harnack} for the Kolmogorov equation continue to hold true for weak solutions to the higher-order Kolmogorov equation \eqref{eq:mr:kolk} as well. Most importantly, one needs to replace the kinetic geometry with the higher-order kinetic geometry; in particular, the cylinders must be defined accordingly, see \cite{anceschi_note_2022}. By abuse of notation, we use the same symbol for the higher-order kinetic cylinders. Recall that $\homdim$ denotes the homogeneous dimension and that we write $(t,x) = (t,x^1,\dots,x^k) \in \R^{1+D}$. Let us now state the main results at the unit scale. 

\begin{theorem} \label{thm:mr:hokin}
	Let $\br \in \R^{D \times D}$. We assume \hyperlink{link:H1}{\textbf{(H1)}} and \hyperlink{link:H2}{\textbf{(H2)}} for $\fra = \fra(t,x^1,\dots,x^k)$. There exists a constant $R = R(\br)>1$ such that the following statements hold.
	\medskip
	
	\textbf{Gain of integrability.} Let $f$ be a nonnegative weak subsolution to \eqref{eq:mr:kolk} in the cylinder $Q_1(0)$, backwards in time and centered at $0 \in \R^{1+D}$. Then 
		\begin{equation*}
			\sup_{{Q}_{\frac{1}{2}}(0)} f \le C \mu^{\frac{\homdim}{p}} \left( \frac{1}{\abs{{Q_1(0)}}} \int_{Q_1(0)} f^p \, \dx(t,x) \right)^{\frac{1}{p}}
		\end{equation*}
		for some constant $C = C(d_1,\dots,d_k,\br,p)>0$. 
		\smallskip
		
		\textbf{Weak Harnack inequality}. Let $f$ be a nonnegative weak supersolution to \eqref{eq:mr:kolk} in the cylinder $Q_R(0)$. For a past cylinder $Q_- = Q'_1(-R,0)$ (directed forwards in time) and a future cylinder $Q_+ = Q_1(0)$, we have the weak Harnack inequality
		\begin{equation*}
			\left( \frac{1}{\abs{{Q_-}}} \int_{Q_-} f^p \, \dx(t,x) \right)^{\frac{1}{p}} \le C \inf_{Q_+} f
		\end{equation*} 
		for any $p \in \left(0,1+\frac{2}{\homdim-2}\right)$, with some constant $C = C(d_1,\dots,d_k,\br,\mu,p)>0$.
		\smallskip 
		
		\textbf{Strong Harnack inequality}. Let $f$ be a nonnegative weak solution to \eqref{eq:mr:kolk} in the cylinder $Q_{R+1}(0)$. For a past cylinder $Q_- = Q'_1(-R,0)$ and a future cylinder $Q_+ = Q_1(0)$, we have the Harnack inequality
		\begin{equation*}
			\sup_{Q_-} f \le C^\mu \inf_{Q_+} f
		\end{equation*} 
		for some constant $C = C(d_1,\dots,d_k,\br,p)>0$. \smallskip
		
		\textbf{H\"older continuity}. There exists $\alpha \ge C^{-\mu}/\log 4$ such that for any weak solution $f$ to \eqref{eq:mr:kolk} in $\tilde{Q} = Q_1(0)$, and for any cylinder $Q \ssubset \tilde{Q}$, there exists $M = M(d_1,\dots,d_k,\br,\mathrm{dist}(\tilde{Q},Q))$ such that the higher-order kinetic H\"older semi-norm is bounded by
  		\begin{equation*}
    		[f]_{\Cdot^\alpha_{\mathrm{hokin}}(Q)} \le M \norm{f}_{\L^\infty(\tilde{Q})}.
	  	\end{equation*}
\end{theorem}
\noindent With the obvious changes the figures \ref{fig:lplinf} -- \ref{fig:holder} visualise the statements of Theorem \ref{thm:mr:hokin}, too. 

\smallskip
The other main results, i.e.\ the Sobolev inequality of Section~\ref{sec:sobolev}, the estimate for the logarithm of supersolutions presented in Section~\ref{sec:weakL1log} and the $\L^p-\L^q$-estimates as in Section~\ref{sec:mvineq} can be proven along the same lines. We state the higher-order kinetic Sobolev inequality and leave any further details to the interested reader.

\begin{theorem} \label{thm:sobolevhypo}
	Let $f \in  \L^2(\R^{1+d_1+\dots+d_{k-1}};\H^1(\R^{d_k}))$ such that $(\partial_t+(\br x) \cdot \nabla_x) f = \nabla_{x^k} \cdot S$ for some $S \in \L^2(\R^{1+d_1+\dots+d_{k}};\R^{d_k})$, then 
	\begin{equation*}
		\norm{f}_{\L^{2\kappa}(\R^{1+d_1+\dots+d_{k}})} \le C \left(\norm{\nabla_{x^k} f}_{\L^{2}(\R^{1+d_1+\dots+d_{k}})} +\norm{S}_{\L^{2}(\R^{1+d_1+\dots+d_{k}})}  \right)
	\end{equation*}
	with $\kappa = \frac{\homdim}{\homdim-2}$ and $C = C(d_1,\dots,d_k,\br)>0$. 
\end{theorem}

\section{Moser iterations for weak (sub-, super-) solutions to the Kolmogorov equation}
\label{sec:mvineq}
By testing the equation with $f^\beta\varphi^2$ for $\beta \in \R \setminus \{-1 \}$ and some cutoff function $\varphi$ we can deduce energy inequalities for the function $w:= f^{\frac{\beta+1}{2}}$ on kinetic cylinders. These energy inequalities combined with the kinetic Sobolev inequality and the Moser iteration yield mean value inequalities for powers of $f$. We have to consider separately the cases $\beta > 0$, where we consider subsolutions and the cases $\beta \in (-\infty,-1)$ and $\beta \in (-1,0)$, where we consider supersolutions. Finally, for weak solutions, we consider exponents in the range $\beta \in (-\infty,-1) \cup (-1,0)$, which allows us to prove a stronger statement. 

\begin{remark}
	The estimates of  Theorem~\ref{thm:loclinf}, Proposition~\ref{prop:u-1} are already available in the literature. We reprove them in Moser's style to obtain the dependency on $\mu$, which allows us to obtain the Harnack inequality with optimal dependency on the ellipticity constants. The authors do not know of any reference where Proposition~\ref{prop:ulplg} and Proposition~\ref{prop:uLinfLpsmallp} are proven in the kinetic setting. 
\end{remark}

\begin{remark}
	In the following, we only treat the case $S = 0$. Moreover, all calculations are based on the weak formulation in Remark~\ref{rem:formal}. We refer to Appendix~\ref{sec:rigorous} for more explanations. 
\end{remark}

\begin{prop} \label{prop:u-1}
	Let $\delta \in (0,1)$, $\tau >0$. Then, for any $(t_0,x_0,v_0) \in \Omega_T$ and $r>0$ such that $Q_{(\sqrt{\tau}\,r,r,r)}(t_0,x_0,v_0)  \subset \Omega_T$ and any weak supersolution $f \ge \epsilon>0$ of \eqref{eq:kolharnack} with $S = 0$ on $\Omega_T$, we have
	\begin{equation*}
		\sup_{U_{\sigma}} f^{-1} \le \left( \frac{C\abs{U_1}^{-1}}{(\sigma'-\sigma)^{\gamma_0}} \right)^{\frac{1}{p}} \norm{f^{-1}}_{\L^p(U_{\sigma'})}
	\end{equation*} 
	for all $p \in (0,1]$ and any $\delta \le \sigma< \sigma' \le 1$. Here, $U_s := Q_{(\sqrt{s \tau} \, r,s^{\frac{1}{3}}r,s r)}(t_0,x_0,v_0)$ for $0<s \le 1$, $C = C(\delta,\mu,n,\tau)>0$ and $\gamma_0 = \gamma_0(n)>0$. 
	
	\noindent	Moreover, if $p \in \left(0,\frac{1}{\mu}\right)$, then $C$ is independent of $\mu$.
\end{prop}

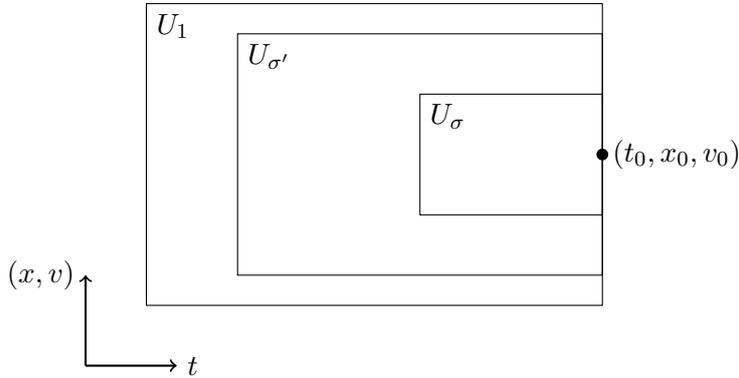
\begin{figure}[H]
\centering
\tikzmath{\x = 0.2; \y = 0.2;}
	\begin{tikzpicture}[scale=4]
  \draw[thick,->] (0, 0) -- (0.3,0) node[right] {$t$};
  \draw[thick,->] (0, 0) -- (0,0.3) node[left] {$(x,v)$};
  \draw[draw=black] (0.4-\x,0.4-\y) rectangle ++(1.5,1);
  \draw[draw=black] (0.7-\x,0.5-\y) rectangle ++(1.2,0.8);
  \draw[draw=black] (1.3-\x,0.7-\y) rectangle ++(0.6,0.4);
  \draw (0.4-\x,1.4-\y) node[anchor=north west] {$U_1$};
  \draw (0.7-\x,1.3-\y) node[anchor=north west] {$U_{\sigma'}$};
  \draw (1.3-\x,1.1-\y) node[anchor=north west] {$U_{\sigma}$};
  \filldraw (1.9-\x,0.9-\y) circle[radius=0.5pt];
  \draw (1.9-\x,0.9-\y) node[anchor= west] {$(t_0,x_0,v_0)$};
 \end{tikzpicture}
 \caption{The sets $U_\sigma$ in Proposition~\ref{prop:u-1} centered at $(t_0,x_0,v_0)$.}
\end{figure}

\begin{proof}
	By kinetic translation and scaling, we may assume $(t_0,x_0,v_0) = 0$ and $r = 1$. Let $\sigma,\sigma' \in \R$ with $\delta \le \sigma< \sigma' \le 1$. 
	
	To simplify the following calculations, we introduce some notation. 
	Define $V_{\rho} = U_{\sigma' \rho}$ for  $0<\rho\le 1$. Given fixed $0<\rho<\rho'\le 1$ we set $\tilde{t}_1 = -\rho' \sigma' \tau$ and $\bar{\rho} = \frac{\rho+\rho'}{2}$. Consider now a cut-off function $\varphi = \varphi(t,x,v)$ with $\varphi = 1$ on $V_{\bar{\rho}}$ and $\varphi = 0$ for $t \le \tilde{t}_1$ or $(x,v) \notin B_{\rho'\sigma'}(0) \times B_{\rho'\sigma'}(0)$. As noted in Remark~\ref{rem:cutoff} we may assume 
	\begin{equation} \label{eq:cutoffLpLinf}
			\norm{\nabla_v \varphi}_\infty \le \frac{C}{\rho'-\rho} \quad \text{ and } \quad \norm{ (\partial_t +v \cdot \nabla_x)\varphi}_\infty \le \frac{C}{\rho'-\rho}
	\end{equation}
	for some constant $C = C(\delta,\tau)>0$, and due to $\rho'-\bar{\rho} = (\rho'-\rho)/2 = \bar{\rho}-\rho$. We introduce $\tilde{B} = B_{\sigma'}(0) \times B_{\sigma'}(0) \subset \R^{2n}$ and write $V_1 = (-\sigma'\tau,0) \times \tilde{B}$.

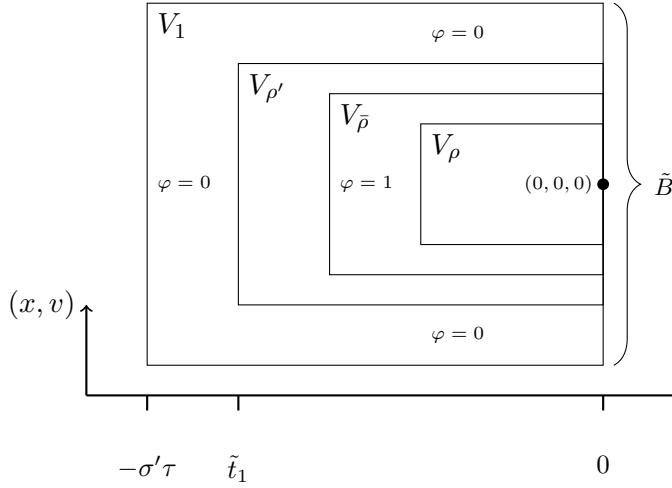
\begin{figure}[H]
\centering
\tikzmath{\x = 0.2; \y = 0.2;}
	\begin{tikzpicture}[scale=4]
  \draw[thick,->] (0, 0) -- (2.1,0) node[right] {$t$};
  \draw[thick,->] (0, 0) -- (0,0.3) node[left] {$(x,v)$};
  \draw[draw=black] (0.4-\x,0.3-\y) rectangle ++(1.5,1.2);
  \draw[draw=black] (0.7-\x,0.5-\y) rectangle ++(1.2,0.8);
  \draw[draw=black] (1.0-\x,0.6-\y) rectangle ++(0.9,0.6);
  \draw[draw=black] (1.3-\x,0.7-\y) rectangle ++(0.6,0.4);
  \draw (0.4-\x,1.5-\y) node[anchor=north west] {$V_1$};
  \draw (0.7-\x,1.3-\y) node[anchor=north west] {$V_{\rho'}$};
  \draw (1.0-\x,1.2-\y) node[anchor=north west] {$V_{\bar{\rho}}$};
  \draw (1.3-\x,1.1-\y) node[anchor=north west] {$V_{\rho}$};
  \filldraw (1.9-\x,0.9-\y) circle[radius=0.5pt];
  \draw (1.9-\x,0.9-\y) node[anchor= east] {\tiny $(0,0,0)$};
  \draw (1.0-\x,0.9-\y) node[anchor= west] {\tiny $\varphi = 1$};
  \draw (0.4-\x,0.9-\y) node[anchor= west] {\tiny $\varphi = 0$};
  \draw (1.3-\x,1.4-\y) node[anchor= west] {\tiny $\varphi = 0$};
  \draw (1.3-\x,0.4-\y) node[anchor= west] {\tiny $\varphi = 0$};
  \draw [thick] (1.9-\x, 0) -- ++(0, -.05) ++(0, -.15) node [below, outer sep=0pt, inner sep=0pt] {\small\(0\)};
  \draw [thick] (0.7-\x, 0) -- ++(0, -.05) ++(0, -.15) node [below, outer sep=0pt, inner sep=0pt] {\small\(\tilde{t}_1\)};
  \draw [thick] (0.4-\x, 0) -- ++(0, -.05) ++(0, -.15) node [below, outer sep=0pt, inner sep=0pt] {\small\(-\sigma' \tau\)};
  \draw [decorate,decoration={brace,amplitude=10pt,mirror,raise=4pt},yshift=0pt] (1.9-\x,0.3-\y) -- (1.9-\x,1.5-\y) node [black,midway,xshift=0.8cm] {\footnotesize $\tilde{B}$};
 \end{tikzpicture}
 \caption{The sets $V_\rho,V_{\bar{\rho}},V_{\rho'},V_1$ and the level sets of $\varphi$ in the proof of Proposition~\ref{prop:u-1}.}
\end{figure}
	Let us comment on this choice of cylinders. To apply the kinetic Sobolev inequality, we need to enlarge the cylinder. Thus, we go from $V_\rho $ to $V_{\bar{\rho}}$ to apply Theorem~\ref{thm:kinemb}. While deducing the energy estimate for a power of $f$, we need to enlarge the cylinder, from $V_{\bar{\rho}}$ to $V_{\rho'}$, too. The choice of $\bar{\rho}$ ensures that $(\rho'-\bar{\rho}),(\bar{\rho}-\rho) \sim (\rho'-\rho)$. 		
	
	Consider now the test function $\varphi^2 f^\beta$ for $\beta < -1$ in the weak formulation \eqref{eq:weaksolf}. Strictly speaking, this is not an admissible test function. We explain the necessary steps to make this argument rigorous in Appendix~\ref{sec:rigorous}. 
	
	Note that 
	\begin{equation*}
		[\partial_t  +v \cdot \nabla_x ](f^{\beta+1}\varphi^2) = (\beta+1) f^\beta\varphi^2 (\partial_t  +v \cdot \nabla_x) f + 2 f^{\beta+1}\varphi (\partial_t  +v \cdot \nabla_x )\varphi,
	\end{equation*}
	thus we obtain 
	\begin{align*}
		0 &\le \int_{V_1} \left( [(\partial_t + v \cdot \nabla_x ) f] \varphi^2 f^\beta+  \langle \fra \nabla_v f, \nabla_v (\varphi^2 f^\beta) \rangle \right) \dx (t,x,v)\\
		&= \int_{V_1} \frac{1}{\beta+1} \left( [\partial_t + v \cdot \nabla_x](f^{\beta+1}\varphi^2) -2 f^{\beta+1}\varphi ((\partial_t  +v \cdot \nabla_x) \varphi)\right) + \langle \fra \nabla_v f , \nabla_v (f^{\beta}\varphi^2) \rangle \dx (t,x,v).
	\end{align*}
	Hence, 
	\begin{equation*}
		\int_{V_1} \langle \fra \nabla_v f , \nabla_v (f^{\beta}\varphi^2) \rangle \dx (t,x,v) \ge \frac{1}{\beta+1} \int_{V_1} 2f^{\beta+1} \varphi (\partial_t +v \cdot \nabla_x) \varphi \dx (t,x,v)
	\end{equation*}
	by an integration by parts in $x$ with $\varphi = 0 $ on the boundary of $\tilde{B}$, by evaluating the time integral and using that $f^{\beta+1}\varphi \ge 0$, and as $\beta+1<0$.  
	
	Next, $\nabla_v(f^\beta \varphi^2) = 2 \varphi   f^\beta \nabla_v \varphi + \beta \varphi^2 f^{\beta-1} \nabla_v f$, whence
	\begin{align*}
		&\beta(\beta+1)\int_{V_1} \langle \fra \nabla_v f , \nabla_v f \rangle f^{\beta-1} \varphi^2 \dx (t,x,v) \\
		 &\le 2\int_{V_1} f^{\beta+1} \varphi (\partial_t +v \cdot \nabla_x) \varphi \dx (t,x,v) - 2(1+\beta) \int_{V_1} \langle \fra \nabla_v f, \nabla_v \varphi \rangle  \varphi f^\beta \dx (t,x,v)
	\end{align*}
	as $\beta+1<0$.	
	
	We introduce the function $w = f^{\frac{\beta+1}{2}}$. We have $\nabla_v w = \frac{\beta+1}{2}f^{\frac{\beta-1}{2}} \nabla_v f$. Therefore, 
	\begin{align*}
		\beta(\beta+1)\int_{V_1} \langle \fra \nabla_v f , \nabla_v f \rangle f^{\beta-1} \varphi^2 \dx (t,x,v)  = \frac{4\abs{\beta}}{\abs{\beta+1}} \int_{V_1} \abs{\nabla_v w}_\fra^2 \varphi^2 \dx (t,x,v)
	\end{align*} 
	as $\beta(\beta+1)>0$. Next, by assumption \hyperlink{link:H2}{\textbf{(H2)}} and Young's inequality we obtain
	\begin{align*}
		2\abs{1+\beta}\abs{\langle \fra \nabla_v f , \nabla_v \varphi\rangle \varphi f^{\beta}}&= 4\abs{\langle \fra \nabla_v w , \nabla_v \varphi\rangle \varphi w} \\
		&\le \frac{ 2\abs{\beta}}{\abs{\beta+1}} \abs{\nabla_v w}_\fra^2\varphi^2 + 2\Lambda\frac{\abs{\beta+1}}{ \abs{\beta}} \abs{\nabla_v \varphi}^2 w^2, 
	\end{align*}
	and absorb the first term. 
	
	Together, we end up with 
	\begin{align*}
		 &\int_{V_1} \abs{\nabla_v w}_\fra^2  \varphi^2 \dx (t,x,v) \\
		 &\le \Lambda\frac{\abs{1+\beta}^2}{\abs{\beta}^2} \int_{V_1} \abs{\nabla_v \varphi}^2 w^2 \dx (t,x,v) + \frac{\abs{\beta+1}}{\abs{\beta}}\int_{V_1} \abs{(\partial_t + v \cdot \nabla_x )\varphi } \varphi w^2 \dx (t,x,v). 
	\end{align*}
	
	Recalling that $(\rho'-\rho)^{-1} \le (\rho'-\rho)^{-2}$ and using the estimates for $\varphi$ we obtain 
	\begin{align*}
		\int_{V_{\bar{\rho}}} \abs{\nabla_v w}_\fra^2  \dx (t,x,v) &\le \int_{V_1} \abs{\nabla_v w}_\fra^2  \varphi^2 \dx (t,x,v) \\
		&\le \frac{C}{(\rho'-\rho)^2}\left( \mu \frac{\abs{1+\beta}^2}{\abs{\beta}^2} + \frac{\abs{\beta+1}}{\abs{\beta}} \right) \int_{V_{\rho'}}  w^2 \dx (t,x,v),
	\end{align*}
	where $C = C(\delta,\tau)>0$.  
	
	We conclude
	\begin{equation} \label{eq:energywbeta-1}
		\norm{\abs{\nabla_v{w}}_\fra}_{\L^{2}(V_{\bar{\rho}})} \le \frac{C}{\rho'-\rho}\left( \mu\frac{\abs{1+\beta}^2}{\abs{\beta}^2} +  \frac{\abs{\beta+1}}{\abs{\beta}} \right)^{\frac{1}{2}}\norm{w}_{\L^{2}(V_{\rho'})} 
	\end{equation}
	and, employing the ellipticity of $\fra$, this yields
	\begin{equation} \label{eq:energywbeta-2}
		\norm{\nabla_v{w}}_{\L^{2}(V_{\bar{\rho}})} \le \frac{C}{\rho'-\rho}\left( \mu^2\frac{\abs{1+\beta}^2}{\abs{\beta}^2} + \mu \frac{\abs{\beta+1}}{\abs{\beta}} \right)^{\frac{1}{2}}\norm{w}_{\L^{2}(V_{\rho'})} 
	\end{equation}
	with $C = C(\delta,\tau)>0$.

Consider the function $s \mapsto s^{\frac{\beta+1}{2}}$ for $s>0$. It is nonincreasing and convex as $\beta < -1$. Hence, $w = f^{\frac{\beta+1}{2}}$ is a nonnegative subsolution. We apply the kinetic Sobolev inequality from Theorem~\ref{thm:kinemb} to deduce
	\begin{align*}
		\norm{w}_{\L^{2\kappa}(V_{\rho})} &\le   \frac{C}{\left({\rho'}-\rho\right)^\frac{3}{2}} \left( \norm{w}_{\L^2(V_{\bar{\rho}})}+\norm{\nabla_v w }_{\L^2(V_{\bar{\rho}})}+\sqrt{\Lambda}\norm{ \abs{\nabla_v w}_\fra }_{\L^2(V_{\bar{\rho}})} \right)
	\end{align*}
	where $C = C(\delta,n,\tau)>0$ and $\kappa = 1+\frac{1}{2n}$, as $\bar{\rho}-\rho = (\rho'-\rho)/2$. Using our energy estimate for $w$ as in equation \eqref{eq:energywbeta-1} and \eqref{eq:energywbeta-2} and the fact that $\abs{\beta}>1$ we deduce
	\begin{equation*}
		\norm{w}_{\L^{2\kappa}(V_{\rho})}\le \frac{C(1+\mu\abs{\beta+1})}{(\rho'-\rho)^{\frac{5}{2}}}  \norm{w}_{\L^2(V_{\rho'})},
	\end{equation*}
	where $C = C(\delta,n,\tau)>0$.

	 With $\gamma = \abs{1+\beta}>0$, rewriting this estimate in terms of $f$ it is equivalent to 
	\begin{equation*}
		\norm{f^{-1}}_{\L^{\gamma \kappa}{(V_{\rho}) }} \le \left[ \frac{C(1+\mu \gamma)^2}{(\rho'-\rho)^5} \right]^{\frac{1}{\gamma}} \norm{f^{-1}}_{\L^\gamma(V_{\rho'})}
	\end{equation*}
	with $C = C(\delta,n,\tau)>0$.
	 We are now able to apply the De~Giorgi-Moser iteration of Lemma~\ref{lem:moser1} to $f^{-1}$ with $\bar{p} = 1$. It follows that there exists constants $C = C(\delta,\mu,n,\tau)>0$ and $\gamma_0 = \gamma_0(n)>0$ such that 
	\begin{equation*}
		\sup_{V_\theta} f^{-1} \le \left( \frac{C}{(1-\theta)^{\gamma_0}} \right)^{1/p} \norm{f^{-1}}_{\L^p(V_1)} 
	\end{equation*}
	for all $\theta \in (0,1)$ and any $p \in (0,1]$. Choosing $\theta = \sigma/\sigma'$, we deduce
	\begin{equation*}
		\sup_{U_{\sigma}} f^{-1} \le \left( \frac{C\abs{U_1}^{-1}}{(\sigma'-\sigma)^{\gamma_0}} \right)^{\frac{1}{p}} \norm{f^{-1}}_{\L^p(U_{\sigma'})}
	\end{equation*}
	for all $p \in (0,1]$ with $C=C(\delta,\mu,n,\tau)>0$. Recall that the constant $C$ in Lemma~\ref{lem:moser1} can be chosen independently of $\mu$ if we additionally assume $p \in \left(0,\frac{1}{\mu}\right)$.
\end{proof}

The following proposition is one of the key ingredients to prove the weak Harnack inequality for supersolutions. Think of it as a De~Giorgi-Moser iteration, which does not go all the way to infinity but stops at an exponent $p_0 \in (0,\kappa)$ with $\kappa = 1+\frac{1}{2n}$. We emphasise that, as we are dealing with supersolutions, the kinetic cylinders need to be stacked forward in time.

\begin{prop} \label{prop:ulplg}
Let $\delta \in (0,1)$ and $\tau >0$. Then, for any $(t_0,x_0,v_0) \in \Omega_T$, $r>0$ such that $Q_{(\sqrt{\tau}\,r,r,r)}'(t_0,x_0,v_0) \subset \Omega_T$, any $p_0 \in (0,\kappa)$ with $\kappa = 1+\frac{1}{2n}$ and all nonnegative weak supersolutions $f$ of \eqref{eq:kolharnack} with $S = 0$ on $\Omega_T$ there holds
\begin{equation*}
	\norm{f}_{\L^{p_0}(U'_{\sigma})} \le \left(  \frac{C\abs{U_1'}^{-1}}{(\sigma'-\sigma)^{\gamma_0}} \right)^{\frac{1}{\gamma}-\frac{1}{p_0}} \norm{f}_{\L^\gamma(U_{\sigma'}')}
\end{equation*}
for all $\delta \le \sigma < \sigma' \le 1$, $0< \gamma \le \frac{p_0}{\kappa}<p_0$. Here, $U_s' = Q_{(\sqrt{s \tau}\,r,s^{\frac{1}{3}}r,s r)}'(t_0,x_0,v_0)$ for $0< s \le 1$, $C = C(\delta,\mu, n ,p_0,\tau)>0 $ and $\gamma_0 = \gamma_0(n)>0$. Moreover, for $p_0 \in \left(0,\frac{1}{\mu}\right)$ the constant $C$ is independent of $\mu$ and $p_0$. 
\end{prop}

\begin{figure}[H]
\centering
\tikzmath{\x = 0.2; \y = 0.2;}
	\begin{tikzpicture}[scale=4]
  \draw[thick,->] (0, 0) -- (0.3,0) node[right] {$t$};
  \draw[thick,->] (0, 0) -- (0,0.3) node[left] {$(x,v)$};
  \draw[draw=black] (0.4-\x,0.4-\y) rectangle ++(1.5,1);
  \draw[draw=black] (0.4-\x,0.5-\y) rectangle ++(1.2,0.8);
  \draw[draw=black] (0.4-\x,0.7-\y) rectangle ++(0.6,0.4);
  \draw (1.9-\x,1.4-\y) node[anchor=north east] {$U_1'$};
  \draw (1.6-\x,1.3-\y) node[anchor=north east] {$U_{\sigma'}'$};
  \draw (1.0-\x,1.1-\y) node[anchor=north east] {$U_{\sigma}'$};
  \filldraw (0.4-\x,0.9-\y) circle[radius=0.5pt];
  \draw (0.4-\x,0.9-\y) node[anchor= east] {$(t_0,x_0,v_0)$};
 \end{tikzpicture}
 \caption{The sets $U_\sigma'$ centered at $(t_0,x_0,v_0)$.}
\end{figure}
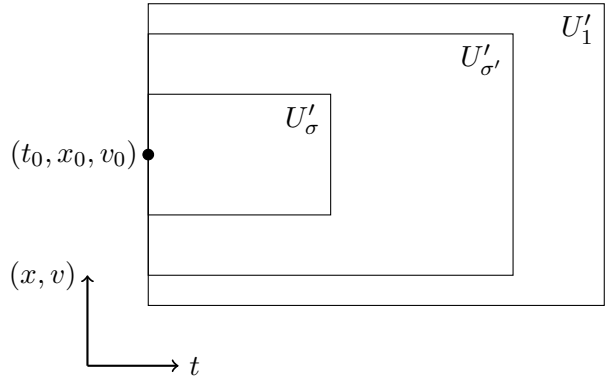

\begin{proof}
We follow the arguments of the proof of Proposition~\ref{prop:u-1}. We may assume $f \ge \epsilon>0$ qualitatively.  

	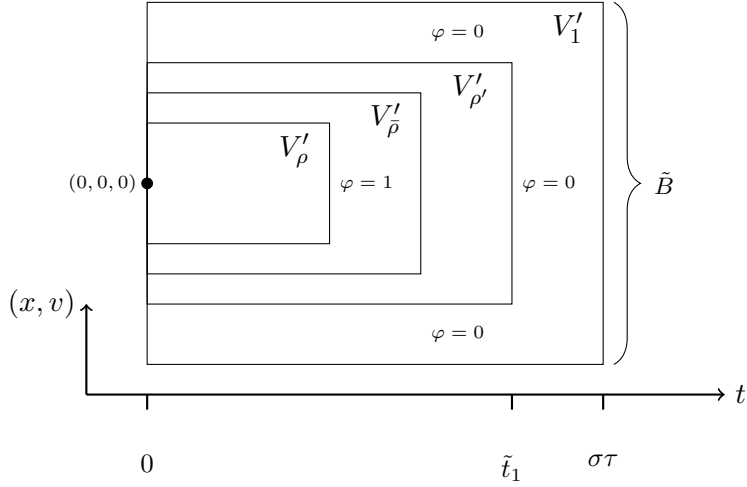
\begin{figure}[H]
\centering
\tikzmath{\x = 0.2; \y = 0.2;}
	\begin{tikzpicture}[scale=4]
  \draw[thick,->] (0, 0) -- (2.1,0) node[right] {$t$};
  \draw[thick,->] (0, 0) -- (0,0.3) node[left] {$(x,v)$};
  \draw[draw=black] (0.4-\x,0.3-\y) rectangle ++(1.5,1.2);
  \draw[draw=black] (0.4-\x,0.5-\y) rectangle ++(1.2,0.8);
  \draw[draw=black] (0.4-\x,0.6-\y) rectangle ++(0.9,0.6);
  \draw[draw=black] (0.4-\x,0.7-\y) rectangle ++(0.6,0.4);
  \draw (1.7-\x,1.5-\y) node[anchor=north west] {$V_1'$};
  \draw (1.37-\x,1.3-\y) node[anchor=north west] {$V_{\rho'}'$};
  \draw (1.1-\x,1.2-\y) node[anchor=north west] {$V_{\bar{\rho}}'$};
  \draw (0.8-\x,1.1-\y) node[anchor=north west] {$V_{\rho}'$};
  \filldraw (0.4-\x,0.9-\y) circle[radius=0.5pt];
  \draw (0.4-\x,0.9-\y) node[anchor= east] {\tiny $(0,0,0)$};
  \draw (1.0-\x,0.9-\y) node[anchor= west] {\tiny $\varphi = 1$};
  \draw (1.6-\x,0.9-\y) node[anchor= west] {\tiny $\varphi = 0$};
  \draw (1.3-\x,1.4-\y) node[anchor= west] {\tiny $\varphi = 0$};
  \draw (1.3-\x,0.4-\y) node[anchor= west] {\tiny $\varphi = 0$};
  \draw [thick] (1.9-\x, 0) -- ++(0, -.05) ++(0, -.15) node [below, outer sep=0pt, inner sep=0pt] {\small\(\sigma \tau\)};
  \draw [thick] (1.6-\x, 0) -- ++(0, -.05) ++(0, -.15) node [below, outer sep=0pt, inner sep=0pt] {\small\(\tilde{t}_1\)};
  \draw [thick] (0.4-\x, 0) -- ++(0, -.05) ++(0, -.15) node [below, outer sep=0pt, inner sep=0pt] {\small\(0\)};
  \draw [decorate,decoration={brace,amplitude=10pt,mirror,raise=4pt},yshift=0pt]
(1.9-\x,0.3-\y) -- (1.9-\x,1.5-\y) node [black,midway,xshift=0.8cm] {\footnotesize
$\tilde{B}$};
 \end{tikzpicture}
 \caption{The sets $V_\rho',V_{\bar{\rho}}',V_{\rho'}'$ and the level sets of $\varphi$ in the proof of Proposition~\ref{prop:ulplg}.}
\end{figure}

In comparison to the proof of Proposition~\ref{prop:u-1}, we consider cylinders in the positive direction of time. Define $V_{\rho}' = U_{\sigma' \rho}'$, let $0<\rho<\rho'\le 1$, set $\tilde{t}_1 = \rho' \sigma' \tau$ and $\bar{\rho} = \frac{\rho+\rho'}{2}$. Consider now a cut-off function $\varphi = \varphi(t,x,v)$ with $\varphi = 1$ on $V_{\bar{\rho}}'$ and $\varphi = 0$ for $t \ge \tilde{t}_1$ or $(x,v) \notin B_{\rho'\sigma'}(0) \times B_{\rho'\sigma'}(0)$. As noted in Remark~\ref{rem:cutoff}, we may assume the bounds on $\varphi$ as in \eqref{eq:cutoffLpLinf}. Set $\tilde{B} = B_{\sigma'}(0) \times B_{\sigma'}(0) \subset \R^{2n}$ and $V_1' = (0,\sigma'\tau) \times \tilde{B}$.

In the weak formulation \eqref{eq:weaksolf} let us consider the test function $f^\beta \varphi^2$ with $\beta \in (-1,0)$. From here we proceed exactly as in the proof of Proposition~\ref{prop:u-1}, noting that $\beta < 0$ and $\beta+1 >0$. We get 

\begin{align*}
	\int_{V_{\bar{\rho}}'} \abs{\nabla_v w}^2_\fra \dx (t,x,v) &\le  \frac{C}{(\rho'-\rho)^2}\left( \mu\frac{\abs{\beta+1}^2}{\abs{\beta}^2}+\frac{\abs{\beta+1}}{\abs{\beta} }\right) \int_{V_{\rho'}'} {w}^2 \dx (t,x,v), 
\end{align*}
where $w = f^{\frac{\beta+1}{2}}$. 

As $f$ is a supersolution we deduce that $w = f^{\frac{\beta+1}{2}}$ for $\beta \in (-1,0)$ is a nonnegative supersolution ($s \mapsto s^{\frac{\beta+1}{2}}$ for $s>0$ is concave and non-decreasing). Employing the kinetic Sobolev inequality, this gives
\begin{equation*}
	\norm{w}_{\L^{2\kappa}(V_{\rho}')} \le \frac{C\left(1+\mu \frac{\abs{\beta+1}}{\abs{\beta}}\right)}{(\rho'-\rho)^{\frac{5}{2}}} \norm{w}_{\L^{2}(V_{\rho'}')} 
\end{equation*}
with $C = C(\delta,n,\tau)>0$. Rewriting this in terms of $f$ with $\gamma = 1+\beta$ we obtain 
\begin{equation*}
	\norm{f}_{\L^{\gamma \kappa}(V_{\rho}')} \le \left( \frac{C\left(1+\mu \frac{\gamma}{1-\gamma}\right)^2}{(\rho'-\rho)^5} \right)^{\frac{1}{\gamma}} \norm{f}_{\L^\gamma(V_{\rho'}')}
\end{equation*}
for all $\gamma \in (0,p_0/\kappa]$ with $C = C(\delta,n, p_0,\tau)>0$. Note $p_0 < \kappa$, hence $1-\gamma$ is bounded away from zero.

Introducing a factor depending on dimension, we can make sure that the assumption on the normalisation of the measure in Lemma~\ref{lem:moser2} is satisfied. Let us now apply the abstract De~Giorgi-Moser iteration of Lemma~\ref{lem:moser2}, to deduce that there exists $C = C(\delta,\mu,n,p_0,\tau)>0$ and $\gamma_0 = \gamma_0(n)>0$ such that 
\begin{equation*}
	\norm{f}_{\L^{p_0}(V_{\theta}')} \le \left( \frac{C}{(1-\theta)^{\gamma_0}} \right)^{\frac{1}{\gamma}-\frac{1}{p_0}} \norm{f}_{\L^\gamma(V_{1}')}
\end{equation*}
for $0<\theta <1$ and $0<\gamma < p_0/\kappa$.  Choosing $\theta = \frac{\sigma}{\sigma'}$ gives the desired estimate. If $p_0 \in \left(0,\frac{1}{\mu}\right)$, then $\kappa-p_0>\frac{1}{2n}$ and thus the constant $C$ is independent of $\mu$ and $p_0$ as evident from the explicit formula in Lemma~\ref{lem:moser2}.
\end{proof}

Considering the case $\beta>0$, we get the local boundedness of weak subsolutions. We prove the estimate for all $p \in (0,\infty)$ and track the precise dependency of the constant on the ellipticity constants. 

\begin{proof}[Proof of Theorem~\ref{thm:loclinf}]
The first part of the proof follows along the lines of the proof of Proposition~\ref{prop:u-1}. Note that we consider $\beta > 1$ first. We consider the same cutoff function $\varphi$ corresponding to the sets $V_{\rho}$ as in the proof of Proposition~\ref{prop:u-1}.

 Testing with $f^{\beta} \varphi^2$, $\beta >1$ a similar calculation to the one above gives
	\begin{align*}
		\int_{V_{\bar{\rho}}} \abs{\nabla_v w}^2_\fra \dx (t,x,v) &\le  C\left[ \mu \frac{(\beta+1)^2}{\beta^2}+\frac{{\beta+1}}{{\beta} }\right] \frac{1}{(\rho'-\rho)^2}\int_{V_{\rho'}} w^2 \dx (t,x,v)
	\end{align*}
	with $C = C(\delta,\tau)>0$ and $w = f^{\frac{\beta+1}{2}}$. As $\beta>1$, we know that $w$ is a subsolution, too. Hence, using Theorem~\ref{thm:kinemb} we conclude
	\begin{equation*}
		\norm{w}_{\L^{2\kappa}(V_\rho)} \le \frac{C \left(1+\mu \frac{\beta+1}{\beta} \right)}{(\rho'-\rho)^{\frac{5}{2}}} \norm{w}_{\L^2(V_{\rho'})}
	\end{equation*}
	with $C = C(\delta,n,\tau)>0$ and in terms of $f$ we have
	\begin{equation*}
		\norm{f}_{\L^{\gamma \kappa}(V_\rho)} \le \left(\frac{C\left(1+\mu \frac{\gamma}{\abs{\gamma-1}} \right)^2}{(\rho'-\rho)^5}\right)^{1/\gamma} \norm{f}_{\L^\gamma(V_{\rho'})}
	\end{equation*}
	with $C = C(\delta,n,\tau)>0$ and $\gamma= 1+\beta >1$. Let $p \ge 2$, then 
	\begin{equation*}
		\norm{f}_{\L^{\gamma \kappa}(V_\rho)} \le \left(\frac{C(1+ \mu\gamma )^2}{(\rho'-\rho)^5}\right)^{1/\gamma} \norm{f}_{\L^\gamma(V_{\rho'})}
	\end{equation*}
	for $C = C(\delta,n,p,\tau)>0$ for all $\gamma \ge p$.
		
	We apply the abstract De~Giorgi-Moser iteration of Lemma~\ref{lem:moser1} to deduce
	\begin{equation*}
		\sup_{V_\delta} f \le \left( \frac{C \mu^{4n+2}}{(1-\delta)^{\gamma_0}} \right)^{1/p} \norm{f}_{\L^p(V_1)}
	\end{equation*} 
	for all $\delta \in (0,1)$ with $C = C(\delta,n,p,\tau)>0$. Choosing $\delta = \sigma/\sigma'$ gives
	\begin{equation*}
		\sup_{U_{\sigma}} f \le \left( \frac{C \mu^{4n+2}}{(\sigma'-\sigma)^{\gamma_0}} \right)^{1/p} \norm{f}_{\L^p(U_{\sigma'})}
	\end{equation*} 
	with $C = C(n,\delta,p,\tau)>0$ (as $\gamma_1 \frac{\kappa}{\kappa-1} = 4n+2$ in Lemma~\ref{lem:moser1}). This concludes the proof of the theorem for all $p \ge 2$.

	In particular, the desired inequality holds for $p =2$. We give an argument to deduce the inequality for $p \in (0,2)$ from the already proven inequality with $p = 2$. As before let $0<\delta \le \sigma < \sigma' \le 1$. We recall $V_{\rho} = U_{\rho \sigma'}$ for $0<\rho \le 1$ and let $0<\rho < \rho' \le 1$.
	
	First, note that for $p \in (0,2)$
	\begin{equation*}
		\norm{f}_{\L^2(V_{\rho'})}^2 \le \norm{f}_{\L^\infty(V_{\rho'})}^{2-p} \norm{f}_{\L^p(V_{\rho'})}^p \le \norm{f}_{\L^\infty(V_{\rho'})}^{2-p} \norm{f}_{\L^p(V_{1})}^p.
	\end{equation*}
	Hence, by the already proven statement for $p = 2$, we deduce
	\begin{align*}
		\sup_{V_\rho} f \le \frac{C \mu^{\frac{4n+2}{2}}}{[(\rho'-\rho)\sigma']^{\gamma_0/2}} \norm{f}_{\L^2(V_{\rho'})} \le \frac{C\mu^{\frac{4n+2}{2}}}{(\rho'-\rho)^{\gamma_0/2}} \norm{f}_{\L^\infty(V_{\rho'})}^{1-\frac{p}{2}} \norm{f}_{\L^p(V_{1})}^{\frac{p}{2}},
	\end{align*}
	with $C = C(\delta,n,\tau)>0$. Applying Young's inequality to the latter inequality implies
	\begin{equation*}
		\sup_{V_{\rho}} f \le \frac{1}{2}\sup_{V_{\rho'}} f + \frac{C\mu^{\frac{4n+2}{p}}}{(\rho'-\rho)^{\gamma_0/p}} \norm{f}_{\L^p(V_1)},
	\end{equation*}
	with $C = C(\delta,n,\tau,p)>0$. With another simple iteration procedure, see \cite[Lemma 4.3]{han_elliptic_2011} or \cite[Proof of Proposition 12]{guerand_quantitative_2022}, we conclude
	\begin{equation*}
		\sup_{V_{\theta}} f \le \frac{C\mu^{\frac{4n+2}{p}}}{(1-\theta)^{\gamma_0/p}} \norm{f}_{\L^p(V_1)},
	\end{equation*}
	for all $\theta \in (0,1)$ and some $C = C(\delta,n,p,\tau)>0$ which proves the theorem for $p \in (0,2)$ by choosing $\theta = \sigma/\sigma'$. 
\end{proof}

Combining the energy estimates for $\beta \in (-1,0)$ and $\beta \in (0,\infty)$, we obtain a stronger result for weak solutions: for small $p \in \left(0,\frac{1}{\mu}\right)$ the constant in the $\L^p-\L^\infty$ inequality does not depend on $\mu$. Together with Proposition~\ref{prop:u-1}, this will be used to improve the constant in the Harnack inequality. Here, we consider cylinders in both directions of time, i.e. $\bar{Q}_r(t_0,x_0,v_0)$.

\begin{prop}  \label{prop:uLinfLpsmallp}
Let $\delta \in (0,1)$ and $\tau >0$. There exist constants $C = C(\delta, n,\tau)>0$ and $\gamma_0 = \gamma_0(n)>0$ such that for any kinetic cylinder $\bar{Q}_{(\sqrt{\tau}\,r,r,r)}(t_0,x_0,v_0) \subset \Omega_T$ with $(t_0,x_0,v_0) \in \Omega_T$ and $r>0$, and any nonnegative weak solution $f$ of equation \eqref{eq:kolharnack} with $S = 0$ on $\Omega_T$ the inequality
	\begin{equation*}
		\sup_{\bar{U}_\sigma} f \le   \left( \frac{C }{(\sigma'-\sigma)^{\gamma_0}} \right)^{\frac{1}{p}} \left( \frac{1}{\abs{\bar{U}_1}} \int_{\bar{U}_{\sigma'}} f^p \dx (t,x,v) \right)^{\frac1p}
	\end{equation*}  
	holds true for any $p \in \left(0,\frac{1}{\mu}\right)$ and for all $\delta \le \sigma < \sigma' \le 1$, where $\bar{U}_s := \bar{Q}_{(\sqrt{s \tau} \, r,s^{\frac{1}{3}}r,s r)}(t_0,x_0,v_0)$ for $s>0$. 
\end{prop}

\begin{proof}
	We introduce $\bar{V}_\rho = \bar{U}_{\rho \sigma'}$. Let $\varphi$ be a cutoff function to $\bar{U}_{ \sigma}$ with support in $\bar{U}_{\sigma'}$. Let $\beta \in (-1,\infty) \setminus \{ 0 \}$, then, in the weak formulation we test with $f^\beta\varphi^2$. Recall that $f$ is assumed to be a weak solution, hence we can use both the energy estimates from Proposition~\ref{prop:ulplg} and Theorem~\ref{thm:loclinf}. Moreover, both of these energy estimates remain true when replacing the cylinders $V_\rho$, $V_{\rho'}$ with $\bar{V}_\rho$. This new energy estimate combined with Theorem~\ref{thm:kinemb} gives		
	\begin{equation*}
			\norm{w}_{\L^{2\kappa}(\bar{V}_{\rho})} \le \frac{C\left(1+\mu \abs{\frac{{\beta+1}}{{\beta}}}\right)}{(\rho'-\rho)^{\frac{5}{2}}} \norm{w}_{\L^{2}(\bar{V}_{\rho'})} 
		\end{equation*}
		for some $C = C(\delta,n,\tau)>0$ and all $\beta \in (-1,\infty) \setminus \{ 0 \}$. Note that we can apply the Sobolev inequality of Theorem~\ref{thm:kinemb} as $w = f^{\frac{\beta+1}{2}}$ is either a subsolution or supersolution in $\bar{V}_{\rho'}$ depending on the value of $\beta$.

		Rewriting this in terms of $f$ and setting $\gamma = \beta+1 \in (0,\infty) \setminus \{ 1  \}$ we get 
		\begin{equation*}
			\norm{f}_{\L^{\gamma \kappa}(\bar{V}_{\rho})} \le \left( \frac{C\left(1+\mu \abs{\frac{\gamma}{\gamma-1}}\right)^2}{(\rho'-\rho)^5} \right)^{1/\gamma}\norm{f}_{\L^{\gamma}(\bar{V}_{\rho'})}
		\end{equation*}
		with $C = C(\delta,n,\tau)>0$. From here, we apply the De~Giorgi-Moser iteration of Lemma~\ref{lem:moser1smallp} to deduce the claim. 
\end{proof}

\section{The proof of the (weak) Harnack inequality following Moser-Bombieri-Giusti}
\label{sec:proofsmr}

\begin{remark}
	We consider only $S = 0$ in the following proofs. In view of \cite{clement_priori_2004} the case $S \neq 0$ follows by working with $\bar{f} = f+ \norm{S}_{\L^q}/\lambda$ as in the previous sections. 
\end{remark}

\begin{proof}[Proof of Theorem~\ref{thm:weakH}] Let us now prove the weak Harnack inequality. We may assume that $f \ge \epsilon$ by replacing $f$ with $f+\epsilon$ and considering the limit $\epsilon \to 0$ in the end. 

We reduce to the case $r = 1$ and $(t_0,x_0,v_0) = 0$ by the invariance of the statement with respect to the kinetic scaling and translation. 

For $0<\delta \le \sigma\le 1$ we set 
\begin{equation*}
	U_\sigma = \left(-\delta \tau - \frac{1}{2}(\sigma-\delta)\tau,0 \right) \times B_\sigma(0) \times B_\sigma(0)
\end{equation*}
and
\begin{equation*}
	U_\sigma' = \left( -2\tau,-(2-\delta) \tau + \frac{1}{2}(\sigma-\delta)\tau \right)\times B_\sigma(0) \times B_\sigma(0) 
\end{equation*}
In particular, $Q_- = U_\delta'$, $Q_+ = U_\delta$.  

First, we apply the weak $\L^1$-estimate implied by Theorem~\ref{thm:weakl1poin}, see Remark \ref{rem:L1log} (2), with parameters $\tau = 2\tau$, $\eta = \frac{1}{2}$ and $\iota = \frac{1}{2}(1-\delta)$ and $K_- = U_1'$ and $K_+ = U_1$. Let $R = R(\delta,\tau)>0$ be given accordingly. We have 
\begin{equation} \label{eq:whlog1}
 \abs{ \{(t,x,v) \in U_1 \colon c(f) - \log(f)>s \} } \le C \mu \abs{U_1}s^{-1}
\end{equation}
and 
\begin{equation} \label{eq:whlog2}
 \abs{ \{(t,x,v) \in U_1' \colon \log(f) - c(f)>s \} } \le C \mu \abs{U_1'}s^{-1}
\end{equation}
for all $s>0$ and some constant $C = C(\delta,n,\tau)>0$, where $c(f)$ is the constant given in \linebreak Theorem~\ref{thm:weakl1poin}.

Using Proposition~\ref{prop:u-1} we deduce
\begin{equation*}
	\sup_{U_{\sigma}} f^{-1} \le \left( \frac{C\abs{U_1}^{-1}}{(\sigma'-\sigma)^{\gamma_0}} \right)^{\frac{1}{\gamma}} \norm{f^{-1}}_{\L^\gamma(U_{\sigma'})}
\end{equation*}
for $\delta \le \sigma < \sigma' \le 1$, $\gamma \in (0,1]$ and some constants $C=C(\delta,\mu,n,\tau)>0$ and $\gamma_0 = \gamma_0(n)>0$. 

First, we apply the lemma of Bombieri-Giusti~\ref{lem:bomb} with $\beta_0 = \infty$. We already know that the first hypothesis is satisfied by $k f^{-1}$ for all $k>0$. Consider now $f_1 = f^{-1}\exp(c(f))$, then equation \eqref{eq:whlog1} shows that the second hypothesis of Lemma~\ref{lem:bomb} is satisfied for $f_1$. Hence, $\sup\limits_{U_\delta}f_1 \le M_1$ or equivalently
\begin{equation} \label{eq:proofWHlower}
	\exp(c(f)) \le M_1 \inf_{Q_+} f
\end{equation}
with some constant $M_1 = M_1(\delta,\mu,n)>0$.

\begin{figure}[H]
\centering
\tikzmath{\x = 0.2; \y = 0.2;}
\begin{tikzpicture}[scale=5]
  \draw[thick,->] (0, 0) -- (2,0) node[right] {$t$};
  \draw[thick,->] (0, 0) -- (0,0.3) node[left] {$(x,v)$};
  \draw[draw=black] (0.4-\x,0.3-\y) rectangle ++(1.5,1.2);
  \draw[draw=black] (0.4-\x,0.4-\y) rectangle ++(1.5,1);
  \draw[draw=black] (1.6-\x,0.7-\y) rectangle ++(0.3,0.4);
  \draw (1.6-\x,1.1-\y) node[anchor=north west] {$Q_+$};
  \filldraw (0.4-\x,0.9-\y) circle[radius=0.5pt];
  \draw (0.4-\x,0.9-\y) node[anchor= east] {$(0,0,0)$};
  \draw[draw=black] (0.4-\x,0.7-\y) rectangle ++(0.3,0.4);
  \draw (0.7-\x,1.1-\y) node[anchor=north east] {$Q_-$};
  \draw[draw=black] (0.4-\x,0.4-\y) rectangle ++(0.65,1);
  \draw[draw=black] (1.25-\x,0.4-\y) rectangle ++(0.65,1);
  \draw[draw=black] (0.4-\x,0.5-\y) rectangle ++(0.5,0.8);
  \draw[draw=black] (1.4-\x,0.5-\y) rectangle ++(0.5,0.8);
  \draw (1.05-\x,1.4-\y) node[anchor=north east] {$K_-$};
  \draw (1.4-\x,1.4-\y) node[anchor=north east] {$K_+$}; 
  \draw (0.9-\x,1.3-\y) node[anchor=north east] {$U_\sigma'$};
  \draw (1.55-\x,1.3-\y) node[anchor=north east] {$U_\sigma$};
  \draw [thick] (0.4-\x, 0) -- ++(0, -.05) ++(0, -.15) node [below, outer sep=0pt, inner sep=0pt] {\small\(-2\tau \)};
   \draw [thick] (0.7-\x, 0) -- ++(0, -.05) ++(0, -.15) node [below, outer sep=0pt, inner sep=0pt] {\small\(-(2-\delta) \tau \)};
	\draw [thick] (1.15-\x, 0) -- ++(0, -.05) ++(0, -.15) node [below, outer sep=0pt, inner sep=0pt] {\small\(\tau \)};
      \draw [thick] (1.6-\x, 0) -- ++(0, -.05) ++(0, -.15) node [below, outer sep=0pt, inner sep=0pt] {\small\( -\delta \tau  \)};
      \draw [thick] (1.9-\x, 0) -- ++(0, -.05) ++(0, -.15) node [below, outer sep=0pt, inner sep=0pt] {\small\(0 \)};
  \draw [decorate,decoration={brace,amplitude=10pt,mirror,raise=4pt},yshift=0pt]
(1.9-\x,0.7-\y) -- (1.9-\x,1.1-\y) node [black,midway,xshift=0.8cm] {\footnotesize
${B}_{\delta}$};
\draw [decorate,decoration={brace,amplitude=10pt,mirror,raise=4pt},yshift=0pt]
(2.1-\x,0.5-\y) -- (2.1-\x,1.3-\y) node [black,midway,xshift=0.8cm] {\footnotesize
${B}_{\sigma}$};
\draw [decorate,decoration={brace,amplitude=10pt,mirror,raise=4pt},yshift=0pt]
(2.3-\x,0.4-\y) -- (2.3-\x,1.4-\y) node [black,midway,xshift=0.8cm] {\footnotesize
${B}_{1}$};
\draw [decorate,decoration={brace,amplitude=10pt,mirror,raise=4pt},yshift=0pt]
(2.5-\x,0.3-\y) -- (2.5-\x,1.5-\y) node [black,midway,xshift=0.8cm] {\footnotesize
$B_R$};
\end{tikzpicture}
\caption{The cylinders $Q_-,Q_+,U_\sigma',U_\sigma,K_-,K_+$ and the spatial balls $B_\delta,B_\sigma,B_1,B_R$ in the proof of Theorem~\ref{thm:weakH} for $(t_0,x_0,v_0) = 0$ and $r = 1$.}
\end{figure}
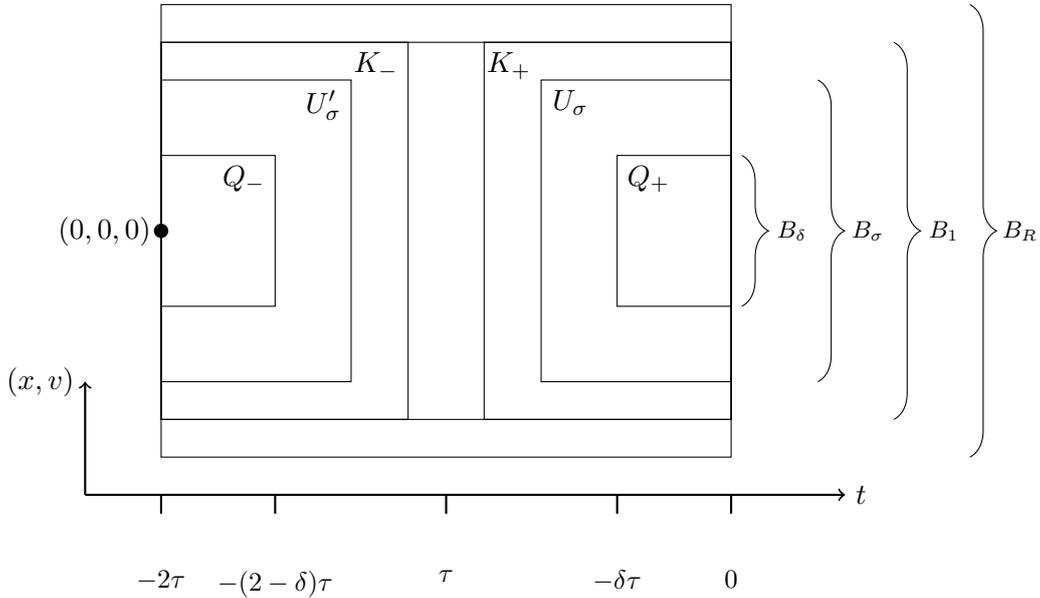

On the other hand, Proposition~\ref{prop:ulplg} yields for all $p \in (0,\kappa)$
\begin{equation*}
	\norm{f}_{\L^p(U_{\sigma}')} \le \left( \frac{C\abs{U_1'}^{-1}}{(\sigma'-\sigma)^{\gamma_0}} \right)^{\frac{1}{\gamma}-\frac{1}{p}} \norm{f}_{\L^\gamma(U_{\sigma'}')}
\end{equation*}
for all $\delta \le \sigma < \sigma' \le 1$ and $0<\gamma<\frac{p}{\kappa}$ with $C = C(\delta,\mu,n,p)>0$ and $\gamma_0 = \gamma_0(n)>0$. Consequently, the first hypothesis of Lemma~\ref{lem:bomb} is satisfied for $f_2 = f\exp(-c(f))$. Moreover, we see that the inequality in \eqref{eq:whlog2} implies that the second hypothesis of Lemma~\ref{lem:bomb} is satisfied for $f_2$ as well. Consequently, $\norm{f_2}_{\L^p(U_\delta')} \le M_2 \abs{U_1'}^{1/p}$ with $M_2 = M_2(\delta,\mu,n,p)>0$ or equivalently
\begin{equation}\label{eq:proofWHlp}
	\norm{f}_{\L^p(U_\delta')} \le M_2 \abs{U_1'}^{1/p} \exp(c(f))
\end{equation}

Combining \eqref{eq:proofWHlower} and \eqref{eq:proofWHlp} we obtain
\begin{equation*}
	\abs{U_1'}^{-1/p}\norm{f}_{\L^p(Q_-)} \le M_1M_2 \inf_{Q_+} f,
\end{equation*}
which proves the weak Harnack inequality. 

To obtain the precise dependency of the constant on $\mu$ for small $p$ we consider $\tilde{f} = f^{\frac{1}{\mu}}$ which is a supersolution, too (for $\alpha \in (0,1)$, $s \mapsto s^\alpha$, $s>0$ is concave and non-decreasing). Moreover, in three ingredients, i.e.\ the weak $\L^1$-estimate of Theorem~\ref{thm:weakl1poin}, Proposition~\ref{prop:u-1} and Proposition~\ref{prop:ulplg}, the constants are independent of $\mu$. Arguing as above, we obtain 
	\begin{equation*}
    	\abs{U_1'}^{-1/\tilde{p}}\norm{\tilde{f}}_{\L^{\tilde{p}} (Q_-)} \le \tilde{C} \inf_{Q_+} \tilde{f}, 
 	\end{equation*}
  	for $\tilde{p} \in (0,1)$ and $\tilde{C}(\delta,n,\tau,\tilde{p} )>0$. Writing $p = \frac{\tilde{p} }{\mu}$ we have 
 	\begin{equation*}
    	\abs{U_1'}^{-1/p}\norm{{f}}_{\L^{{p}} (Q_-)} \le \tilde{C}(\delta,n,\tau,\mu p)^\mu \inf_{Q_+} f ,
	\end{equation*}
  	for all $p \in \left(0,\frac{1}{\mu}\right)$.  
\end{proof}

\begin{proof}[Proof of Theorem~\ref{thm:harnack}]

The Harnack inequality for weak solutions is now an easy consequence of the weak Harnack inequality for supersolutions and the local boundedness of solutions for subsolutions, i.e.\ Theorem~\ref{thm:loclinf} and Theorem~\ref{thm:weakH}. 

However, if we want to improve the dependency on $\mu$ we need to take a closer look. Indeed, combining the weak Harnack inequality for small $p \in \left(0,\frac{1}{\mu}\right)$ with the $\L^p$--$\L^\infty$ estimate for weak solutions and small $p \in \left(0,\frac{1}{\mu}\right)$ as in Proposition~\ref{prop:uLinfLpsmallp}, e.g. at $p = \frac{1}{2\mu}$, we obtain the Harnack inequality with optimal dependency on the ellipticity constants, see Section~\ref{sec:optC}. 

Alternatively, this can also be proven directly as a consequence of Theorem~\ref{thm:weakl1poin}, Proposition~\ref{prop:u-1}, Proposition~\ref {prop:uLinfLpsmallp} and two applications of the lemma of Bombieri and Giusti with $\beta_0 =  \infty$ as done in \cite{moser_pointwise_1971}.
\end{proof}

\begin{proof}[Proof of the H\"older continuity, Theorem~\ref{thm:holder}]
	It is classical to deduce the H\"older continuity from a (weak) Harnack inequality. We refer to \cite{bonforte_explicit_2020} for a detailed exposition of the parabolic case, where they also obtain explicit constants. This proof transfers over to the kinetic case, keeping in mind the definitions of kinetic cylinders, inducing the different notion of continuity as explained in \eqref{eq:kinholder}.
\end{proof}

\appendix
\addcontentsline{toc}{section}{Appendix}

\addtocontents{toc}{\SkipTocEntry}
\section{A guide on how to make the calculations rigorous}
\label{sec:rigorous}

We recall the notion of weak solutions as given in Definition~\ref{def:weaksol}. We treat $S = 0$ for simplicity. A function $f \in \L^\infty((0,T);\L^2(\Omega_x \times \Omega_v)) \cap \L^2((0,T)\times \Omega_x;\Hdot^1( \Omega_v))$ is a weak (sub-, super-) solution to the Kolmogorov equation \eqref{eq:kolharnack} on $\Omega_T$ if for all $\varphi \in\C^\infty_c(\Omega_T)$ with $\varphi \ge 0$ we have
	\begin{equation} 
		\int_{\Omega_T} -f (\partial_t + v \cdot \nabla_x ) \varphi+  \langle \fra \nabla_v f, \nabla_v \varphi \rangle \dx (t,x,v) =  \, (\le, \, \ge) \, 0.
	\end{equation}

In the trajectorial proofs of Section~\ref{sec:sobolev} and Section~\ref{sec:weakL1log}, we differentiate the weak solution along a trajectory and for this we need differentiability in $t,x$, which may not be available for this notion of solution. Moreover, we need to be able to test with the solution $f$ and powers thereof, which is yet another technical obstacle. To make the arguments of the present article rigorous, we need to consider an approximation of $f$ which has more regularity in time and position so that $\partial_t  + v \cdot \nabla_x$ can be understood.

In time, we regularise with Steklov averages, see \cite{MR2865434,felsinger_local_2013} in the parabolic setting. One could also work with a Yosida approximation, see \cite{winkert_global_2016}. Another good reference is  \cite{aronson_local_1967}. 

For $0<h<T$ and $f \in \L^\infty((0,T);\L^2(\Omega_x \times \Omega_v))$ we set
\begin{equation*}
	S_h f \colon [0,T-h] \to \L^2(\Omega_x\times \Omega_v), \quad S_h f(t) = \frac{1}{h}\int_{t}^{t+h} f(s) \dx s,
\end{equation*}
the Steklov average of $f$. This construction ensures that $S_h f \in \H^1((0,T-h);\L^2(\Omega_x\times \Omega_v))$.

For the regularisation in the position variable $x$ we choose any radial cutoff function $0 \le \psi \in \C_c^\infty(B_1(0))$ with $\psi = 1$ in $B_{\frac{1}{2}}(0) \subset \R^n$ and $\int \psi \dx x = 1$, we set $\psi_\delta = \delta^{-n}\psi(\cdot/\delta)$ ($\delta>0$) and define the operator
\begin{equation*}
	J_\delta \colon \L^1(\Omega_x^\delta) \to \L^1(\Omega_x) \quad J_\delta u = \psi_\delta \ast u,
\end{equation*}
where $\Omega_x^\delta = \{ x \in \R^n \colon \mathrm{dist}(x,\Omega_x^c) > \delta \} $. 

We aim to deduce the equivalence of Definition~\ref{def:weaksol} to the following regularised weak formulation. 

\begin{definition} \label{def:weaksolreg}
	We say $f \in \L^\infty((0,T);\L^2(\Omega_x \times \Omega_v)) \cap \L^2((0,T)\times\Omega_x;\Hdot^1( \Omega_v))$ is a weak (sub-, super-) solution to the Kolmogorov equation \eqref{eq:kolharnack} on $\Omega_T$ if for all $\delta>0$, any $0<h<T$ and every $\varphi \in \H^1_v(\Omega_v)$ with $\varphi \ge 0$ we have
	\begin{equation*}
		\int_{\Omega_v}  \varphi(\partial_t + v \cdot \nabla_x)S_h J_\delta f +  \langle S_h J_\delta(\fra \nabla_v f), \nabla_v \varphi \rangle \dx v =  \, (\le, \, \ge) \,0 \mbox{ for a.e. } (t,x) \in (0,T-h) \times \Omega_x^{\delta}.
	\end{equation*}
\end{definition}

\subsubsection*{Equivalence of weak formulations}

We give the argument why this weak formulation is equivalent to the one from Definition~\ref{def:weaksol}. We treat only the case of weak solutions and note that for super- or subsolutions, one simply needs to replace some equalities with inequalities. 

By a density argument, we can extend the class of test functions in Definition~\ref{def:weaksol} to all 
\begin{equation*}
	\varphi \in \H^1((0,T);\L^2(\Omega_x \times \Omega_v)) \cap \L^2((0,T)\times \Omega_x ;\H^1( \Omega_v))
\end{equation*}
with $\varphi |_{t = 0} = \varphi|_{t = T} = 0$.

We choose a test function $\varphi \in  \H^1((-h,T);\L^2(\Omega_x \times \Omega_v)) \cap \L^2((-h,T) \times \Omega_x;\H^1(\Omega_v))$ with essential support in a subset of $(0,T-h) \times \Omega_x^{\delta} \times \Omega_v$. We consider $\eta = S_{\bar{h}}J_\delta \varphi $, defined by 
\begin{equation*}
	S_{\bar{h}} J_\delta \varphi = \frac{1}{h}\int_{t-h}^t J_\delta\varphi(s) \dx s,
\end{equation*}
which is an admissible test function.

We obtain 
\begin{align*}
	 0 & = \int_{\Omega_T} -f (\partial_t + v \cdot \nabla_x) S_{\bar{h}} J_\delta \varphi+  \langle \fra \nabla_v f, \nabla_v S_{\bar{h}} J_\delta \varphi \rangle \dx (t,x,v) \\
	 & = \int_{\Omega_T} -f S_{\bar{h}}J_\delta(\partial_t + v \cdot \nabla_x) \varphi+  \langle \fra \nabla_v f, S_{\bar{h}} J_\delta\nabla_v  \varphi \rangle \dx (t,x,v) \\
	 &= \int_{\Omega_{T-h}} -S_h J_\delta f (\partial_t + v \cdot \nabla_x)  \varphi+  \langle S_h J_\delta (\fra \nabla_v f), \nabla_v \varphi \rangle \dx (t,x,v) \\
	 &= \int_{\Omega_{T-h}} \varphi(\partial_t + v \cdot \nabla_x) S_h J_\delta f   +  \langle S_h J_\delta (\fra \nabla_v f), \nabla_v \varphi \rangle \dx (t,x,v).
\end{align*}
We used the symmetry of $J_\delta$, that $J_\delta$ commutes with any derivative, that $S_h$ and $S_{\bar{h}}$ commute with $\nabla_x, \nabla_v$ and $\partial_t S_{\bar{h}} = S_{\bar{h}} \partial_t $. The third equation is an application of the theorem of Fubini. As $S_h J_\delta f \in \H^1((0,T-h);\L^2(\Omega_x^{\delta} \times \Omega_v)) \cap \L^2((0,T-h)\times \Omega_x^{\delta} ;\Hdot^1(\Omega_v))$ we can perform the integration by parts in the last equality noting that the boundary terms vanish by choice of $\varphi$. 

By yet another density argument we obtain that $f$ is a weak solution to the Kolmogorov equation \eqref{eq:kolharnack} if and only if for all $\delta>0$, any $0<h<T$ and every $\varphi \in \C([0,T-h];\L^2(\Omega_x^{\delta} \times \Omega_v)) \cap \L^2((0,T-h)\times \Omega_x^{\delta} ;\Hdot^1(\Omega_v))$ with $\varphi \ge 0$ and $\varphi|_{t = 0} = \varphi|_{t = T} = 0$ we have
	\begin{equation}
	\int_{\Omega_{T-h}} \varphi(\partial_t + v \cdot \nabla_x) S_h J_\delta f   +  \langle S_h J_\delta (\fra \nabla_v f), \nabla_v \varphi \rangle \dx (t,x,v) = 0.
	\end{equation}

To finally obtain the weak formulation in Definition~\ref{def:weaksolreg}, we choose any $\varphi \in \H^1_v(\Omega_v)$ and consider suitable localisations in $(t,x)$ to apply the Lebesgue differentiation theorem to deduce the weak formulation in Definition~\ref{def:weaksolreg}. We note that the null set depends on the choice of the test function $\varphi$. 

The other direction follows by integrating against the desired test function, undoing the above calculations, and then considering the limits $\delta \to 0$ and $h \to 0$.

\subsubsection*{Arguments for the main results} Let us now explain how to use this regularised weak formulation in the proofs of the main results. In general, we note that the restriction to $(0,T-h) \times \Omega_x^{\delta} \times \Omega_v$ is no problem, as we can always rescale to a unit scale far away from the boundary and then go back by translation and scaling invariances. 

In the proof of Theorem~\ref{thm:weakl1poin} we consider $S_h J_\delta f$ instead of $f$. This allows us to write the trajectorial derivative at least in an integrated form. The use of the supersolution property and the subsequent integration by parts needs to be understood as testing the weak formulation in Definition~\ref{def:weaksolreg} with a test function evaluated at the inverse of the trajectory multiplied with $(S_h J_\delta f)^{-1}$ (see below for the justification of this). The same strategy works in the proof of Theorem~\ref{thm:kinemb}.

Whenever we test with powers of the solution, extra care needs to be taken. For that we introduce a version of the function $s \mapsto s^{\beta}$ with linear growth (at infinity) as
\begin{equation*}
	G(s) = \begin{cases}
		\min \{ s^\beta,\theta^{\beta-1} s \} & \beta \ge 1 \\
		s^{\beta} & \beta < 1 \mbox{ and } \beta \neq 0
	\end{cases} 
\end{equation*}
for some constant $\theta>0$. On $(\epsilon,\infty)$ this function is Lipschitz with Lipschitz constant depending on $\beta,\theta$ and $\epsilon$.

In the proof of Proposition~\ref{prop:u-1}, Proposition~\ref{prop:ulplg}, Proposition~\ref{prop:uLinfLpsmallp} and Theorem~\ref{thm:loclinf} we test the regularised weak formulation with $\varphi^2 G(S_h J_\delta f)$, where $\varphi$ is a suitable cutoff function and for $\beta \in (-\infty,-1)$, $\beta \in (-1,0)$, $\beta \in (-1,\infty)  \setminus \{ 0 \}$ and $\beta \in [1,\infty)$ respectively. For the first three cases, we assume $f \ge \epsilon>0$ and let $\epsilon \to 0$ in the end, which is an important qualitative assumption. In the case of subsolutions, we may drop this assumption as $G$ is Lipschitz on $[0,\infty)$ for $\beta \ge 1$. By our choice of mollification operators, we have $S_h J_\delta f \ge \epsilon$.

After testing and integration, we need to prove
\begin{equation*}
	\int_{\Omega_T} \langle S_hJ_\delta (\fra \nabla_v f) , \nabla_v \left( \varphi^2 G( S_hJ_\delta f) \right) \rangle \dx (t,x,v) \to \int_{\Omega_T} \langle \fra \nabla_v f , \nabla_v \left( \varphi^2 G(f) \right) \rangle \dx (t,x,v) 
\end{equation*}
as $h \to 0$ and $\delta \to 0$. We note that $\fra \nabla_v f \in \L^2(\Omega_T)$, whence $S_h J_\delta (A \nabla_v f) \to \fra \nabla_v f$ in $\L^2(\Omega_T)$ for $h \to 0$ and $\delta \to 0$. We need to prove that the other term converges in $\L^2(\Omega_T)$. For that, we calculate
\begin{equation*}
	\nabla_v \left( \varphi^2 G(S_h J_\delta f) \right)  = 2 \varphi G(S_h J_\delta f)  \nabla_v\varphi + G'(S_h J_\delta f) \varphi^2 \nabla_v \left( S_h J_\delta f \right).
\end{equation*}
By construction, we have
\begin{equation*}
	\abs{G(S_h J_\delta f) - G(f)} \le C \abs{S_h J_\delta f- f}
\end{equation*}
for some constant $C>0$ as $G$ is Lipschitz. Hence, $2 \varphi  G(S_h J_\delta f) \nabla_v\varphi  \to 2 \varphi \nabla_v\varphi G(f) $ in $\L^2(\Omega_T)$. 

It remains to treat the last term. We write
\begin{align*}
	G'(S_h J_\delta f) \varphi^2 \nabla_v \left(S_h J_\delta f \right) - G'(f) \varphi^2 \nabla_v f &= G'(S_h J_\delta f) \varphi^2 S_h J_\delta\nabla_v f - G'(S_h J_\delta f) \varphi^2 \nabla_v f \\
	&\hphantom{=}+G'(S_h J_\delta f) \varphi^2 \nabla_v f - G'(f) \varphi^2 \nabla_v f
\end{align*}
and deduce
\begin{align*}
	\abs{G'(S_h J_\delta f) \varphi^2 \nabla_v \left(S_h J_\delta f \right) - G'(f) \varphi^2 \nabla_v f}^2 &\le \abs{G'(S_h J_\delta f)}^2\abs{S_h J_\delta \nabla_v f -\nabla_v f}^2\varphi^2  \\
	&\hphantom{=}+ \abs{G'(S_h J_\delta f)- G'(f)}^2 \varphi^2 \abs{\nabla_v f}^2.
\end{align*}

As $G'$ is a bounded function on $(\epsilon,\infty)$ the first term converges in $\L^1(\Omega_T)$ and for the second we have the $\L^1(\Omega_T)$-convergence by dominated convergence. Indeed, $S_h J_\delta f \to f$ implies that we have almost everywhere convergence along a subsequence. From here, we derive the energy estimate and then let $\theta \to \infty$. 

For the term involving the time derivative, we argue similarly and employ the chain rule, integrate in time and then drop this nonnegative term, as for our approach, we only need a bound on the gradient. Lastly, for the term involving $v \cdot \nabla_x$, we perform an integration by parts, noting that the boundary terms vanish by choice of the test function. This leads to a remaining $ \L^2$ term, which we keep on the right-hand side of our estimates. 

\begin{remark}
    In view of the above considerations we could also loosen the regularity assumption to $f \in \L^1_{\mathrm{loc}}((0,T) \times \Omega_x;\L^2(\Omega_v))\cap \L^2((0,T)\times \Omega_x;\Hdot^1( \Omega_v))$.
\end{remark}

\addtocontents{toc}{\SkipTocEntry}
\section{Admissible exponents in the weak Harnack inequality} \label{sec:opt}

\begin{lemma}
	The range for the exponent $p \in (0,1+\frac{1}{2n})$ in the weak Harnack inequality, Theorem~\ref{thm:weakH} is sharp. 
\end{lemma}

\begin{proof}
	We consider the case $\fra(t,x,v) = \id_n$. The fundamental solution in this case can be explicitly calculated as
    \begin{equation*}
        \Gamma(t,x,v) = \frac{\sqrt{3}^n}{(2\pi)^n t^{2 n}} \exp \left(-\frac{1}{t}|v|^2+\frac{3}{t^2}\langle v, x\rangle-\frac{3}{t^3}|x|^2\right).
    \end{equation*}
    Let $k \in \N$. We introduce the truncated fundamental solution of the Kolmogorov equation as
	\begin{equation*}
		\Phi_k(t,v,x) = \begin{cases}
		\Gamma(t,x,v) &: t > \frac{1}{k} \\
		\Gamma \left(\frac{1}{k},x,v \right) &: t \le \frac{1}{k}.
		\end{cases}
	\end{equation*}
	This is a weak supersolution to the Kolmogorov equation in a suitable cylinder $Q = (0,1) \times B_R(0) \times B_R(0)$, where $R = R(n)>0$. A straightforward calculation shows that 
	\begin{equation*}
		\inf_{(t_1,1) \times B_R(0) \times B_R(0)} \Phi_k \le \C_1t_1^{-2n}
	\end{equation*}
	for $0<t_0<t_1 <1 $ and some constant $C_1 = C_1(n)>0$ and $k$ large enough. Let $\delta >0$ and $r_x,r_v \in (\delta,R)$. Calculating the integral on the left-hand side of the weak Harnack inequality, we estimate
	\begin{equation*}
		\int_0^{t_0} \int_{B_{r_x}(0)} \int_{B_{r_v}(0)} \abs{\Phi_k}^p \dx v \dx x \dx t \ge \C_2 \int_{1/k}^{t_0} t^{-2np+2n} \dx t 
	\end{equation*}
	where $C_2 = C_2(\delta,n)>0$ is some constant. The latter integral converges if and only if $p<1+\frac{1}{2n}$, which shows the claim by letting $k \to \infty$.
\end{proof}

\addtocontents{toc}{\SkipTocEntry}
\section{Optimal constant in the Harnack inequality} \label{sec:optC}
The example of J.~Moser \cite[page 729]{moser_pointwise_1971} works in the kinetic setting, too, to prove optimality. Indeed, every weak solution to the parabolic diffusion problem is an $x$-independent solution of the Kolmogorov equation \eqref{eq:kolharnack}. For the reader's convenience, we recall the example of J.~Moser. In $\Omega = (-2,2) \times B_{\pi/2}(0) \subset \R^{1+2}$ we consider the partial differential equation
\begin{equation*}
	\partial_t f =\lambda \partial_{v_1}^2 f+\Lambda \partial_{v_2}^2 f,
\end{equation*}
for functions $f \colon \Omega \to \R$ and $0 < \lambda \le \Lambda$. This equation is solved by the positive function
\begin{equation*}
	f(t,v)=e^{\left( \frac{1}{4\lambda}-\Lambda \right) t - \frac{1}{2\lambda} v_1} \cos v_2,
\end{equation*}
which can be verified by a direct calculation. Taking logarithms, we obtain 
$$
\log \frac{f(0,(0,0))}{f(1,(1,0))}=\Lambda+\frac{1}{4 \lambda}>\frac{1}{4}\left(\Lambda+\lambda^{-1}\right).
$$
Interpreting $f$ as a function of $(t,x,v)$ constant in $x$, this function solves the Kolmogorov equation $(\partial_t  +v \cdot \nabla_x) f =\lambda \partial_{v_1}^2 f+\Lambda \partial_{v_2}^2 f$, which shows that the optimal dependency of the Harnack constant on $\mu$ is exponential.

\addtocontents{toc}{\SkipTocEntry}
\section{Noncriticality of smooth kinetic trajectories} \label{sec:smoothnotwork}

We explain why smooth (for small $r$) trajectories cannot be used to obtain criticality. For illustration purposes, we treat $n = 1$ and we only connect $(t_0,x_0,v_0) = (0,0,0)$ with some given point $(t_1,x_1,v_1) \in \R^{1+2n}$ with $t_1 \neq 0$.

We furthermore assume $g_1(r) = r^{\frac{3}{2}}$ and thus $g_1'(r) = \frac{3}{2}r^{\frac{1}{2}}$ in the setting of Section~\ref{sec:connecting}. Consider a function $g_2(r)$ with 
\begin{equation} \label{eq:smoothg1}
	g_2(r) = \sum_{k = 0}^N b_k r^{e_k} + \O(r^L)
\end{equation}
close to $r = 0$ with $e_0 = \frac{3}{2}$ and $(e_1,\dots,e_k) \in \left(\frac{3}{2},\infty\right)^N$ in increasing order and coefficients $(b_1,\dots,b_N) \in \R^{N+1}$. Note that the following arguments also work if one assumes a more general structure of $g_1(r)$, for example, in the spirit of the structure of $g_2(r)$. We leave the details to the reader. 

We need $e_i  \ge \frac{3}{2}$ because any term of lower order would lead to a singularity of the forcing $\dot{\gamma}_v$, worse than $r^{-\frac{1}{2}}$, which we could not control, see Remark~\ref{rem:difficulty}. 

\begin{remark} \label{rem:difftime}
  \begin{enumerate}
  \item We emphasise that $r^{\frac{3}{2}}\cos(\log(r))$ and $r^{\frac{3}{2}}\sin(\log(r))$ do not admit such an expansion. Indeed, for every $g_2(r)$ of the form in \eqref{eq:smoothg1} the forcing has the asymptotics $\lim\limits_{r\to 0^+} r^{\frac{1}{2}}g_2''(r) = b_0$. 
  \item Note that the unusual notion of smoothness comes from the fact that we are working with singular forcings of order $r^{-\frac{1}{2}}$. Think of it as a Taylor expansion with root terms. If one would consider a different time scaling like $\gamma_t(r) = t_0+ r^2(t_1-t_0)$, the same argument would show that as soon as $g_1,g_2$ admit a Taylor expansion at $0$, the criticality (in this case $r^{-1}$) is not attained.  
  \end{enumerate}
\end{remark}

Again we are interested in $\A_{t_1}(r) = \D_{t_1}\W(r) \W(1)^{-1} \D_{t_1}^{-1}$ with 
\begin{equation*}
  \W(r) = \begin{pmatrix}
    r^{\frac32} & g_2(r) \\
    \frac{3}{2}r^{\frac12} & g_2'(r) 
  \end{pmatrix}.
\end{equation*}

First we observe that we need $g_2'(1) \neq \frac{3}{2} g_2(1)$ so that $\W(1)$ is invertible. 

We calculate 
\begin{align*}
  \det \W(r)
  &= g_2'(r)r^{\frac32} - \frac{3}{2}g_2(r)r^{\frac12} = \sum_{k = 0}^N \left(e_k-\frac{3}{2}\right)b_k r^{e_k+\frac{1}{2}}+ \O(r^{L+\frac{1}{2}}) \\
  &=  \sum_{k = 1}^N \left(e_k-\frac{3}{2}\right)b_k r^{e_k+\frac{1}{2}}+ \O(r^{L+\frac{1}{2}}) = \O(r^{e_1+\frac12}).
\end{align*}
and emphasise the cancellation of the $r^{2}$-term. 

Next, we study the second column of $\A_{t_1}(r)^{-1} = \D_{t_1} \W(1) \W(r)^{-1} \D_{t_1}^{-1}$, i.e. 
\begin{equation*}
  \displaystyle	
  t_1 \left(-g_2(r)+r^{\frac32}g_2(1)\right)  = -t_1\left({(b_0-g_2(1)) r^{\frac32} + \sum\limits_{k = 1}^N b_kr^{e_k} + \O(r^L) } \right)  \end{equation*}
and
\begin{equation*}
  -\frac{3}{2}g_2(r)+r^{\frac32}g_2'(1)  = -\left({\left(\frac{3}{2}b_0-g_2'(1)\right) r^{ \frac32} + \sum\limits_{k = 1}^N b_kr^{e_k} + \O(r^L)} \right).
\end{equation*}
Recall that $\det \A_{t_1}(r) = \O(r^{e_1+\frac{1}{2}})$, which implies that in order to reach criticality we need to have $b_0 = g_2(1)$ and $b_0 = \frac{2}{3}g_2'(1)$ but this contradicts the assumption $g_2'(1) \neq \frac{3}{2} g(1)$, which we made so that $\W(1)$ is invertible. This shows that a function with the asymptotic behaviour for $r \sim 0$ as that in \eqref{eq:smoothg1} can never reach criticality.

\addtocontents{toc}{\SkipTocEntry}
\section{De~Giorgi-Moser iterations and the lemma of Bombieri and Giusti}
\label{sec:abstract}
We state here the De~Giorgi-Moser iterations as abstract tools following \cite[Section~2.1]{zacher_weak_2013}, \cite[Section2]{clement_priori_2004} and \cite{bonforte_explicit_2020}.

Let $(\Omega,\A,\nu)$ be a finite measure space and $U_\sigma \subset \Omega$, $0<\sigma \le 1$ a family of measurable sets with $U_{\sigma} \subset U_{\sigma'}$ for $\sigma \le \sigma'$. For $p \in (0,\infty)$ we denote by $\L^p(U_\sigma):= \L^p(U_\sigma,\dx \nu)$ the Lebesgue space of all measurable functions $f \colon U_\sigma \to \R$ endowed with the (semi-) norm $\norm{f}_{\L^p(U_\sigma)}:= \left( \int_{U_\sigma} \abs{f}^p \dx \nu \right)^{1/p}$.

The first De~Giorgi-Moser iteration goes all the way to infinity. We present three different but similar variants. 

\begin{lemma} \label{lem:moser1}
	Let $C \ge 1$, $\gamma_1,\gamma_2 >0$, $\kappa > 1$, $\mu \ge 1$ and ${p} >0 $. We consider a $\nu$-measurable function $f \colon U_1 \to \R$ satisfying
	\begin{equation} \label{eq:first}
		\norm{f}_{\L^{\alpha\kappa}(U_{\sigma})} \le \left( \frac{C(1+\mu\alpha)^{\gamma_1}}{(\sigma'-\sigma)^{\gamma_2}} \right)^{\frac{1}{\alpha}} \norm{f}_{\L^\alpha(U_{\sigma'})}
	\end{equation} 
	for $0 < \sigma<\sigma' \le 1$ and $\alpha \in [p,\infty)$. Then, there exist constants
	\begin{equation*}
		M = M(C,\gamma_1, \gamma_2,\kappa,{p},\mu)=C^{\frac{\kappa}{\kappa-1}}\left(1+ \mu {p}\right)^{\gamma_1{\frac{\kappa}{\kappa-1}}}{2}^{\gamma_2 \frac{\kappa^2}{(\kappa-1)^2}} \kappa^{\gamma_1 \frac{\kappa}{(\kappa-1)^2}}>0
	\end{equation*}
	and
	\begin{equation*}
		\tilde{\gamma} = \tilde{\gamma}(\gamma_2,\kappa) = \frac{\kappa}{\kappa-1} \gamma_2 >0
	\end{equation*}
	such that 
	\begin{equation*}
		\sup_{U_\delta} \abs{f} \le \left( \frac{M}{(1-\delta)^{\tilde{\gamma}}} \right)^{\frac{1}{p}} \norm{f}_{\L^p(U_1)} 
	\end{equation*}
	for all $\delta \in (0,1)$. In particular, if $p \in (0,\bar{p}]$ for some $\bar{p}>0$ with $\bar{p}\mu  \lesssim 1$, then $M$ can be chosen independently of $p$ and $\mu$.
\end{lemma} 

\begin{proof}
	This is a slight variation of \cite[Lemma 2.1]{zacher_weak_2013}.
\end{proof}

\begin{lemma} \label{lem:moser1smallp}
	Let $C \ge 1$, $\gamma_1,\gamma_2 >0$, $\kappa > 1$ and $\mu \ge 1$. Suppose that $f \colon U_1 \to \R$ is a $\nu$-measurable function with
	\begin{equation*}
		\norm{f}_{\L^{\alpha \kappa}(U_{\sigma})} \le \left( \frac{C\left(1+\mu \abs{\frac{\alpha}{{\alpha-1}}}\right)^{\gamma_1}}{(\sigma'-\sigma)^{\gamma_2}} \right)^{\frac{1}{\alpha}}\norm{f}_{\L^{\alpha }(U_{\sigma'})} 
	\end{equation*}
	for $0<\sigma <\sigma' \le 1$ and $\alpha >0$, $\alpha \neq 1$. Then, there exists constants $M = M(C,\gamma_1,\gamma_2,\kappa)>0$ and $\gamma_0 = \gamma_0(\gamma_2,\kappa)>0$ such that 
	\begin{equation*}
		\sup_{U_\delta} \abs{f} \le \left( \frac{M}{(1-\delta)^{\gamma_0} } \right)^{\frac1p} \norm{f}_{\L^p(U_1)} 
	\end{equation*}
	for all $\delta \in (0,1)$ and any $p \in \left(0,\frac{1}{\mu}\right)$. 
\end{lemma}

\begin{proof}
	This is the iteration procedure originally used in \cite[Lemma 1]{moser_pointwise_1971} and a very detailed exposition can be found in \cite[Lemma 2]{bonforte_explicit_2020}.
\end{proof}

The second De~Giorgi-Moser iteration does not go all the way to infinity but stops at a given exponent $p_0$. It is fundamental in the proof of the weak Harnack inequality.  

\begin{lemma} \label{lem:moser2}
	Suppose that $\nu(U_1) \le 1$. Let $C\ge 1$, $\gamma_1,\gamma_2 >0$, $\kappa >1$, $\mu \ge 1$ and $0 <p_0<\kappa$. Assume that $f \colon U_1 \to \R$ is a measurable function satisfying
	\begin{equation*}
		\norm{f}_{\L^{\alpha \kappa}(U_{\sigma})} \le \left( \frac{C(1+\mu \frac{\alpha}{1-\alpha})^{\gamma_1}}{(\sigma'-\sigma)^{\gamma_2}} \right)^{\frac{1}{\alpha}} \norm{f}_{\L^\alpha(U_{\sigma'})}
	\end{equation*}
	for all $0<\sigma'<\sigma \le 1$ and $0 < \alpha \le \frac{p_0}{\kappa}<1$. Then, there exist constants 
	$$M = M(C,\gamma_1,\gamma_2,\mu,p_0,\kappa)=2^{\frac{\gamma_2 \kappa^3(1+\kappa)}{(\kappa-1)^3}}C^{\frac{\kappa}{\kappa-1}(1+\kappa)}\left(1+ \mu \frac{p_0}{\kappa-p_0}\right)^{\gamma_1 \frac{\kappa}{\kappa-1}(1+\kappa)}>0$$ 
	and 
	$$\gamma_0 = \gamma_0(\gamma,\kappa)=(1-\delta)^{\frac{\gamma_2 \kappa}{\kappa-1}(1+\kappa)}>0$$
	such that 
	\begin{equation*}
		\norm{f}_{\L^{p_0}(U_{\delta})} \le \left( \frac{M}{(1-\delta)^{\gamma_0}} \right)^{\frac1p-\frac{1}{p_0}} \norm{f}_{\L^p(U_1)}
	\end{equation*}
	for all $\delta \in (0,1)$ and any $p \in (0,p_0/\kappa]$.
	
	In particular, if $p_0 \mu \le 1$ (which implies $p_0\le 1$), then $M$ is independent of $\mu$ and $p_0$. 
\end{lemma}

\begin{proof}
	This is a slight improvement of \cite[Lemma 2.2]{zacher_weak_2013}.
\end{proof}

The following lemma goes back to the work of Bombieri and Giusti \cite{bombieri_harnacks_1972}, it allows us to combine an $\L^\beta-\L^{\beta_0}$ estimate with a logarithmic weak $\L^1$ estimate to obtain an improved $\L^{\beta_0}$-estimate. The proof is based on an exponential decomposition and yet another iteration argument. 

\begin{lemma} \label{lem:bomb}
	Let $0 < \beta_0 \le \infty$, $C_1,C_2>0$, $\delta,\eta \in (0,1)$ and $\gamma>0$. Suppose that a positive measurable function $f \colon U_1 \to \R$ satisfies the following two assumptions
	\begin{enumerate}
		\item \begin{equation*}
			\norm{f}_{\L^{\beta_0}(U_{\sigma})} \le \left( \frac{C_1}{(\sigma'-\sigma)^\gamma \nu(U_1)} \right)^{\frac{1}{\beta}-\frac{1}{\beta_0}} \norm{f}_{\L^\beta(U_{\sigma'})},
		\end{equation*}
		for all $0<\delta \le \sigma' < \sigma \le 1$ and $0< \beta \le \min \{ 1, \eta \beta_0 \}$. 
		\item \begin{equation*}
			\nu(\{ \log f > s \}) \le C_2\nu(U_1) s^{-1} 
		\end{equation*}
		for all $s >0$.
	\end{enumerate}
	Then, 
	\begin{equation*}
		\norm{f}_{\L^{\beta_0}(U_\delta)} \le M \nu(U_1)^{\frac{1}{\beta_0}},
	\end{equation*}
	where $	M = M(\beta_0,C_1,C_2,\delta,\eta,\gamma)>0$. 
\end{lemma}

\begin{proof}
	See \cite{bombieri_harnacks_1972,moser_pointwise_1971}. A proof can also be found in \cite{clement_priori_2004,saloff-coste_aspects_2002}. 	
\end{proof}

\bigskip

\noindent {\bf Acknowledgement.} Lukas Niebel thanks Christian Seis for his support and is funded by the Deutsche Forschungsgemeinschaft (DFG, German Research Foundation) under Germany's Excellence Strategy EXC 2044 --390685587, Mathematics M\"unster: Dynamics--Geometry--Structure.

\bibliographystyle{plain}
\bibliography{kinetic_moser.bib}

\end{document}